\documentclass[]{amsart}
\usepackage{amsfonts,amssymb,amsmath, amsthm}

\theoremstyle{plain}
\newtheorem{assumption}{Assumption}
\newtheorem*{assumption0}{Assumption 0}
\newtheorem{theorem}{Theorem}[section]
\newtheorem{definition}[theorem]{Definition}
\newtheorem{corollary}[theorem]{Corollary}
\newtheorem{lemma}[theorem]{Lemma}
\newtheorem{proposition}[theorem]{Proposition}
\newtheorem{example}[theorem]{Example}

\newtheorem{preremark}[theorem]{Remark}
\newenvironment{remark}{\begin{preremark}\normalfont}{\end{preremark}}

\allowdisplaybreaks

\renewcommand{\l}{\mathcal L}
\renewcommand{\r}{\mathcal R}
\newcommand{\e}{\mathcal E}
\newcommand{\F}{\mathcal F}

\newcommand{\R}{\mathbb R}


\title[Harnack inequalities for non-symmetric forms]{Parabolic Harnack inequality for time-dependent non-symmetric Dirichlet forms}
\author{Janna Lierl \and Laurent Saloff-Coste}{ \thanks{Both authors partially supported by NSF Grant DMS 1004771}}

\begin{document}

\begin{abstract}
In the context of a metric measure Dirichlet space satisfying volume doubling and Poincar\'e inequality, we prove the parabolic Harnack inequality for weak solutions of the heat equation associated with local nonsymmetric bilinear forms. In particular, we show that these weak solutions are locally bounded. 
\end{abstract}

\maketitle
\tableofcontents

\section*{Introduction}
This paper is concerned with the parabolic Harnack inequality in the context of an abstract Dirichlet space satisfying volume doubling and Poincar\'e inequality provided that the metric is induced by the Dirichlet form.

As observed by Moser \cite{Moser64, Moser67, Moser71}, the parabolic Harnack inequality implies that weak
solutions of the heat equation are locally bounded and H\"older continuous.
Further, Nash \cite{Nash58} and later Aronson \cite{Aro67} developed heat kernel estimates and other related results. See also \cite{AS67, PE84, FS86, Str88, NS91}.

On Euclidean space, Aronson and Serrin \cite{AS67} developed the theory of parabolic Harnack 
inequalities for quasi-linear divergence form equations having the proper 
structure. This includes time-dependent linear equations in divergence form 
with uniformly elliptic second order term, 
first and zero order terms with bounded coefficients, that is,
\begin{eqnarray} \label{uel}
\lefteqn{\partial _t u(t,x) = \sum_{i,j}\partial_j (a_{i,j}(t,x) \partial_i u(t,x))}&&
\nonumber \\
&+&\sum_ib_i(t,x) \partial_i u(t,x)+\sum_j \partial_j(d_j(t,x)u(t,x)) +c(t,x)u(t,x),\label{dfue}
\end{eqnarray}
to be interpreted in the weak sense and where  $a_{i,j},b_i,d_j$ and  $c$ 
are bounded measurable functions with, $ \forall\,\xi\in \mathbb R^n$,
$\sum _{i,j}a_{i,j}(t,x)\xi_i\xi_j\ge \epsilon |\xi|^2$, $\epsilon>0$.
In particular, if the lower order coefficients $b_i,d_i$ and $c$ all vanish, then 
the weak solutions satisfy a global scale invariant parabolic Harnack inequality 
even when $a_{i,j}$ is time-dependent and not necessarily symmetric. 
This implies a two-sided  Aronson heat kernel estimate that is global in 
time and space. 

One goal of this paper is to obtain similar results in the 
context of metric measure spaces under natural assumptions on the geometry of the space (volume doubling and Poincar\'e inequality). See the parabolic Harnack inequality of Theorem \ref{thm:local VD+PI = local HI}, and applications to heat kernel estimates in Section \ref{sec:applications}.

For purely second order divergence form operators (with no time dependence) on complete Riemannian manifolds,
Grigor'yan  \cite{Gri91} and Saloff-Coste \cite{SC92} observed that the parabolic Harnack inequality
is equivalent to the volume doubling property and the Poincar\'e 
inequality. This characterization of the parabolic Harnack inequality has been very useful in  
the development of analysis on rough spaces including spaces equipped with
a sub-Riemannian structure \cite{JS87, SC95}, Lipschitz manifolds, Alexandrov spaces \cite{KMS01}, 
polytopal complexes and Gromov-Hausdorff limits of Riemannian manifolds \cite{Sturm06}.

Biroli and Mosco \cite{BM95} and Sturm \cite{SturmIII}
 extended these ideas in symmetric strongly local, regular Dirichlet spaces equipped with a non-degenerate intrinsic distance.
The paper \cite{SturmIII} states a 
parabolic Harnack inequality for local weak solutions of the heat equation 
associated with symmetric, time-dependent, strongly local Dirichlet forms 
that are all uniformly comparable to a fixed symmetric strongly local regular 
Dirichlet form that satisfies the doubling property and the Poincar\'e inequality and defines a metric that induces the original topology of the space.
The parabolic Harnack inequality of \cite{SturmIII} relies on mean value estimates that are studied in \cite{SturmII} in a more general context that applies to
a class of (non-symmetric) Dirichlet forms.

The main results of the present work are $L^p$-mean value estimates, the local boundedness of weak solutions, and the parabolic Harnack inequality in the context of time-dependent non-symmetric forms. 
 It is also essential for applications that this work
treats equations associated with forms that are not necessarily
Dirichlet forms but, instead, introduces a setting that mimic general
operators in divergence forms in Euclidean space. See Section \ref{ssec:nonsymmetric forms}.

The results obtained in this paper involve two types of assumptions. 
The first type concerns the structure of the local bilinear forms $\e_t$ and some quantitative conditions. These are introduced in Sections \ref{ssec:nonsymmetric forms} and  \ref{ssec:assumptions on the forms} 
as Assumptions 0, \ref{as:e_t} and \ref{as:p=0}. The second type of assumptions 
concerns the underlying space, these are introduced in Section \ref{ssec:assumptions on the space} as Assumptions \ref{as:sobolev} and \ref{as:VD+PI}.

The main results (Harnack inequality and H\"older continuity of weak solutions) are stated in Theorem \ref{thm:local VD+PI = local HI}, Corollary \ref{cor:global PHI} and 
Corollary \ref{cor:Hoelder}. Applications to heat kernels are described in Section \ref{sec:applications}. 
The present work is in part motivated by further applications to the study of the heat kernel with Dirichlet boundary 
condition and a boundary Harnack principle for harmonic functions of non-symmetric operators. See \cite{LierlSC1, LierlSC3}.

\section{Framework}
The classical theory of symmetric Dirichlet forms is developed in \cite{FOT94}. For the notion of non-symmetric Dirichlet forms see \cite{MR92}.

Let $X$ be a locally compact separable metrizable space and let $\mu$ be a positive Radon measure on $X$ with full support.  On this space, we will consider bilinear forms that 
generalize  (non-symmetric) local Dirichlet forms.

\subsection{The reference form} \label{ssec:reference form}

Throughout this paper, we fix a symmetric, strongly local, 
regular Dirichlet form $(\e,D(\e))$ on $L^2(X,\mu)$ with energy measure $\Gamma$.  
In particular,
\[ \e(u,v) = \int d\Gamma(u,v), \quad \forall u,v \in D(\e). \]
Note that in \cite{FOT94} the energy measure $\Gamma(u,v)$ is denoted as $\frac{1}{2}\mu^c_{<u,v>}$.
For each $u \in D(\e) \cap L^{\infty}(X,\mu)$, $\Gamma(u,u)$ is the unique Radon measure satisfying
\[ \int f d\Gamma(u,u) = \e(u f, u) - \frac{1}{2} \e(f,u^2), \quad \forall f \in D(\e) \cap L^{\infty}(X,\mu). \] 
Further, $\Gamma$ satisfies a sort of  Cauchy-Schwarz inequality (cf.~\cite[Lemma 5.6.1]{FOT94})
\begin{align} \label{eq:CS}
\left| \int fg \, d\Gamma(u,v) \right|
\leq & \left(\int f^2 d\Gamma(u,u) \right)^{\frac{1}{2}} 
      \left(\int g^2 d\Gamma(v,v) \right)^{\frac{1}{2}},
\end{align}
for any $u,v \in D(\e)$ and any bounded Borel measurable functions $f:X \to (-\infty,+\infty)$, $g:X \to (-\infty,+\infty)$.

The energy measure $\Gamma$ satisfies a chain rule: For any $v, u_1, u_2, \ldots, u_m \in D(\e) \cap L^{\infty}(X,\mu)$, $u = (u_1, \ldots, u_m)$, and $\Phi \in \mathcal C^1(\R^m)$ with $\Phi(0)=0$, we have $\Phi(u) \in D(\e) \cap L^{\infty}(X,\mu)$ and
\begin{align} \label{eq:chain rule for Gamma}
d\Gamma(\Phi(u),v) = \sum_{i=1}^{m} \Phi_{x_i}(\tilde u) d\Gamma(u_i, v),
\end{align}
where $\Phi_{x_i}:=\partial \Phi / \partial x_i$ and $\tilde u$ is a quasi-continuous version of $u$, see \cite[(3.2.27) and Theorem 3.2.2]{FOT94}. 
When $\Phi_{x_i}$ is bounded for all $i \in \{1,\ldots,m\}$ in addition, then $\Phi(u) \in D(\e)$ and \eqref{eq:chain rule for Gamma} holds for any $u_1, \ldots, u_m \in D(\e)$ and any $v \in D(\e) \cap L^{\infty}(X,\mu)$; see \cite[(3.2.28)]{FOT94}.

In some cases we will tacitly apply \eqref{eq:CS} or \eqref{eq:chain rule for Gamma} with possibly unbounded functions. Whenever we do this, it will be easy to see that a simple approximation of the type $u_m := u \wedge m$ will justify our reasoning. For any $0 \le u \in \F$, $\e(u - u \wedge m, u - u \wedge m) \to 0$ as $m \to \infty$ by \cite[Theorem 1.2.4]{FOT94}. 

By \cite[Theorem 1.4.2]{FOT94}, $\mathcal D(\e) \cap L^\infty(X,\mu)$ is an algebra. Hence, inequality \eqref{eq:CS} together with a Leibniz rule (\cite[Lemma 3.2.5]{FOT94}) implies that
\begin{align} \label{eq:Gamma(fg)}
 \int d\Gamma(fg,fg)
& \leq  2\int f^2 d\Gamma(g,g)  
      + 2\int g^2 d\Gamma(f,f),
\end{align}
for any $f,g \in D(\e) \cap L^{\infty}(X,\mu)$. Here, on the right hand side, quasi-continuous versions of $f$ and $g$ must be used.

Because the domain of $\e$ plays a fundamental role, we set
\[ \F:=D(\e) \quad \mbox{ and } \quad  \|f\|_\F := \left(
\int |f|^2 d\mu+ \int d\Gamma(f,f)\right)^{1/2}. \]
In our context, the space $\F$ plays the role of the first order $L^2$ Sobolev space. By definition, the (essential) support of $f\in L^2(X,\mu)$ is the support of the measure $|f| d\mu$.
For an open set $U \subset X$, we set
\begin{align*}
& \F_{\mbox{\tiny{c}}}(U) := \{ f \in D(\e) : \textrm{ The support of } f \textrm{ is compact in } U \}, \\
& \F_{\mbox{\tiny{{loc}}}}(U)  :=  \{ f \in L^2_{\mbox{\tiny{loc}}}(U) : \forall \textrm{ compact } K \subset U, \ \exists f^{\sharp} \in D(\e), f\big|_K = f^{\sharp}\big|_K \mbox{ $\mu$-a.e.} \}.
\end{align*} 
For functions in $\F_{\mbox{\tiny{loc}}}(U)$ we always take their quasi-continuous versions.
Note that $\Gamma(f,g)$ can be defined locally on $U$ for $f,g \in \F_{\mbox{\tiny{loc}}}(U)$ by virtue of \cite[Corollary 3.2.1]{FOT94}. For any $v, u_1, \ldots, u_m \in \F_{\mbox{\tiny{loc}}}(U) \cap L^{\infty}_{\mbox{\tiny{loc}}}(U,\mu)$ and $\Phi \in \mathcal C^1(\R^m)$, we have $\Phi(u) \in \F_{\mbox{\tiny{loc}}}(U) \cap L^{\infty}_{\mbox{\tiny{loc}}}(U,\mu)$ and the chain rule \eqref{eq:chain rule for Gamma} holds. 
For convenience, we set $\F_{\mbox{\tiny{c}}} := 
\F_{\mbox{\tiny{c}}}(X)$ and $\F_{\mbox{\tiny{loc}}} := 
\F_{\mbox{\tiny{loc}}}(X)$.
We will use this notation throughout. One fundamental assumption for the results of this paper is that all other bilinear forms on $L^2(X,\mu)$ that we will consider will share with $\e$ the same domain $\F$.

We will make frequent use of the following properties of strongly local forms. We refer to these properties simply as \emph{strong locality}.
\begin{proposition} \label{prop:strongly local}
\begin{enumerate} 
\item If $u,v\in \F$ and there exists $c\in \mathbb R$ such that 
$(u-c)v=0$ almost everywhere, then $\e(u,v)=0$.
\item
If $u,v\in \F$ and there exists $c\in \mathbb R$ such that 
$(u-c)v=0$ almost everywhere, then $\int \psi^2 d\Gamma(u,v)=0$ for all $\psi \in \F_{\mbox{\em \tiny{c}}} \cap L^{\infty}(X,\mu)$.
\end{enumerate}
\end{proposition}

\begin{proof} (i) is proved in \cite[Proofs of Theorem 2.4.2 and Theorem 2.4.3]{CF12}. 
(ii) can be derived from (i).
\end{proof}

\subsection{Some preliminary computations for symmetric strongly local forms}

For a non-negative function $u \in \F_{\mbox{\tiny{loc}}}$ and a positive integer $n$ let 
\[ u_n := u \wedge n. \]

We will be using indices to functions in various ways. To avoid confusion, subcripts $n$ or $m$ will make the function bounded, subcripts $k$ or $l$ or running indices to denote
a sequence, subscript $h$ will denote a Steklov average, and subscript $\varepsilon$ will make the function uniformly positive. We will reintroduce these notations in the appropriate
places.

For a function $\psi \in \F$ we write $d\Gamma(\psi,\psi) \leq c \, d\mu$ if the energy measure $d\Gamma(\psi,\psi)$ is absolutely continuous with respect to $\mu$ and has a Radon-Nikodym derivative that is bounded above by $c \in (0,\infty)$. 
If $\int \widetilde{u}^2 d\Gamma(\psi,\psi) < \infty$ for some function $\psi \in \F_{\mbox{\tiny{c}}}$, then we write $u \in L^2(X,d\Gamma(\psi,\psi))$. In particular, this is the case whenever $\psi \in \F_{\mbox{\tiny{c}}}$ and $d\Gamma(\psi,\psi) \leq c \, d\mu$.

\begin{proposition} \label{prop:u_mg converges in F}
Let $0 \leq u \in \F$ and let $(f_k) \subset \F \cap L^{\infty}$ be a sequence that converges to $u$ in $\F$ as $k \to \infty$. Let $g \in \F \cap L^{\infty}$.
Suppose that $u \in L^2(X,d\Gamma(g,g))$ and that either $\widetilde{f_k} \leq 2\widetilde{u}$ q.e.~for all $k$ or $\widetilde{f_k} \leq m$ q.e.~for some integer $m$ and for all $k$. Then $ug \in \F$, and there exists a subsequence $(f_{k_l})$ such 
that $(f_{k_l} g)$ converges to $u g$ in $\F$ as $l \to \infty$.
\end{proposition}

\begin{proof}
First, note that $f_k g \in \F \cap L^{\infty}$.
Since $(f_k)$ is a Cauchy sequence in $\F$, there exists by \cite[Theorem 2.1.4]{FOT94} a subsequence $(\widetilde {f_{k_l}})$ of the continuous 
modifications $(\widetilde{f_k})$ which converges to $\widetilde{u}$ quasi-everywhere as $l \to \infty$. 
By \eqref{eq:Gamma(fg)},
\begin{align*}
& \quad \int d\Gamma((f_{k_{l+1}} -f_{k_l})g, (f_{k_{l+1}} - f_{k_l})g) \\
& \leq 2 ||g||_{\infty} \int d\Gamma(f_{k_{l+1}} - f_{k_l}, f_{k_{l+1}} - f_{k_l}) + 2\int (\widetilde{f_{k_{l+1}}} - \widetilde{f_{k_l}})^2 d\Gamma(g,g).
\end{align*}
Letting $l \to \infty$, the first term on the right hand side tends to $0$, because $g$ is bounded and $(f_{k_l})$ is a Cauchy sequence in $\F$. 
The second integral on the right hand side tends to $0$ by the dominated convergence theorem and by the fact that  $d\Gamma(g,g)$ charges no set of zero capacity 
by \cite[Lemma 3.2.4]{FOT94}. This shows that $(f_{k_l} g)$ is a Cauchy sequence in $\F$, 
and $(\widetilde{f_{k_l} g}) \to \widetilde{ug}$ quasi-everywhere as $l\to \infty$. Thus, the assertion follows from \cite[Theorem 2.1.4]{FOT94}.
\end{proof}

\begin{lemma} \label{lem:SUP sym2}
Let $p \in \R$,  $\psi \in \F \cap \mathcal C_{\mbox{\em{\tiny{c}}}}(X)$, and $0 \leq u \in \F_{\mbox{\em{\tiny{loc}}}}$ with $u \in L^2(X,d\Gamma(\psi,\psi))$.
Assume one of the following hypotheses.
\begin{enumerate}
\item
$p\geq 2$,
\item
$u$ is locally uniformly positive.
\end{enumerate} 
Then $u u_n^{p-2} \in \F_{\mbox{\em{\tiny{loc}}}}$, $u u_n^{p-2} \psi^2 \in \F_{\mbox{\em{\tiny{c}}}}$, and for any $k >0$ we have
\begin{align} \label{eq:SUP sym}
(1-p) \e(u, u u_n^{p-2} \psi^2) 
& \leq  4k \int u^2 u_n^{p-2} d\Gamma(\psi,\psi)  \nonumber \\
      & \quad + \left( \frac{|1-p|^2}{k} + (1-p) \right) \int \psi^2  u_n^{p-2}  d\Gamma(u,u) \nonumber \\       
      & \quad - ( (1-p)^2 + (1-p)) \int \psi^2  u_n^{p-2} d\Gamma(u_n,u_n).
\end{align}
\end{lemma}

\begin{proof}
First consider the case when $u$ is locally bounded. Then \eqref{eq:chain rule for Gamma}, \eqref{eq:Gamma(fg)} and the strong locality of the reference form easily yield that
 $u u_n^{p-2} \in \F_{\mbox{\tiny{loc}}}$ and $u u_n^{p-2} \psi^2 \in \F_{\mbox{\tiny{c}}}$. 
In order to prove the estimate \eqref{eq:SUP sym}, write
\begin{align*}
(1-p) \e(u, u u_n^{p-2} \psi^2)
& =  2(1-p) \int \psi u u_n^{p-2} d\Gamma(u,\psi)
 + (1-p) \int \psi^2 d\Gamma(u,u u_n^{p-2}).
\end{align*}
The first integral on the right hand side can be estimated using the Cauchy-Schwarz inequality \eqref{eq:CS}. We will estimate the second integral on the right hand side.
We have
\[ \int (n^{p-2} - u_n^{p-2} ) \psi^2 d\Gamma(u,u-u_n) = 0 \]
by \eqref{eq:chain rule for Gamma} and Proposition \ref{prop:local}(iii), and we have
 $\int u_n^{p-2} \psi^2 d\Gamma(u-u_n,u_n) = 0$ by Proposition \ref{prop:local}(iii).
Hence,
\begin{align*}
& \int n^{p-2} \psi^2 d\Gamma(u,u-u_n) \\
& = \int u_n^{p-2} \psi^2 d\Gamma(u,u-u_n) \\
& = \int u_n^{p-2} \psi^2 d\Gamma(u,u) 
- \int u_n^{p-2} \psi^2 d\Gamma(u-u_n, u_n) - \int u_n^{p-2} \psi^2 d\Gamma(u_n,u_n) \\
& = \int u_n^{p-2} \psi^2 d\Gamma(u,u) 
- \int u_n^{p-2} \psi^2 d\Gamma(u_n,u_n).
\end{align*}
Therefore, by \eqref{eq:chain rule for Gamma}, and the strong locality of the energy measure $d\Gamma$, we get
\begin{align*}
& (1-p) \int \psi^2 d\Gamma(u,u u_n^{p-2}) \\
& =  (1-p) \int \psi^2 d\Gamma(u,(u-u_n)n^{p-2}) + (1-p) \int \psi^2 d\Gamma(u_n,u_n^{p-1}) \\
& =  (1-p) \int n^{p-2} \psi^2 d\Gamma(u,u-u_n) - (1-p)^2 \int u_n^{p-2} \psi^2 d\Gamma(u_n,u_n) \\
& =  (1-p) \int u_n^{p-2} \psi^2 d\Gamma(u,u) 
 - ( (1-p)^2 + 1-p) \int u_n^{p-2} \psi^2 d\Gamma(u_n,u_n).
\end{align*}
This proves the assertion when $u$ is locally bounded. 

Now we consider the case when $u$ is unbounded. Since $u \in \F_{\mbox{\tiny{loc}}}$, for any compact set $K \subset X$ there exists $v \in \F$ such that $u=v$ on $K$ $\mu$-almost everywhere. Replacing $v$ by $v \vee 0$, we may assume that $v$ is non-negative. Since $v$ may be unbounded, let us approximate $v$ by $v_m := v \wedge m$. Since $\F$ is the domain of a Dirichlet form, it is clear that $v_m \in \F \cap L^{\infty}$. We also know that $v_m v_n^{p-2} \in \F$ and $v_m v_n^{p-2} \psi^2 \in \F_{\mbox{\tiny{c}}}$ because $\F \cap L^{\infty}$ is an algebra. As $m \to \infty$, $v_m$ converges to $v$ in $\F$ by \cite[Theorem 1.2.4]{FOT94}, and also pointwise. It then follows from \eqref{eq:Gamma(fg)} that
\begin{align*}
 \e((v_m - v_k) v_n^{p-2},(v_m - v_k) v_n^{p-2}) 
 & \le  2 \int (v_m - v_k)^2 d\Gamma(v_n^{p-2},v_n^{p-2})  \\
 & \quad +  2 \int v_n^{2p-4} d\Gamma(v_m-v_k,v_m-v_k) \\
 & \longrightarrow 0 \quad \mbox{ as } k,m \to \infty.
\end{align*}
This shows that $(v_m v_n^{p-2})_m$ is a Cauchy sequence in $\F$. By the closedness of $\F$ and the fact that $v_m v_n^{p-2} \to v v_n^{p-2}$ in $L^2$, it follows that $v v_n^{p-2} \in \F$. In particular, $u u_n^{p-2} \in \F_{\mbox{\tiny{loc}}}$. By similar reasoning, $u u_n^{p-2} \psi^2 \in \F_{\mbox{\tiny{c}}}$.

To see that \eqref{eq:SUP sym} holds also when $u$ is unbounded, we approximate $u$ by $u_m$ and apply \eqref{eq:SUP sym} to $u_m$. 
Because of the 
assumption $u \in L^2(X,d\Gamma(\psi,\psi))$, the dominated convergence theorem, the Markov property, and the strong locality of $d\Gamma$, we can take the limit as $m \to \infty$. This completes the proof of \eqref{eq:SUP sym}.
\end{proof}

\section{Non-symmetric perturbations}

\subsection{Local bilinear forms} \label{ssec:nonsymmetric forms}
In this section, we explain structural properties of (possibly non-symmetric) bilinear forms on $L^2(X,\mu)$, which may or may not be satisfied for a given bilinear form. 
In Assumption 0 below we compile those properties that we will assume in the rest of the paper. Additional hypotheses will be made
in Section \ref{ssec:assumptions on the forms}, namely Assumption \ref{as:e_t} and \ref{as:p=0}. Though the latter assumptions are of a quantitative nature, they also have
 implications on the structure of the forms.

Let $(\e_*,D(\e_*))$ be a (possibly non-symmetric)
bilinear form on $L^2(X,\mu)$. We will be concerned only with forms that are \emph{local}, that is,
$\e_* (f,g)=0$ for any pair $f,g\in D(\e_*)$ with compact 
disjoint supports.  The form  $(\e_*,D(\e_*))$ 
is called \emph{strongly local} if $\e_* (f,g)=0$ for any pair $f,g\in D(\e_*)$ 
with compact supports with $f$ constant on a neighborhood of the support of $g$
or vice versa.   We say that $1$ is locally in the domain $D(\e_*)$ if
for any compact set $K\subset X$ there is a function 
$f_K\in D(\e_*)$ with compact support and such that $ f_K=1$ in a neighborhood of $K$. 
If that is the case and $\e_*$ is local then 
$\e_* (u, 1)$ and $\e_* (1,u)$ are well defined for any function $u\in D(\e_*)$
with compact support. Indeed, assuming that the support of $u$ is $K$, set
$\e_*(u,1)=\e_*(u,f_K)$ and note that the result is independent of the choice of the function $f_K\in D(\e_*)$ which has compact support and equals $1$ 
on a neighborhood of $K$.

Let
\[\e_*^{\mbox{\tiny{sym}}}(f,g) = \frac{1}{2} \big( \e_*(f,g) + 
\e_*(g,f) \big) \]
be the symmetric part of $\e_*$ and
\[ \e_*^{\mbox{\tiny{skew}}}(f,g) = 
\frac{1}{2} \big( \e_*(f,g) - \e_*(g,f) \big) \]
the skew-symmetric part.

\begin{example}
On $X=\mathbb R$, for any choice of $k_1,k_2\in \mathbb N$,
the form $(f,g)\mapsto \e_*(f,g)= \int  f^{(k_1)} g^{(k_2)} dx$,
where $f^{(k)}$ denotes the $k$-th derivative of $f$ and 
$f,g\in \mathcal C^\infty_{\mbox{\emph{\tiny c}}}(\mathbb R)$, is local. It is strongly local
if and only if $(k_1,k_2) \neq (0,0)$. However, if $|k_1-k_2|$ is odd, the symmetric part of the form is degenerate.
\end{example} 

\begin{definition} \label{def:l and r}
Assume that $(\e_*, D(\e_*))$ is local
and that $1$ is locally in $D(\e_*)$.
Define the bilinear forms  $\l_*$ and $\r_*$ by
\begin{align*}
\l_*(u,v) & = \frac{1}{4} \big[ \e_*(uv,1) - \e_*(1,uv) + 
\e_*(u,v) - \e_*(v,u) \big],\\
\r_*(u,v) & = \frac{1}{4} \big[ \e_*(1,uv) - \e_*(uv,1) + \e_*(u,v) - \e_*(v,u) \big]
= - \l(v,u),
\end{align*}
for any $u,v \in D(\e_*)$ with $uv$ having compact support and 
$uv \in D(\e_*)$.
\end{definition}
\begin{remark} \begin{enumerate} \item
Without further assumption, it is not clear 
that there are many $u,v\in D(\e_*)$ such that $uv\in D(\e_*)$. We will assume below that $D(\e_*)$ is the domain $\F$ of the regular reference form.
\item  The locality of  $\e_*$ implies that
 $\l_*$ is \emph{left-strongly local}, i.e.~$\l_*(u,v) = 0$ if 
$u,v$ have compact support and  
$u$ is constant on a neighborhood of the support of $v$. Moreover,
for any  $u,v \in D(\e_*)$ with $uv$ of 
compact support and  
$ uv \in D(\e_*)$,
\begin{equation} 
 \e_*^{\mbox{\tiny{skew}}}
(u,v) = \l_*(u,v) + \r_*(u,v). 
\end{equation}
\end{enumerate}
\end{remark}

\begin{example} \label{exa-high} In the case  
$(f,g)\mapsto \e_*(f,g)= \int ( f^{(k_1)} g^{(k_2)} +fg) dx$, 
$f,g\in \mathcal C^\infty_{\mbox{\emph{\tiny c}}}(\mathbb R)$, we have
\begin{eqnarray*}
\e_*^{\mbox{\em\tiny sym}}(f,g) &=&
\int \left(\frac{1}{2}(f^{(k_1)} g^{(k_2)} + f^{(k_2)} g^{(k_1)}) 
+fg\right) dx,\\ 
\e_*^{\mbox{\em\tiny skew}}(f,g) &=&
\int\frac{1}{2}(f^{(k_1)} g^{(k_2)} - f^{(k_2)} g^{(k_1)}) dx.
\end{eqnarray*}
If $k_1=1, k_2=0$, then 
$\l_*(f,g)= \frac{1}{2}\int f'g \,dx$.
If $k_1=2, k_2=0$, then 
$\l_*(f,g)= 0$ by integration by parts. If $k_1=2, k_2=1$, then  $\l_*(f,g) = \frac{1}{2} \int f'' g' - f' g'' dx$. Anticipating on the definition
of a ``chain rule skew form'' given below, 
note that the skew-symmetric part of $\e_*$ satisfies a  
chain rule in the case $k_1=1,k_2=0$
but not in the case $k_1=2, k_2=1$. 
\end{example}

We now make an important extra hypothesis. We consider a bilinear form $\e_*$ whose domain $D(\e_*)$ is equal to the  domain $\F$ of our reference form $(\e,\F)$.
Since the reference form is a symmetric strongly local regular Dirichlet form, $\F$ has many good properties including the fact that $1\in \F_{\mbox{\tiny loc}}$ and that $\F_{\mbox{\tiny c}} \cap L^\infty (X,\mu)$ is an algebra and is dense in the Hilbert space $(\F,\|\cdot\|_\F)$. 
We will use freely the notation $\e^{\mbox{\tiny sym}}_*, \e^{\mbox{\tiny skew}}_*,\l_*$ and
$\r_*$. Note that, by locality of $\e_*$, $\l(u,v)$ and $\r(u,v)$ are well-defined for any $u \in \F_{\mbox{\tiny{loc}}} \cap L^{\infty}_{\mbox{\tiny{loc}}}(X,\mu)$ and $v \in \F_{\mbox{\tiny{c}}} \cap L^{\infty}(X,\mu)$.

Let $\mathcal D$ be a linear subspace of $\F \cap \mathcal C_{\mbox{\tiny{c}}}(X)$ such that
\begin{enumerate}
\item
$\mathcal D$ is dense in $(\F,\| \cdot \|_{\F})$.
\item
If $f \in \mathcal D$ then $(f \vee 0) \in \mathcal D$ and $(f \wedge m) \in \mathcal D$ for any positive integer $m$.
\item
If $f \in \mathcal D$ then $\Phi(f) \in \mathcal D$ for any function $\Phi \in \mathcal{C}^1(\R^m)$ with $\Phi(0)=0$, where $m$ is a positive integer.
\end{enumerate}
Readers who do not appreciate this generality may simply take $\mathcal D$ to be $\F \cap \mathcal C_{\mbox{\tiny{c}}}(X)$. An application that requires $\mathcal D$ to be a proper subspace of $\F \cap \mathcal C_{\mbox{\tiny{c}}}(X)$ can be found in \cite[Lemma 4.1]{LierlHKEf}.

\begin{definition} \label{def:leibniz rule}
Assuming $\e_*$ is local with $D(\e_*)=\F$, 
we say that  $\e^{\mbox{\em\tiny{skew}}}_*$ is a 
{\em chain rule skew form relative to $\mathcal D$} 
if the following two properties hold:
\begin{enumerate} 
\item
(Leibniz rule) For any $u,v,f \in \mathcal D$, we have
 \[ \l_*(uf,v) = \l_*(u,fv) + \l_*(f,uv). \]
\item (Chain rule) Let $v, u_1, u_2, \ldots, u_m \in \mathcal D$ and $u = (u_1, \ldots, u_m)$. For any $\Phi \in \mathcal C^2(\R^m)$,
\begin{align*}
\l_*(\Phi(u),v) = \sum_{i=1}^{m} \l_*(u_i, \Phi_{x_i}(u) v).
\end{align*}
\end{enumerate}
\end{definition}

\begin{remark}
In Example \ref{ex:Euclidean} below we give an example of a chain rule skew form on Euclidean Space $X = \R^n$. Note that, in this example, $\l_*$ includes terms of first order 
and terms of second order. In general, we are not able isolate the first order terms or the second order terms, and we therefore avoid the notion of ``order''. Instead, it makes sense
to speak of a ``local non-symmetric form with symmetric strongly local part'' or of a ``strongly local non-symmetric form''.
\end{remark}

\begin{remark}
When $\e_*$ is a non-symmetric strongly local regular Dirichlet form, its skew-symmetric part is a chain rule skew form with respect to the domain of the form itself, 
see \cite[Theorems 3.2 and 3.8]{HMS10}. In this case, $\l_*(u,fv)$ is the same as $\frac{1}{2} \langle L(u,v),f \rangle$ in the notation of \cite{HMS10}. 
Though \cite{HMS10} is written in greater generality, it seems to us that they implicitly assume that $\langle L(u,v),f \rangle$ is skew-symmetric, 
in which case $\l_* = \e_*^{\mbox{\tiny{skew}}} = \r_*$.
\end{remark}

The following assumption on the structure of the bilinear form $\e_*$ refers to the domain $\F$ of the reference form $(\e,\F)$ defined in Section \ref{ssec:reference form}, and to the subset $\mathcal D \subset \F \cap \mathcal{C}_{\mbox{\tiny{c}}}(X)$ introduced before Definition \ref{def:leibniz rule}.
\begin{assumption0}
The form $(\e_*,D(\e_*))$ is local,
its domain $D(\e_*)$ is $\F$ and:
\begin{enumerate}
\item The form $\e_*$ satisfies 
$$\forall\,f,g\in \F, \quad |\e_*(f,g)|\le C_*\|f\|_\F\|g\|_\F$$
for some constant $C_* \in (0,\infty)$.
\item For all $f,g\in \F \cap L^{\infty}(X,\mu)$ with $fg \in \F_{\mbox{\em \tiny{c}}}(X)$,
\[ |\e_*^{\mbox{\em\tiny sym}}(fg,1)|\le C_*\|f\|_\F\|g\|_\F, \]
for some constant $C_* \in (0,\infty)$.
\item
For all $f \in \F \cap \mathcal C_{\mbox{\em \tiny{c}}}(X)$, 
\[ C^{-1} \, \e(f,f) \leq \e_*^{\mbox{\em \tiny{s}}}(f,f) \leq C \, \e(f,f), \]
for some constant $C \in (0,\infty)$, where $\e_*^{\mbox{\em \tiny{s}}}(f,f) := \e_*^{\mbox{\em \tiny{sym}}}(f,f) - \e_*^{\mbox{\em \tiny{sym}}}(f^2,1)$.
\item The skew-symmetric part $\e^{\mbox{\em\tiny{skew}}}_*$ is a chain rule skew form relative to $\mathcal D$. 
 \end{enumerate}
\end{assumption0}

\begin{remark} \label{rem:assumption 0}
\begin{enumerate}
\item  
Under Assumption 0(i), the form $\e_*$ as well as its symmetric part $\e^{\mbox{\tiny sym}}_*$, and its skew-symmetric part $\e^{\mbox{\tiny skew}}_*$
are continuous on $\F\times \F$. 
\item
Under Assumption 0(iii), $f\mapsto \|f\|_\F$ and $f\mapsto 
(\e^{\mbox{\tiny{s}}}_*(f,f)+\int|f|^2d\mu)^{1/2}$ are two equivalent norms on $\F$. 
\item
Under Assumption 0(i)-(iii), the bilinear form
\[ \e^{\mbox{\tiny s}}_*
(f,g):=\e^{\mbox{\tiny sym}}_*(f,g)-
\e^{\mbox{\tiny sym}}_*(fg,1), \]
defined for $f,g\in \F \cap L^{\infty}(X,\mu)$ with $fg\in \F_{\mbox{\tiny c}}$, extends continuously to $\F \times \F$. The extension $(\e^{\mbox{\tiny s}}_*,\F)$ is a regular strongly local 
symmetric Dirichlet form with domain $\F$. The proof of this elementary fact will be contained in a forthcoming paper.
\item
Assumption 0(iii) holds if and only if there exists a constant $C \in (0,\infty)$ such that the energy measure $\Gamma_*$ of $\e^{\mbox{\tiny s}}_*$ satisfies
\begin{align} \label{eq:Gamma_* comparable} 
C^{-1} \int f^2 d\Gamma(g,g) \leq \int f^2 d\Gamma_*(g,g) \leq C \int f^2 d\Gamma(g,g),
\end{align}
for all $f:X \to (-\infty,+\infty)$ bounded Borel measurable and all $g \in \F \cap \mathcal C_{\mbox{\tiny{c}}}(X)$. In this case, \eqref{eq:Gamma_* comparable} extends to all functions $g \in \F$.
\item
Under an additional assumption, the maps $(f,g) \mapsto \e_*(fg,1)$ and $(f,g) \mapsto \e_*(1,fg)$, extend continuously to $\F \times \F$. 
Also the bilinear forms $\l_*$ and $\r_*$ can be extended continuously to $\F \times \F$. See Proposition \ref{prop:extending L}. For the purpose of this paper, we will not need these extensions.
\item
We use the bilinear forms $\l_*$ and $\r_*$ only in Proposition \ref{prop:skew identity with p}, in Proposition \ref{prop:|u| is subsolution}, and in Proposition \ref{prop:extending L} below.
 \end{enumerate} 
\end{remark}

\begin{remark}
\begin{enumerate}
\item
Suppose $(\e,\F)$ is a local positivity preserving closed form which satisfies Assumption 0(i)-(iii). Then it also satisfies Assumption 0(iv). 
 \item Suppose that, in addition to Assumption 0, it holds that
$$|\e^{\mbox{\tiny sym}}_*(f^2,1)|\le C\|f\|_2\|f\|_\F,$$
for any $f \in \F \cap \mathcal C_{\mbox{\tiny{c}}}(X)$, where $\| f \|_2 =  \left( \int |f|^2 d\mu \right)^{1/2}$. Then there exists 
$\lambda\in \mathbb R$ such that $\e_*+\lambda \langle\cdot,\cdot\rangle_\mu$
is a coercive closed form.  In addition, this form is positivity preserving, as can easily be seen by applying Proposition \ref{prop:local}. See also \cite{MaRock}. In fact, the form $(\e_*,\F)$ itself is closed and 
positivity preserving.
\end{enumerate} 
\end{remark}

We will make frequent use of the following properties of local forms. 
\begin{proposition} \label{prop:local}
Under {\em Assumption 0}, the following holds:
\begin{enumerate} 
\item If $u,v\in \F$ are such that $uv=0$ $\mu$--a.e.\ then $\e_*(u,v)=0$. 
\item If $u,v\in \F$ are such that there exists $c\in \mathbb R$ such that 
$(u-c)v=0$ almost everywhere, then $\e^{\mbox{\em\tiny s}}_*(u,v)=0$.
\item
If $u,v\in \F$ are such that there exists $c\in \mathbb R$ such that 
$(u-c)v=0$ almost everywhere, then $\int \psi^2 d\Gamma_*(u,v)=0$ for all $\psi \in \F_{\mbox{\em \tiny{c}}} \cap L^{\infty}(X,\mu)$.
\item
If $u \in \F$, $v \in \F_{\mbox{\em \tiny{c}}}$ are such that $uv \in \F$ and there exists $c\in \mathbb R$ such that 
$(u-c)v=0$ almost everywhere, then $\l_*(u,v)=0$. We refer to this property as the \emph{strong left-locality} of $\l_*$. Similarly, $\r_*$ is strongly right-local.
\end{enumerate}
\end{proposition}

\begin{proof} (ii) and (iii) holds because $(\e_*^{\mbox{\tiny{s}}},\F)$ is in fact a strongly local regular Dirichlet form, and by Proposition \ref{prop:local}. See Remark \ref{rem:assumption 0}(iii).

For the proof of (i), it suffices to repeat the reasoning in \cite[Proof of Theorem 2.4.2]{CF12} and apply Assumption 0(i). 
The statement (iv) follows from the definition of $\l_*$, the bilinearity of $\e_*$, and (i). 
\end{proof}

\begin{example} \label{exa-eucl}
On $X=\mathbb R$ equipped with Lebesgue measure $dx$, let $a,b,c$ be bounded measurable functions and
consider 
$$\e_*(f,g)= \int f'g'dx+ \int a(f' g-g'f)dx +\int b(f'g+fg')dx +\int cfgdx $$ 
with domain the first Sobolev space $\F=W^{1,2}(\mathbb R)$. 
Then $\e_*$ obviously satisfies Assumption 0.
Assume that the distributions $a',b'$ are signed Radon measures 
(obviously, this is not always the case!). 
The form $\e_*$ is not a Dirichlet form in general. Indeed,
for  $\e_*$ to be a Dirichlet form it is necessary that
$$c \ge b'-a' \mbox{ and } c\ge b'+a'.$$ 
If $\gamma$ is a non-negative Radon measure such that $c+\gamma\ge b'-a'$ 
and $c+\gamma\ge b'+a'$ then $(f,g)\mapsto \e_* (f,g)+\gamma(fg)$ is a 
Dirichlet form on $L^2(\R,dx)$ but its domain $\F\cap L^2(\mathbb R,\gamma)$
will, in general, be smaller than the first Sobolev space $\F$. 
\end{example}

\begin{example} \label{ex:Euclidean}
On Euclidean space $X=\R^n$, consider the form
\begin{align*}
 \e_*(f,g) = & \int  \left(
\sum_{i,j=1}^n  a_{i,j} \partial_i f \partial_j g  
           + \sum_{i=1}^n b_i \partial_i f \, g  
           + \sum_{i=1}^n f \, d_i \partial_i g  
           + c f g \right)\, dx,
\end{align*}
with coefficients $a=(a_{i,j})$, $b=(b_i)$, $d=(d_i)$, $c$ satisfying 
\begin{enumerate}
\item $\sum_{i,j=1}^n |a_{i,j} - a_{j,i}| \leq M$ for some $M>0$,
\item there are positive constants $k_0, K_0$ such that
$k_0 |\xi|^2 \leq \sum_{i,j} a_{i,j} \xi_i \xi_j \leq K_0 |\xi|^2$ for all $\xi \in \R^n$.
\end{enumerate}

Set $\tilde a_{i,j} := (a_{i,j} + a_{j,i})/2$ and $\check a_{i,j} = (a_{i,j} - a_{j,i})/2$. Then the symmetric part of $\e$ is 
\begin{align*}
 \e^{\mbox{\emph{\tiny{sym}}}}_*(f,g) 
 = & \int \sum_{i,j=1}^n \tilde a_{i,j} \partial_i f \partial_j g \, dx
           + \int \sum_{i=1}^n \frac{b_i+d_i}{2} \partial_i f \, g \, dx \\
   &        + \int \sum_{i=1}^n f \, \frac{b_i+d_i}{2} \partial_i g \, dx
       + \int c f g \, dx,
\end{align*}
while the skew-symmetric part of $\e$ is 
\begin{align*}
 \e^{\mbox{\emph{\tiny{skew}}}}_*(f,g) 
 = & \int \sum_{i,j=1}^n \check a_{i,j} \partial_i f \partial_j g \, dx
           + \int \sum_{i=1}^n \frac{b_i-d_i}{2} \partial_i f \, g \, dx \\
   &        + \int \sum_{i=1}^n f \, \frac{-b_i + d_i}{2} \partial_i g \, dx.
\end{align*}
The symmetric part can be written as $\e^{\mbox{\emph{\tiny{sym}}}}_*(f,g) 
= \e^{\mbox{\emph{\tiny{s}}}}_*(f,g) + 
\e^{\mbox{\emph{\tiny{sym}}}}_*(fg,1)$, 
where $\e^{\mbox{\emph{\tiny{s}}}}_*$ is the symmetric strongly local part
\begin{align*}
 \e^{\mbox{\emph{\tiny{s}}}}_*(f,g) & = \int \sum_{i,j=1}^n \tilde a_{i,j} \partial_i f \partial_j g \, dx.
\end{align*}
The skew-symmetric part can be written as $\e^{\mbox{\emph{\tiny{skew}}}}_*(f,g) = \l(f,g) + \r(f,g)$ with
\[ \l_*(f,g) = \int \sum_{i,j=1}^n \frac{\check a_{i,j}}{2} \partial_i f \, \partial_j g \, d\mu 
                 + \int \sum_{i=1}^n \frac{b_i - d_i}{2} \partial_i f \, g \, d\mu. \]
In the context of this paper, the coefficients $(a_{i,j}),(b_i), (d_i), c$
can be allowed to be functions of the time-space variable $(t,x)$ so that 
the form $\e_*$ above would also depend on $t$. In this example, if all coefficients are bounded measurable, the underlying domain $\F$ is the 
first Sobolev space and Assumption 0 is satisfied (with respect to that space).
\end{example}

\begin{remark}
Condition (i) and (ii) in the above example is equivalent to
\begin{enumerate}
\item[(ii')]  There are positive constants $k_0, K_0$ such that 
\[| \sum_{i,j} a_{i,j} \xi_i \zeta_j |\leq K_0 |\xi| |\zeta| \quad  \textrm{ for all } \xi,\zeta \in \R^n \]
and
\[k_0 |\xi|^2 \leq \sum_{i,j} a_{i,j} \xi_i \xi_j \quad  \textrm{ for all } \xi \in \R^n. \]
\end{enumerate}
\end{remark}

\subsection{Preliminary computations for local bilinear forms} \label{ssec:algebraic computations}
Let $(\e_*,\F)$ be a bilinear form satisfying the structural hypotheses of Assumption 0. 

Recall that for a non-negative function $u \in \F_{\mbox{\tiny{loc}}}$ and a positive integer $n$, we set
\[ u_n := u \wedge n. \]
Due to Assumption 0(iii), $d\Gamma_*(\psi,\psi) \leq c \, d\mu$ if and only if $d\Gamma(\psi,\psi) \leq c \, d\mu$. 

\begin{lemma} \label{lem:approximation nonsym} 
Let $p \in \R$ and $\psi \in \F \cap \mathcal C_{\mbox{\emph{\tiny{c}}}}(X)$. Let $0 \leq u \in \F_{\mbox{\emph{\tiny{loc}}}}$ with $u \in L^2(X,d\Gamma(\psi,\psi))$.
Assume either of the following hypotheses:
\begin{enumerate}
\item
$p\geq 2$,
\item
$p \neq 0$ and $u$ is locally uniformly positive.
\end{enumerate}
Then there exists a sequence of non-negative functions $(f_k) \subset \mathcal D$ such that $f_k \to u$ in $\F$, and
\begin{align*}
\e_*(u, u u_n^{p-2} \psi^2) 
= & \lim_{k \to \infty} \e_*(f_k, f_k (f_k \wedge n)^{p-2} \psi^2).
\end{align*}
\end{lemma}

\begin{proof}
It suffices to prove the lemma for $u \in \F$, because $\e_*$ is local and $\psi$ has compact support. 
By the strong locality of the reference form, we have $u \in L^2(X,d\Gamma(u_n^{p-2},u_n^{p-2}))$. 
As $u \in L^2(X,d\Gamma(\psi,\psi))$ by assumption, it follows by \eqref{eq:chain rule for Gamma} and \eqref{eq:CS} that
 $u \in L^2(X,d\Gamma(g,g))$, where $g := u_n^{p-2} \psi^2 \in \F \cap L^{\infty}$.

Since $u_m := u \wedge m$ converges to $u$ in $\F$ as $m \to \infty$, there exists by
Proposition \ref{prop:u_mg converges in F} a subsequence $(u_{m_k})$ such that $u_{m_k} \to u$ in $\F$ and 
$u_{m_k} g \to ug$ in $\F$ as $k \to \infty$. Let $g_k := (u_{m_k} \wedge n)^{p-2} \psi^2$. As we let $k \to \infty$, we obviously have $g_k \to g$ quasi-everywhere. 
Applying \eqref{eq:chain rule for Gamma}, Proposition \ref{prop:u_mg converges in F}, and passing to a subsequence which we again denote by $(g_k)$, we find that
 $g_k \to g$ in $\F$. 
Therefore, by \eqref{eq:chain rule for Gamma} and the dominated convergence theorem,
\begin{align*}
& \quad  || u_{m_k} g - u_{m_k} g_k||_{\F}^2 \\
&  \leq 
||u_{m_k}^2||_{\infty} ||g-g_k||_{\F}^2 + \int (g-g_k)^2 d\Gamma(u_{m_k}, u_{m_k}) + \int(g-g_k)^2 u_{m_k}^2 d\mu \\
&  \longrightarrow 0, \quad \mbox{ as } k \to \infty.
\end{align*}
Combining the above with Assumption 0(i), we obtain
\begin{align*}
& \quad | \e_*(u, u u_n^{p-2} \psi^2) - \e_*(u_{m_k}, u_{m_k}(u_{m_k} \wedge n)^{p-2} \psi^2)| \\
& \leq |\e_*(u - u_{m_k}, u g)| + |\e_*(u_{m_k}, ug - u_{m_k} g)| + |\e_*(u_{m_k}, u_{m_k} g  - u_{m_k} g_k)| \\
&  \longrightarrow 0, \quad \mbox{ as } k \to \infty.
\end{align*}
This shows that
\begin{align*}
\e_*(u, u u_n^{p-2} \psi^2) 
& = \lim_{k \to \infty} \e_*(u_{m_k}, u_{m_k}(u_{m_k} \wedge n)^{p-2} \psi^2).
\end{align*}

Since $\mathcal D$ is dense in $(\F,\| \cdot\|_{\F})$, there exists, for each $k$, a sequence $(f_l) \subset \mathcal D$ such that
$f_l \to u_{m_k}$ in $\F$ as $l \to \infty$. Since we may assume that $0 \leq f_l \leq m_k$ for all $k$, applying Proposition \ref{prop:u_mg converges in F} and passing to a subsequence, we get that
 $f_l g_k \to u_{m_k} g_k$ in $\F$ as $l \to \infty$. As we let $k \to \infty$, we can find a diagonal subsequence 
$(f_{k,l(k)}) \subset \mathcal D$ such that $f_{k,l(k)} \to u$ in $\F$, and, by Assumption 0(i),
\begin{align*}
\e_*(u, u u_n^{p-2} \psi^2) 
& = \lim_{k \to \infty} \e_*(f_{k,l(k)}, (f_{k,l(k)} \wedge n)^{p-2} \psi^2).
\end{align*}
\end{proof}

The next lemma is an immediate consequence of Lemma \ref{lem:SUP sym2}. We state it here for convenience because we will apply it together with the other results in this subsection, though in the main text we will refer to it as Lemma \ref{lem:SUP sym2}.

\begin{lemma}
Let $p \in \R$,  $\psi \in \F \cap \mathcal C_{\mbox{\em{\tiny{c}}}}(X)$, and $0 \leq u \in \F_{\mbox{\em{\tiny{loc}}}}$ with $u \in L^2(X,d\Gamma_*(\psi,\psi))$.
Assume one of the following hypotheses.
\begin{enumerate}
\item
$p\geq 2$,
\item
$u$ is locally uniformly positive.
\end{enumerate} 
Then $u u_n^{p-2} \in \F_{\mbox{\em{\tiny{loc}}}}$, $u u_n^{p-2} \psi^2 \in \F_{\mbox{\em{\tiny{c}}}}$, and for any $k >0$ we have
\begin{align*} 
(1-p) \e_*^{\mbox{\em{\tiny{s}}}}(u, u u_n^{p-2} \psi^2) 
& \leq  4k \int u^2 u_n^{p-2} d\Gamma_*(\psi,\psi)  \nonumber \\
      & \quad + \left( \frac{|1-p|^2}{k} + (1-p) \right) \int \psi^2  u_n^{p-2}  d\Gamma_*(u,u) \nonumber \\       
      & \quad - ( (1-p)^2 + (1-p)) \int \psi^2  u_n^{p-2} d\Gamma_*(u_n,u_n).
\end{align*}
\end{lemma}

\begin{proposition} \label{prop:skew identity with p} 
Let $p \in \R$, $\psi \in \F \cap \mathcal C_{\mbox{\em{\tiny{c}}}}(X)$, 
and $0 \leq u \in \F_{\mbox{\emph{\tiny{loc}}}}$ with $u \in L^2(X, d\Gamma_*(\psi,\psi))$. Assume either of the following hypotheses:
\begin{enumerate}
\item
$p\geq 2$,
\item
$p \neq 0$ and $u$ is locally uniformly positive.
\end{enumerate}
Then
\begin{align*}
\e_*^{\mbox{\emph{\tiny{skew}}}}(u, u u_n^{p-2} \psi^2) 
& =  \e_*^{\mbox{\emph{\tiny{skew}}}}(u u_n^{\frac{p-2}{2}},u u_n^{\frac{p-2}{2}}\psi^2) + \frac{2-p}{p} \e_*^{\mbox{\emph{\tiny{skew}}}}(u_n^{p/2},u_n^{p/2} \psi^2) \\
& \quad +  \frac{2-p}{p} \e_*^{\mbox{\emph{\tiny{skew}}}}(u_n^p \psi^2,1) .
\end{align*}
\end{proposition}

\begin{proof}
Due to Lemma \ref{lem:approximation nonsym} , Assumption 0(i) and an approximation argument, it suffices to consider the case when $0 \leq u \in \mathcal D$. By a density argument and the properties of $\mathcal D$, it suffices to consider $\psi \in \mathcal D$.

Adding and subtracting $\l_*(u u_n^{p-2},u \psi^2)$ on the right hand side and using the Leibniz rule for $\r_*$ (in the second argument), we obtain
\begin{align*}
\e_*^{\mbox{\tiny{skew}}}(u, u u_n^{p-2} \psi^2) 
& =  \l_*(u, u u_n^{p-2} \psi^2) + \r_*(u, u u_n^{p-2} \psi^2) \\
& =  \l_*(u, u u_n^{p-2} \psi^2) - \l_*(u u_n^{p-2},u \psi^2) + \r_*(u^2 u_n^{p-2},\psi^2).
\end{align*}
We consider the three terms on the right hand side separately. 
By the Leibniz rule for $\r_*$ and the fact that $\r_*(f,g) + \l_*(g,f)=0$ and $\r_*(f,g) + \l_*(f,g) = \e_*^{\mbox{\tiny{skew}}}(f,g)$ by definition, we have
\begin{align*}
\r_*(u^2 u_n^{p-2},\psi^2)
& =  \r_*(u^2 u_n^{p-2},\psi^2) + \r_*(u u_n^{\frac{p-2}{2}} \psi^2, u u_n^{\frac{p-2}{2}}) +  \l_*(u u_n^{\frac{p-2}{2}}, u u_n^{\frac{p-2}{2}}\psi^2) \\
& =  \e_*^{\mbox{\tiny{skew}}}(u u_n^{\frac{p-2}{2}},u u_n^{\frac{p-2}{2}}\psi^2).
\end{align*}
Due to the strong left-locality of $\l_*$ and the fact that $u = u_n + (u-u_n)$, we have
\begin{align} \label{eq:l1 1.16}
 \l_*(u, u u_n^{p-2} \psi^2)
& =  \l_*(u_n, u u_n^{p-2} \psi^2) + \l_*(u-u_n, u u_n^{p-2} \psi^2) \nonumber \\
& =  \l_*(u_n, u_n^{p-1} \psi^2) + \l_*(u-u_n, u u_n^{p-2} \psi^2).
\end{align}
Observe that, by the locality and bilinearity of $\l_*$,
\begin{align*}
\l_*((u-u_n) u_n^{p-2},u \psi^2) 
& = \l_*((u-u_n) n^{p-2},u \psi^2) = \l_*(u-u_n, n^{p-2} u \psi^2) \\
& = \l_*(u-u_n, u_n^{p-2} u \psi^2),
\end{align*}
hence, applying again the strong left-locality of $\l_*$, the Leibniz rule for $\l_*$, we get
\begin{align} \label{eq:l2 1.16}
- \l_*(u u_n^{p-2},u \psi^2)
& = - \l_*(u_n^{p-1},u \psi^2) - \l_*((u-u_n) u_n^{p-2},u \psi^2) \nonumber\\
& = - \l_*(u_n^{p-1},u_n \psi^2) -  \l_*(u-u_n, u u_n^{p-2} \psi^2) \nonumber \\
& = - \l_*(u_n^p, \psi^2) + \l_*(u_n,u_n^{p-1}\psi^2) -  \l_*(u-u_n, u u_n^{p-2} \psi^2) \nonumber \\
& = - (p-1) \l_*(u_n,u_n^{p-1} \psi^2) - \l_*(u-u_n, u u_n^{p-2} \psi^2).
\end{align}
In the last equality, we applied the chain rule for $\l_*$ with $\Phi(x) = x^p$ for $x \geq 0$ and $\Phi(x)=0$ for $x < 0$ (for $0 \neq p < 2$), 
and $\Phi(x) = x^p$ (for $p\geq2$).
Combining \eqref{eq:l1 1.16} and \eqref{eq:l2 1.16}, and applying the chain rule and the Leibniz rule for $\l_*$, we obtain
\begin{align*}
& \l_*(u, u u_n^{p-2} \psi^2) - \l_*(u u_n^{p-2},u \psi^2) \\
& = (2-p) \l_*(u_n, u_n^{p-1} \psi^2) \\
& = \frac{(2-p)}{p} \l_*(u_n^p,\psi^2) \\
& =  \frac{2(2-p)}{p} \l_*(u_n^{p/2},u_n^{p/2} \psi^2)  \\
& =  \frac{2-p}{p} \e_*^{\mbox{\tiny{skew}}}(u_n^p \psi^2,1) + \frac{2-p}{p} \e_*^{\mbox{\tiny{skew}}}(u_n^{p/2},u_n^{p/2}\psi^2).
\end{align*}
\end{proof}

\begin{corollary} \label{cor:skew identity with p for u_eps} 
Let $0 \neq p \in \R$, $\psi \in \F \cap \mathcal C_{\mbox{\em{\tiny{c}}}}(X)$, 
and $0 \leq u \in \F_{\mbox{\emph{\tiny{loc}}}}(X)$. Assume that $u$ is locally bounded. Let $\varepsilon >0$.
Then
\begin{align*}
\e_*^{\mbox{\emph{\tiny{skew}}}}(u_{\varepsilon}, u_{\varepsilon}^{p-1} \psi^2) 
& =  \frac{2}{p} \e_*^{\mbox{\emph{\tiny{skew}}}}(u_{\varepsilon}^{p/2},u_{\varepsilon}^{p/2}\psi^2) 
+  \frac{2-p}{p} \e_*^{\mbox{\emph{\tiny{skew}}}}(u_{\varepsilon}^p \psi^2,1) .
\end{align*}
\end{corollary}

\begin{proof}
By the regularity of the reference form, by Proposition \ref{prop:u_mg converges in F} and by \eqref{eq:chain rule for Gamma} and \eqref{eq:CS}, we can find a sequence $(f_k) \subset \F \cap \mathcal C_{\mbox{\tiny{c}}}(X)$ such that
$f_k \to u$ in $\F$, $(f_k+\varepsilon)^q \to u_{\varepsilon}^q$ locally in $\F$, and $(f_k+\varepsilon)^{q} \psi^2 \to u_{\varepsilon}^{q} \psi^2$ in $\F$, 
for finitely many $q \in \R$.
Hence the assertion follows from Proposition \ref{prop:skew identity with p} together with an approximation argument and Assumption 0(i).
\end{proof}

\subsection{Time-dependence and quantitative assumptions on the bilinear forms} \label{ssec:assumptions on the forms}

In this section, we consider families of 
time-dependent forms each of which is of the type introduced in Assumption 0.

For every $t \in \R$, let $(\e_t,\F)$ be a (possibly non-symmetric) local 
bilinear form. Throughout, we assume further that for every 
$f,g \in \F$ the map $t \mapsto \e_t(f,g)$ is measurable and that, 
for each $t$, $\e_t$ satisfies the structural hypotheses
introduced in Assumption 0.  In particular, for every $t \in \R$,   
$$\e_t^{\mbox{\tiny{s}}}(f,g) = 
\e^{\mbox{\tiny sym}}_t(f,g) - 
\e_t^{\mbox{\tiny{sym}}}(fg,1),$$ 
$f,g \in \F$, $fg \in \F_{\mbox{\tiny{c}}}$, extends to $\F\times \F$ as  
a symmetric regular strongly local Dirichlet form with domain $\F$ and energy measure
$\Gamma_t$.

\begin{assumption} \label{as:e_t}
\begin{enumerate}
\item
There is a constant $C_1 \in [1,\infty)$ so that for all $t \in \R$ and all bounded Borel measurable functions $f:X \to (-\infty,+\infty)$, and all $g \in \F \cap \mathcal C_{\mbox{\emph{\tiny{c}}}}(X)$,
\begin{align*}
C_1^{-1} \int f^2 d\Gamma(g,g) \leq \int f^2 d\Gamma_t(g,g) \leq C_1 \int f^2 d\Gamma(g,g),
\end{align*}
where $\Gamma_t$ is the energy measure of $\e_t^{\mbox{\emph{\tiny{s}}}}$.
\item
There are constants $C_2, C_3 \in [0,\infty)$ so that for all $t \in \R$ and all $f \in \F \cap \mathcal C_{\mbox{\emph{\tiny{c}}}}(X)$,
\begin{align*}
|\e_t^{\mbox{\emph{\tiny{sym}}}}(f^2,1)| 
\leq 2 \left(\int f^2 d\mu \right)^{\frac{1}{2}} \left(C_2 \int d\Gamma(f,f) + C_3 \int f^2 d\mu \right)^{\frac{1}{2}}
\end{align*}
\item
There are constants $C_4, C_5 \in [0,\infty)$ such that for all $t \in \R$ and all $f,g \in \F \cap \mathcal C_{\mbox{\emph{\tiny{c}}}}(X)$,
\begin{align*}
 \left| \e_t^{\mbox{\emph{\tiny{skew}}}}(f,fg^2) \right|
 \leq  2 \left( \int f^2 d\Gamma(g,g) \right)^{\frac{1}{2}} \left( C_4 \int g^2 d\Gamma(f,f) + C_5 \int f^2 g^2 d\mu \right)^{\frac{1}{2}}.
\end{align*}
\end{enumerate}
\end{assumption}

\begin{assumption} \label{as:p=0}
There are constants $C_6, C_7 \in [0,\infty)$ such that for all $t \in \R$,
\begin{align*} \big|\e_t^{\mbox{\emph{\tiny{skew}}}}(f,f^{-1} g^2)\big| 
& \leq  2 \left( \int d\Gamma(g,g) \right)^{\frac{1}{2}} \left( C_6 \int g^2 d\Gamma(\log f,\log f)  \right)^{\frac{1}{2}} \\
 & \quad + 2 \left( \int d\Gamma(g,g) + \int g^2 d\Gamma(\log f,\log f) \right)^{\frac{1}{2}} \left( C_7  \int g^2 d\mu \right)^{\frac{1}{2}},
\end{align*}
for all $g \in \F \cap \mathcal{C}_{\mbox{\emph{\tiny{c}}}}$ and all $0 \leq f \in \F_{\mbox{\emph{\tiny{loc}}}} \cap \mathcal C$ with $f + f^{-1} \in L^{\infty}_{\mbox{\emph{\tiny{loc}}}}(X,\mu)$.
\end{assumption}

\begin{example}
Consider the divergence form bilinear forms on Euclidean space of Example \ref{ex:Euclidean}. Suppose all coefficients $a_{i,j}, b_i, d_i, c$ are bounded. 
Then these bilinear forms satisfy Assumptions \ref{as:e_t} and \ref{as:p=0}.
The constants $C_4$, $C_6$ can be taken to be equal to $0$ 
when $(a_{i,j})$ is symmetric. The constants $C_2$, $C_5$, $C_7$ can be taken to be equal to $0$ when $b_i=d_i=0$ for all $i$. The constant $C_3$ can be taken to be equal to $0$ when $c=0$.
\end{example}

\begin{remark} \label{rem:assumptions}
\begin{enumerate}
\item
In Assumption \ref{as:e_t}, 
(i) extends to all $g \in \F$, 
(ii) extends to all $f \in \F_{\mbox{\tiny{c}}} \cap L^{\infty}(X,\mu)$, and 
(iii) extends to all $0 \leq f \in \F \cap L^{\infty}(X,\mu)$, $g \in \F \cap L^{\infty}(X,\mu)$, and by locality also to 
$0 \leq f \in \F_{\mbox{\tiny{loc}}} \cap L^{\infty}_{\mbox{\tiny{loc}}}(X,\mu)$, $g \in \F_{\mbox{\tiny{c}}} \cap L^{\infty}(X,\mu)$.
This follows by approximating $f$ and $g$ by functions in $\F \cap \mathcal C_{\mbox{\tiny{c}}}(X)$ and by applying 
Assumption 0 and Proposition \ref{prop:u_mg converges in F}.
\item
Assumption \ref{as:p=0} extends to all $g \in \F_{\mbox{\tiny{c}}} \cap L^{\infty}(X,\mu)$ and $0 \leq f \in \F_{\mbox{\tiny{loc}}}$ 
with $f + f^{-1} \in L^{\infty}_{\mbox{\tiny{loc}}}(X,\mu)$, and to all $f,g \in \F \cap L^{\infty}(X,\mu)$ with $f \geq 0$ and $f+f^{-1} \in L^{\infty}(X,\mu)$.
\item
When we apply Assumption \ref{as:e_t} and \ref{as:p=0} in computations, we will often make use of the elementary inequality $\sqrt{a+b} \leq\sqrt{a} + \sqrt{b}$ for $a,b \geq 0$. 
\item
The forms $\e_t$ satisfy the above assumptions if and only if the adjoint forms $\e^*_t(f,g):=\e_t(g,f)$ satisfy them.
\item
If Assumption \ref{as:e_t}(iii) is satisfied with $C_4=0$, then Assumption \ref{as:p=0} is satisfied with $C_6=0$. To see this, apply Assumption \ref{as:e_t}(iii) to $\e^{\mbox{\tiny{skew}}}_t(f,f^{-1}g^2) = \e^{\mbox{\tiny{skew}}}_t(f,f (f^{-1}g)^2)$.
\end{enumerate}
\end{remark}

\begin{definition}
Suppose {\em Assumption 0} and {\em Assumption \ref{as:e_t}}(i), (ii) are satisfied.
\[ D(L_t) = \{ f \in \F : g \mapsto \e_t(f,g) \textrm{ is continuous w.r.t. } \Vert \cdot \Vert_{2} \textrm{ on } \F_{\mbox{\emph{\tiny{c}}}} \}. \]
For $f \in D(L_t)$, let $L_t f$ be the unique element in $L^2(X)$ such that
 \[ - \int L_t f g d\mu  = \e_t(f,g) \quad \textrm{ for all } g \in \F_{\mbox{\emph{\tiny{c}}}}. \]
Then we say that $(L_t,D(L_t))$ is the infinitesimal generator 
of $(\e_t,\F)$ on $X$.
\end{definition}
See, e.g., \cite[Section IV.2]{EE87}, \cite{MR92}, or \cite{Osh13}.
The proof of the following proposition is straight-forward.
\begin{proposition} \label{prop:lower bounded form}
Under {\em Assumption 0} and {\em Assumption \ref{as:e_t}}(i), (ii), there exist constants $\alpha,c \in (0,\infty)$, depending on $C_1,C_2,C_3$, such that
\begin{align} \label{eq:lower bounded form}
 \e_t(f,f) + \alpha \int f^2 d\mu \ge c \| f \|_{\F}^2, \quad \forall f \in \F, t \in \R,
 \end{align}
and $\e_t(f,g)+\alpha \int fgd\mu$ is a positivity preserving coercive closed form with domain $\F$. In particular,  the semigroup generated by $(L_t,D(L_t))$ is positivity preserving. When $C_2=C_3 = 0$, then we can take $\alpha = c = 0$.
\end{proposition}

\begin{remark} \label{rem:lower bounded form}
The lower boundedness \eqref{eq:lower bounded form} ensures the existence
of weak solutions to the heat equation $\frac{\partial}{\partial t} u = L_t u$ under Assumption 0 and Assumption \ref{as:e_t}(i), (ii). See \cite{LM72} for an abstract treatment of this matter.
\end{remark}

We will need the following Lemma for the proof of Lemma \ref{lem:estimate subsol p<2}.
Let $\varepsilon >0$ and set
\[ u_{\varepsilon} := u + \varepsilon. \]
\begin{lemma} \label{lem:e(1,)} 
Suppose {\em Assumption \ref{as:e_t}} is satisfied.
Let $p \in \R$. Let $\psi \in \F_{\mbox{\emph{\tiny{c}}}} \cap L^{\infty}(X,\mu)$ and $0 \leq u \in \F_{\mbox{\emph{\tiny{loc}}}}$. Suppose $u$ is locally bounded. 
Then for any $t \in \R$, $k \geq 1$,
\begin{align*}
 |\e_t(\varepsilon,u_{\varepsilon}^{p-1} \psi^2)| 
& \leq  \frac{4}{k} \int u_{\varepsilon}^p d\Gamma(\psi,\psi) 
+ \frac{(p-1)^2}{k} \int \psi^2 u_{\varepsilon}^{p-2} d\Gamma(u_{\varepsilon},u_{\varepsilon})  \\
& \quad + 2(C_2+C_3^{\frac{1}{2}}+C_5) k \int u_{\varepsilon}^p \psi^2 d\mu.
\end{align*}
\end{lemma}
\begin{proof} 
 Observe that 
\[ |\e_t(\varepsilon,u_{\varepsilon}^{p-1} \psi^2)| 
=  \varepsilon |\e_t^{\mbox{\tiny{skew}}}(1,u_{\varepsilon}^{p-1} \psi^2) 
               + \e_t^{\mbox{\tiny{sym}}}(1,u_{\varepsilon}^{p-1} \psi^2)|. \]
We apply Assumption \ref{as:e_t}(iii) and Remark \ref{rem:assumptions}(i) with $f = 1$ and $g = u_{\varepsilon}^{\frac{p-1}{2}} \psi$.
\begin{align*}
\varepsilon |\e_t^{\mbox{\tiny{skew}}}(1,u_{\varepsilon}^{p-1} \psi^2)|
& \leq 2 \left( \int \varepsilon d\Gamma( u_{\varepsilon}^{\frac{p-1}{2}} \psi, u_{\varepsilon}^{\frac{p-1}{2}} \psi) \right)^{\frac{1}{2}}
\left( C_5 \int \varepsilon u_{\varepsilon}^{p-1} \psi^2 d\mu \right)^{\frac{1}{2}} \\
& \leq \frac{2}{k}  \int \varepsilon u_{\varepsilon}^{p-1} d\Gamma(\psi,\psi) 
    + \frac{2}{k}  \int \varepsilon \psi^2 d\Gamma(u_{\varepsilon}^{\frac{p-1}{2}},u_{\varepsilon}^{\frac{p-1}{2}}) \\ 
& \quad  + k C_5 \int \varepsilon u_{\varepsilon}^{p-1} \psi^2 d\mu \\
& \leq  \frac{2}{k}  \int u_{\varepsilon}^p d\Gamma(\psi,\psi) 
      + \frac{2}{k} \frac{(p-1)^2}{4} \int \psi^2 u_{\varepsilon}^{p-2} d\Gamma(u_{\varepsilon},u_{\varepsilon}) \\
& \quad  + k C_5 \int u_{\varepsilon}^p \psi^2 d\mu,
\end{align*}
for any $k>0$. Here we used \eqref{eq:Gamma(fg)}, \eqref{eq:chain rule for Gamma}, and the fact that $\varepsilon  \leq u_{\varepsilon}$.
By Assumption \ref{as:e_t}(ii) and Remark \ref{rem:assumptions}(i) applied with $f = u_{\varepsilon}^{\frac{p-1}{2}} \psi$,
\begin{align*}
& \quad \varepsilon |\e_t^{\mbox{\tiny{sym}}}(1,u_{\varepsilon}^{p-1} \psi^2)| \\
& \leq 2 \left( \int \varepsilon u_{\varepsilon}^{p-1} \psi^2 d\mu \right)^{\frac{1}{2}} 
\left( C_2 \int \varepsilon \, d\Gamma( u_{\varepsilon}^{\frac{p-1}{2}} \psi, u_{\varepsilon}^{\frac{p-1}{2}} \psi) + C_3 \int \varepsilon u_{\varepsilon}^{p-1} \psi^2 d\mu \right)^{\frac{1}{2}}.
\end{align*}
Now the right hand side can be estimated using \eqref{eq:Gamma(fg)}, \eqref{eq:chain rule for Gamma}, and the fact that $\varepsilon  \leq u_{\varepsilon}$.
\end{proof}

\section{Solutions of the heat equation} 
Let $(\e,\F)$ be a reference form as in Section \ref{ssec:reference form}.
Let $(\e_t,\F)$, $t \in \R$, be a family of local bilinear forms.
Throughout, we assume that, for all
$f,g \in \F$, the map $t \mapsto \e_t(f,g)$ is measurable.
Further, we assume that $(\e_t,\F)$, $t \in \R$ satisfies Assumption 0.

We point out that Assumption 0, \ref{as:e_t}, \ref{as:p=0} are all satisfied in the following special case, which we will refer to as the {\em "symmetric strongly local case"}:
For each $t \in \R$, $(\e_t,\F)$ is a symmetric strongly local regular Dirichlet form and there exists a constant $C_1 \in [1,\infty)$ such that
\begin{align}  \label{eq:sym strongly local case}
C_1^{-1} \e(f,f) \le \e(f,f) \le C_1 \e_t(f,f), \quad \forall f \in \F, \forall t \in \R.
\end{align}

\subsection{Local very weak solutions}

Let $L^2_{\mbox{\tiny{loc}}}(I \to \F;U)$ be the space of all functions $u:I \times U \to \R$ such that for any open interval $J$ relatively compact in $I$, and any open subset $A$ relatively compact in $U$, there exists a function $u^{\sharp} \in L^2(I \to \F)$ such that $u^{\sharp} = u$ a.e.~in $J \times A$.

\begin{definition} \label{def:local very weak solution}
Let $I$ be an open interval and $U \subset X$ open. Set $Q = I \times U$. A function $u: Q \to \R$ is a \emph{local very weak solution} of the heat equation $\frac{\partial}{\partial t} u = L_t u$ in $Q$, if
\begin{enumerate}
\item
$u \in L^2_{\mbox{\em\tiny{loc}}}(I \to \F;U)$,
\item 
For almost every $a,b \in I$,
\begin{align} \label{eq:loc weak sol}
 \forall \phi \in \F_{\mbox{\emph{\tiny{c}}}}(U), \quad  \int u(b,\cdot)  \phi \, d\mu - \int u(a,\cdot)  \phi \, d\mu + \int_a^b \e_t(u(t,\cdot),\phi) dt = 0.
\end{align} 
\end{enumerate} 
\end{definition}

\begin{definition}
Let $I$ be an open interval and $U\subset X$ open. Set $Q = I \times U$. A function $u: Q \to \R$ is a \emph{local very weak subsolution} of $\frac{\partial}{\partial t} u = L_t u$ in $Q$, if
\begin{enumerate}
\item
$u \in L^2_{\mbox{\em \tiny{loc}}}(I \to \F;U)$,
\item 
For almost every $a,b \in I$ with $a < b$, and any non-negative $\phi \in \F_{\mbox{\emph{\tiny{c}}}}(U)$,
\begin{align} \label{eq:weak sol}
 \int u(b,\cdot)  \phi \, d\mu - \int u(a,\cdot)  \phi \, d\mu + \int_a^b \e_t(u(t,\cdot),\phi) dt
\leq 0.
\end{align}
\end{enumerate} 
We also write $\frac{\partial}{\partial t} u \leq L_t u$ very weakly in $Q$ to indicate that a function $u$ is a local very weak subsolution in $Q$.

A function $u$ is called a \emph{local very weak supersolution} if $-u$ is a local very weak subsolution.
\end{definition}

\begin{remark}
\begin{enumerate}
\item 
A local very weak solution may not have a weak time-derivative as defined in Section \ref{ssec:weak time-derivative}. 
Moreover, it may not be in $\mathcal C(\overline{I} \to L^2(X,\mu))$. In Section \ref{sec:PHI}, however, we will 
consider only local very weak solution that are contained in $\mathcal C(\overline{I} \to L^2(X,\mu))$. 
\item 
In Section \ref{ssec:weak solutions} we recall the notion of local weak solutions which involves weak time-derivatives. We show in Proposition \ref{prop:weak and very weak} the relation between local very weak solutions and local weak solutions. 
\end{enumerate}
\end{remark}

\begin{proposition} \label{prop:|u| is subsolution}
If $u$ is a local very weak solution of $\frac{\partial}{\partial t} u = L_t u$ in $Q = I \times U$ then $|u|$ is a (non-negative) local very weak subsolution of $\frac{\partial}{\partial t} u = L_t u$ in $Q$. 
\end{proposition}

Proposition \ref{prop:|u| is subsolution} is proved in Section \ref{sec:proofs}.

It may not be true that the absolute value of a local weak solution is a local weak subsolution. We do not have an analog of Proposition \ref{prop:weak and very weak} for {\em sub}solutions.

For an open relatively compact time interval $I$, let $L^2(I \to \F)$ be the Hilbert space of those functions $v: I \to \F$ such that 
 \[  \Vert v \Vert_{L^2(I \to \F)} := \left( \int_I \Vert v(t) \Vert_{\F}^2 \, dt \right)^{1/2} < \infty. \]

The next proposition is a generalization of \cite[Theorem 2.14(i)]{FOT94} to time-dependent functions. Though it applies to any $u \in L^2(I \to \F)$, we will mainly
use it to approximate sub- and supersolutions.

\begin{proposition} \label{prop:dt-qe convergence}
Consider a Cauchy sequence $(u_n)$ in $L^2(I \to \F)$. Then there exists $u \in L^2(I \to \F)$ and a subsequence $(u_{n_k})$ such that, for almost every $t \in I$, $\widetilde{u_{n_k}}(t,\cdot) \to \widetilde{u}(t,\cdot)$ quasi-everywhere, and $u_{n_k} \to u$ in $L^2(I \to \F)$.
\end{proposition}

\begin{proof}
Choose a subsequence $(u_{n_k})$ such that 
\[ \int_I || u_{n_{k+1}}(t,\cdot) - u_{n_k}(t,\cdot) ||_{\F}^2 dt \leq 2^{-3k} \]
for all $k$.
Set
\[ E^t_k := \left\{ x \in X : |\widetilde{u_{n_{\nu+1}}}(t,x) - \widetilde{u_{n_{\nu}}}(t,x)| > 2^{-\nu} \mbox{ for some } \nu \geq k \right\}. \]
Clearly, $E^t_k \supset E^t_{k+1}$. For some $\nu \geq k$,
\[ \mbox{Cap}(E^t_k) \leq \big(2^{\nu} ||u_{n_{\nu+1}}(t,\cdot) - u_{n_{\nu}}(t,\cdot)||_{\F} \big)^2, \]
and,
\begin{align*}
2^{2\nu} \int_I ||{u_{n_{\nu+1}}}(t,\cdot) - {u_{n_{\nu}}}(t,\cdot)||_{\F}^2 dt  \leq 2^{-k}.
\end{align*}
As $k \to \infty$, $\int_I \mbox{Cap}(E^t_k) dt \to 0$. Thus, passing again to a subsequence, we find that, for almost every $t \in I$,
\[ 2^{2\nu} ||u_{n_{\nu+1}}(t,\cdot) - u_{n_{\nu}}(t,\cdot)||_{\F}^2 \longrightarrow 0, \]
hence
$\mbox{Cap}(E^t_k) \to 0$ as $k \to \infty$. 
For any such $t$, let
 $F^t_k := X \setminus E^t_k$. Then $(F^t_k)$ is an $\e$-nest. For any $x \in F^t_k$ and any $l,m > N \geq k$, we have
\[ |\widetilde{u_{n_l}}(t,x) - \widetilde{u_{n_m}}(t,x)| 
\leq \sum_{\nu=N+1}^{\infty} |\widetilde{u_{n_{\nu}}}(t,x) - \widetilde{u_{n_{\nu+1}}}(t,x)|  
\leq 2^{-N}. \]
That is, $\widetilde{u_{n_l}}(t,\cdot)$ converges uniformly on each $F^t_k$, hence quasi-everywhere on $X$. Let $\widetilde{u}(t,\cdot)$ be the limit function. As in \cite[Proof of Theorem 2.1.4(i)]{FOT94} it can be proved that $\widetilde{u}(t,\cdot)$ is indeed quasi-continuous: For any $\epsilon >0$, let $k$ be large enough so that $\mbox{Cap}(E^t_k) < \epsilon /2$. By \cite[Theorem 2.1.2(i)]{FOT94}, there exists an open subset $G^t_k \subset X$ such that $\mbox{Cap}(G^t_k) < \epsilon /2 $ and, for every $l \geq k$, $\widetilde{u_{n_l}}(t,\cdot)$ is continuous on $X \setminus G^t_k$. Let $G^t := E^t_k \cup G^t_k$. Then $\mbox{Cap}(G^t) < \epsilon$ and $\widetilde{u_{n_k}}(t,\cdot)$ converges to $\widetilde u(t,\cdot)$ uniformly on $X \setminus G^t$. Hence, $\widetilde{u}(t,\cdot)$ is a quasi-continuous version of some function $u(t,\cdot) \in \F$. 

Since $u_{n_k}$ is by assumption a Cauchy sequence in the Hilbert space $L^2(I \to \F)$, passing again to a subsequence we obtain that $u_{n_k}(t,\cdot)$ converges in $\F$, at almost every $t \in I$. Its limit in $\F$ must be the same as the pointwise (in $x$) limit, so
 $||u_{n_k}(t,\cdot) - u(t,\cdot)||_{\F} \to 0$ as $k \to \infty$, at almost every $t \in I$. Being a Cauchy sequence in the Hilbert space
 $L^2(I \to \F)$, the sequence $(u_{n_k})$ is convergent. Its limit must be the same as its pointwise (in $t$) limit $u$.
 That is, $u \in L^2(I \to \F)$ and $u_{n_k} \to u$ in $L^2(I \to \F)$.
\end{proof}

\subsection{Steklov averages}

Let $d$ be a metric on $X$ inducing the original topology, and assume that open balls $B(x,r) = \{ y \in X: d(x,y) < r \}$ are relatively compact.

\begin{definition}
Let $I :=(a,b)  \subset \R$. For any $f \in L^1(I \to \F)$ and $h \in (0,b-a)$, the \emph{Steklov average} of $f$ is defined as
\[  f_h(t) := \frac{1}{h} \int_t^{t+h} f(s) ds, \quad t \in (a,b-h), \]
where the right hand side is a Bochner integral in the Banach space $(\F,|| \cdot ||_{\F})$. 
\end{definition}

\begin{lemma} \label{lem:steklov integrated}
Let $I :=(a,b)  \subset \R$ and $h \in (0,b-a)$. Let $f:I \to [0,\infty)$ be Lebesgue integrable.
Then
\[ \int_a^{b-h} f_h(t) \, dt
\leq h f_h(a) + h f_h(b-h) + \int_{a+h}^{b-h} f(t) dt.\]
\end{lemma}

\begin{proof}
Integration by parts yields
\begin{align*}
 \int_a^{b-h} f_h(t) dt
& = 
(b-h) f_h(b-h) - a f_h(a) - \int_{a}^{b-h} \frac{t}{h} [f(t+h) - f(t) ] dt \\
& = 
(b-h) f_h(b-h) - a f_h(a) + \int_a^{a+h} \frac{t}{h} f(t) dt \\
& \quad - \int_{b-h}^b \frac{t-h}{h} f(t) dt
+ \int_{a+h}^{b-h} \frac{t-(t-h)}{h} f(t) dt \\
& \leq 
h f_h(a) + h f_h(b-h) + \int_{a+h}^{b-h} f(t) dt.
\end{align*}
In the last line we used the assumption that $f$ is a non-negative function.
\end{proof}

\begin{lemma} \label{lem:steklov convergence}
Let $I :=(a,b)  \subset \R$. For any $f \in L^1(I \to \F)$, its Steklov average $f_h$ is in $L^2((a,b-h) \to \F)$. For almost every $t \in I$, 
\[ \lim_{h\to 0} \frac{1}{h} \int_t^{t+h} ||f(t) - f(s)||_{\F} \, ds = 0, \]
and
\[ \lim_{h\to 0} || f_h(t) - f(t) ||_{\F} = 0. \]
In particular, for any $\epsilon \in (0,b-a)$, the Steklov average $f_h$ converges to $f$ in $L^1((a,b-\epsilon) \to \F)$ as $h \to 0$.
\end{lemma}

\begin{proof}
See \cite[Theorem 9]{DiestelUhl77}. The convergence in $L^1((a,b-\epsilon) \to \F)$ follows from the pointwise convergence and the dominated convergence theorem. Indeed, we have
\[  \int_a^{b-\epsilon} \| f_h(t) \|_{\F}  \, dt \le \int_a^{b-\epsilon} ( \| f(\cdot) \|_{\F})_h(t) \, dt, \]
and the right hand side can be bounded from above by applying Lemma \ref{lem:steklov integrated}, hence $ (f_h)_{h \ge 0}$ is a bounded sequence in $L^1((a,b-\epsilon) \to \F)$.
\end{proof}

\begin{corollary} \label{cor:steklov convergence uniformly in t}
Let $I :=(a,b)  \subset \R$. For any $f \in \mathcal C(\overline{I} \to L^2(X,\mu))$,
\[ \lim_{h\to 0} \sup_{t \in I} || f_h(t) - f(t) ||_{L^2(X)} = 0. \]
\end{corollary}

\begin{corollary} \label{cor:steklov convergence in L^2}
Let $I :=(a,b)  \subset \R$ and $\epsilon \in (0,b-a)$. For any $f \in L^2(I \to \F)$, 
\[  \lim_{h \to 0} \int_a^{b-\epsilon} \frac{1}{h} \int_t^{t+h} || f(s) - f(t)  ||_{\F}^2 \, ds \, dt = 0. \]
In particular, the Steklov average $f_h$ is in $L^2((a,b-\epsilon) \to \F)$ and converges to $f$ in $L^2((a,b-\epsilon) \to \F)$ as $h \to 0$.
\end{corollary}

\begin{proof}
By Lemma \ref{lem:steklov convergence}, $\frac{1}{h} \int_t^{t+h} || f(s) - f(t)  ||_{\F}^2 \, ds$ converges to $0$ pointwise at almost every $t \in (a,b-\epsilon)$. By the triangle inequality 
\[ \frac{1}{h} \int_t^{t+h} || f(s) - f(t)  ||_{\F}^2 \, ds 
\leq \frac{2}{h} \int_t^{t+h} ||f(s)||_{\F}^2 ds + 2||f(t)||_{\F}^2, \]
hence, by Lemma \ref{lem:steklov integrated},
\[ \int_a^{b-\epsilon} \frac{1}{h} \int_t^{t+h} || f(s) - f(t)  ||_{\F}^2 \, ds \, dt \leq 6 \int_I ||f(t)||_{\F}^2 dt \quad \mbox { for all } h \in (0,\epsilon). \]
Thus, the first assertion follows by the dominated convergence theorem.
By Jensen's inequality,
\[ ||f_h(t) -  f(t)||_{\F}^2 \leq \frac{1}{h} \int_t^{t+h} || f(s) - f(t)  ||_{\F}^2 \, ds, \]
hence the second assertion follows.
\end{proof}

\subsection{Estimates for local very weak subsolutions and supersolutions} \label{ssec:estim for sub and supsol}

Fix parameters $\tau > \tau' > 0$. Let $B=B(x,r) \subset X$, $a \in \R$. For $\sigma \in (0,1)$, set
\begin{align*}
\sigma B & = B(x,\sigma r), \\
            I &= (a - \tau r^2, a + \tau r^2), \\
 I^-_{\sigma} &= (a - \sigma \tau r^2, a + \tau' r^2), \\
I^+_{\sigma}  &= (a - \tau' r^2, a + \sigma \tau r^2), \\
Q &=  Q(\tau,x,a,r) = I \times B(x,r), \\
 Q^-_{\sigma} &= I^-_{\sigma} \times \sigma B, \\
 Q^+_{\sigma} &= I^+_{\sigma} \times \sigma B.
\end{align*}
The parameter $\tau'$ is introduced to make sure that functions which are locally $L^2$-integrable over $I$ can be integrated over $ I^-_{\sigma}$ or $ I^+_{\sigma}$.

Let $0 < \sigma' < \sigma < 1$ and $\omega = \sigma - \sigma'$.
Let $\psi \in \F_{\mbox{\tiny{c}}}(B) \cap \mathcal{C}_{\mbox{\tiny{c}}}(B)$ be such that $0 \leq \psi \leq 1$, $\textrm{supp}(\psi) \subset \sigma B$, $\psi = 1$ in $\sigma'B$, and $d\Gamma(\psi,\psi) \leq 4 (\omega r)^{-2} d\mu$.
Let $\chi$ be a smooth function of the time variable $t$ such that $0 \leq \chi \leq 1$,  $\chi = 0 \textrm{ in } (-\infty, a - \sigma \tau r^2)$, $\chi = 1$ in $(a - \sigma' \tau r^2, \infty)$ and $0 \leq \chi' \leq 2/(\omega \tau r^2)$. Let $d\bar \mu = d\mu \times dt$. Recall that for $0 \leq u \in \F_{\mbox{\tiny{loc}}}$ we set $u_n = u \wedge n$.

\begin{theorem} \label{thm:estimate subsol p>2}
Suppose {\em Assumption \ref{as:e_t}} is satisfied. Let $p \geq 2$. Let $u$ be a non-negative local very weak subsolution of the heat equation for $L_t$ in $Q$, that is, $\frac{\partial}{\partial t}u \leq L_t u$ very weakly in $Q$.  Suppose $\int_{I^-_{\sigma}} \int_{\sigma B} u^p d\mu dt < \infty$. Then there are $a_1=a_1(C_1) \in (0,1)$, and $A_1, A_2 \in [0,\infty)$ depending on $C_1$ - $C_5$ such that
\begin{align} \label{eq:subsol}
& \sup_{t \in I^-_{\sigma'}} \int u^p \psi^2 d\mu + a_1 \int_{I^-_{\sigma'}} \int \psi^2 d\Gamma(u^{p/2},u^{p/2}) dt \nonumber \\
\leq & \, A_1(p^2+1) \int_{Q_{\sigma}^-} u^p d\Gamma(\psi,\psi) dt  \nonumber \\
& + \left( A_2(C_2+C_3^{1/2}+C_5) + \frac{2}{\omega\tau r^2} \right) (p^2+1) \int_{Q_{\sigma}^-} u^p \psi^2 d\mu \, dt.
\end{align}
\end{theorem}

The above theorem, as well as the following lemmas are proved in Section \ref{sec:proofs}.

\begin{lemma} \label{lem:estimate subsol p<2}
Suppose {\em Assumption \ref{as:e_t}} is satisfied. Let $p \in (1+\eta,2]$ for some small $\eta >0$. Let $u$ be a non-negative local very weak subsolution of the heat equation for $L_t$ in $Q$, that is, $\frac{\partial}{\partial t}u \leq L_t u$ very weakly in $Q$. Suppose that $u$ is locally bounded. Then there are $a_1 =a_1(C_1) \in (0,1)$, and $A_1, A_2 \in [0,\infty)$ depending on $C_1$--$C_5$ and $\eta$ such that
\begin{align} \label{eq:subsol p<2}
& \sup_{t \in I^-_{\sigma'}} \int u^p \psi^2 d\mu + a_1 \int_{I^-_{\sigma'}} \int \psi^2 d\Gamma(u^{p/2},u^{p/2}) dt \nonumber \\
\leq & \, A_1(p^2+1) \int_{Q_{\sigma}^-} u^p d\Gamma(\psi,\psi) dt \nonumber \\
& + \left( A_2(C_2+C_3^{1/2}+C_5) + \frac{2}{\omega\tau r^2} \right) (p^2+1) \int_{Q_{\sigma}^-} u^p \psi^2 d\mu \, dt.
\end{align}
\end{lemma}

\begin{lemma} \label{lem:estimate supsol}
Suppose {\em Assumption \ref{as:e_t}} is satisfied. Let $0 \neq p \in (-\infty, 1-\eta)$ for some small $\eta > 0$. Let $u \in \F_{\mbox{\emph{\tiny{loc}}}}(Q)$ be a 
non-negative local very weak supersolution of the heat equation for $L_t$ in $Q$, that is, $\frac{\partial}{\partial t}u \geq L_t u$ very weakly in $Q$. Suppose that $u$ is locally bounded. 
Let $u_{\varepsilon}:= u + \varepsilon$ for $\varepsilon>0$.
\begin{enumerate}
\item
If $p<0$, there are $a_1=a_1(C_1) \in (0,1)$, and $A_1, A_2 \in [0,\infty)$ depending on $C_1$ - $C_5$ such that
\begin{align} \label{eq:supsol p<0}
& \sup_{t \in I^-_{\sigma'}} \int u_{\varepsilon}^p \psi^2 d\mu + a_1 \int_{I^-_{\sigma'}} \int \psi^2 d\Gamma(u_{\varepsilon}^{p/2},u_{\varepsilon}^{p/2}) dt \nonumber \\
\leq & \, A_1(p^2+1) \int_{Q_{\sigma}^-} u_{\varepsilon}^p d\Gamma(\psi,\psi) dt \nonumber \\
& + \left( A_2(C_2 + C_3^{1/2}+C_5) + \frac{2}{\omega\tau r^2} \right) (p^2+1) \int_{Q_{\sigma}^-} u_{\varepsilon}^p \psi^2 d\mu \, dt.
\end{align}
\item
If $p\in (0,1-\eta)$, there are $a_1=a_1(C_1) \in (0,1)$, and $A_1, A_2 \in [0,\infty)$ depending on $C_1$ - $C_5$, $\eta$ such that
\begin{align} \label{eq:supsol 0<p<1}
& \sup_{t \in I^+_{\sigma'}} \int u_{\varepsilon}^p \psi^2 d\mu + a_1 \int_{I^+_{\sigma'}} \int \psi^2 d\Gamma(u_{\varepsilon}^{p/2},u_{\varepsilon}^{p/2}) dt \nonumber \\
\leq & \, A_1(p^2+1) \int_{Q_{\sigma}^+} u_{\varepsilon}^p d\Gamma(\psi,\psi) dt \nonumber \\
& + \left( A_2(C_2 + C_3^{1/2} + C_5) + \frac{2}{\omega\tau r^2} \right) (p^2+1) \int_{Q_{\sigma}^+} u_{\varepsilon}^p \psi^2 d\mu \, dt.
\end{align}
\end{enumerate}
\end{lemma}

\section{Main results} \label{sec:PHI}
\subsection{Harnack type spaces} \label{ssec:assumptions on the space}

The \emph{intrinsic distance} $d := d_{\e}$ induced by $(\e,\F)$ is defined as
\[ d_{\e}(x,y) := \sup \big\{ f(x)-f(y): f \in \F_{\mbox{\tiny{loc}}} \cap \mathcal{C}(X), \, d\Gamma(f,f) \leq d\mu \big\},  \]
for all $x,y \in X$, where $\mathcal{C}(X)$ is the space of continuous functions on $X$.
Consider the following properties of the intrinsic distance that may or may not be satisfied. They are discussed in \cite{Stu95geometry, SturmI}.
\begin{align}
& \mbox{The intrinsic distance } d  \mbox{ is finite everywhere and defines } \nonumber\\
& \mbox{the original topology of } X. \tag{A1} \\
& (X,d) \mbox{ is a complete metric space. } \tag{A2}
\end{align}
Note that if (A1) holds true, then (A2) is by \cite[Theorem 2]{Stu95geometry} equivalent to 
\begin{align}
\forall x \in X, r > 0, \textrm{ the open ball } B(x,r) \mbox{ is relatively compact in } (X,d). \tag{A2'}
\end{align}
Moreover, (A1)--(A2) imply that $(X,d)$ is a geodesic space, i.e.~any two points in $X$ can be connected by a minimal geodesic in $X$. See \cite[Theorem 1]{Stu95geometry}.
If (A1) and (A2) hold true, then by \cite[Proposition 1]{SturmI},
\[  d_{\e}(x,y) = \sup \big\{ f(x)-f(y): f \in \F \cap \mathcal{C}_{\mbox{\tiny{c}}}(X), \, d\Gamma(f,f) \leq d\mu \big\}, \quad x,y \in X. \]
It is sometimes sufficient to consider property (A2) only on an open, connected subset $Y$ of $X$, that is,
\begin{align}
\textrm{ For any ball } B(x,2r) \subset Y, B(x,r) \textrm{ is relatively compact.} \tag{A2-$Y$}
\end{align}
Note that  an open set $Y$ such that $\overline{Y}$ is complete in $(X,d)$
automatically satisfies (A2-$Y$), see, e.g., \cite[Lemma 1.1(i)]{SturmIII}.

\begin{definition}
$(\e,\F)$ satisfies the \emph{volume doubling property} on $Y$, if there exists a constant $D_Y \in (0,\infty)$ such that for every ball $B(x,2r) \subset Y$,
\begin{align} \label{eq:VD}
V(x,2r) \leq D_Y \, V(x,r), \tag{VD}
\end{align}
where $V(x,r) = \mu( B(x,r))$ denotes the volume of $B(x,r)$.
\end{definition}

\begin{definition}
$(\e,\F)$ satisfies the \emph{Poincar\'e inequality} on $Y$ if there exists a constant $P_Y \in (0,\infty)$ such that for any ball $B(x,2r) \subset Y$,
\begin{align} \label{eq:PI}
\forall f \in \F, \  \int_{B(x,r)} |f - f_B|^2 d\mu  \leq  P_Y \, r^2 \int_{B(x,2r)} d\Gamma(f,f), \tag{PI}
\end{align}
where $f_B = \frac{1}{V(x,r)} \int_{B(x,r)} f d\mu$ is the mean of $f$ over $B(x,r)$. 
\end{definition}

\begin{definition}
$(\e,\F)$ satisfies the localized \emph{Sobolev inequality} on $Y$ if there exist constants $\nu > 2$ and $S_Y > 0$ such that for any ball $B=B(x,r)$ with $B(x,2r) \subset Y$, we have
\begin{align}  \label{eq:sobolev}
 \left( \int_B |f|^{\frac{2\nu}{\nu-2}} d\mu \right)^{\frac{\nu-2}{\nu}}   
    \leq  S_Y \frac{ r^2 }{ V(x,r)^{2 / \nu} }  \left( \int_B d\Gamma(f,f) + r^{-2} \int_B f^2 d\mu \right),
\end{align}
for all $f \in \F_{\mbox{\emph{\tiny{c}}}}(B)$.
\end{definition}

\begin{theorem}\textbf{\em (\cite[Theorem 2.6]{SturmIII})}
Let $(X,\mu,\e,\F)$ be as above and $Y \subset X$. If \emph{(A1), (A2-$Y$)}, volume doubling and Poincar\'e inequality hold on $Y$, then the localized Sobolev inequality holds on $Y$.
\end{theorem}


In what follows we will consider the following assumptions where $Y$ is a fixed open subset of $X$.
\begin{assumption} \label{as:sobolev}
The form $(\e,\F)$ satisfies \emph{(A1)}, \emph{(A2-$Y$)}, and the localized Sobolev inequality on $Y$.
\end{assumption}

\begin{assumption} \label{as:VD+PI}
The form $(\e,\F)$ satisfies \emph{(A1)}, \emph{(A2-$Y$)}, volume doubling on $Y$ and the Poincar\'e inequality on $Y$.
\end{assumption}

\begin{remark}
Assumption \ref{as:sobolev} implies the volume doubling property on $Y$. See, e.g., \cite[Theorem 5.2.1]{SC02}.
\end{remark}

\subsection{Mean value estimates and local boundedness} \label{ssec:Mean value estimates}
We follow \cite{AS67} and \cite{SC02}. In this section we suppose 
that Assumption \ref{as:e_t} holds true  and that Assumption \ref{as:sobolev} is 
satisfied on the open set $Y$. 
We use the notation of Section \ref{ssec:estim for sub and supsol}. Let $A_1$ and $A_2$ be large enough so that the estimates of Section \ref{ssec:estim for sub and supsol} hold with these constants.
Let $A_2' = A_2(C_2 + C_3^{1/2} + C_5)$.
Fix a ball $B=B(x,r)$ with $B(x,2r) \subset Y$.
Let $\psi \in \F_{\mbox{\tiny{c}}}(B)$ be a cut-off function
so that $0 \leq \psi \leq 1$, $\textrm{supp}(\psi) \subset \sigma B$, $\psi = 1$ in $\sigma'B$, and
\[ d\Gamma(\psi,\psi) \leq 4(\omega r)^{-2} d\mu. \]

\begin{theorem} \textbf{\em (Mean value estimate with $p>1$)} \label{thm:MVE subsol p>2}
Suppose {\em Assumptions \ref{as:e_t}} and {\em \ref{as:sobolev}} 
are satisfied.
Let $p> 1+\eta$ for some small $\eta >0$. Fix a ball $B = B(x,r)$ with $B(x,2r) \subset Y$ and $\tau > 0$.
Then there exists a constant $A=A(\tau, \nu, \eta, C_1 \mbox{ - } C_5)$ such that, for any real $a$, any $0 < \delta < \delta' \leq 1$, and any 
 non-negative local very weak subsolution $u$ of the heat equation for $L_t$ in $Q = Q(\tau,x,a,r)$, we have
\begin{equation} \label{eq:MVE subsol p>2}
 \sup_{Q^-_{\delta}} \{ u^p \}  
\leq  \frac{ A \, S_Y^{\nu/2} [ (A_1 + A_2' \tau r^2)(p^2+1) ]^{1+\nu/2} }{(\delta' - \delta)^{2+\nu} r^2 V(x,r) } 
      \int_{Q^-_{\delta'}} u^p d\bar \mu,
\end{equation}
where $\nu$, $S_Y$ are the constants of the localized Sobolev inequality \eqref{eq:sobolev}.
\end{theorem}

\begin{proof}
For simplicity, we assume that $\tau = \delta'=1$. First, consider the case $p\geq 2$.
Let $E(B) = S_Y r^2 V(x,r)^{-2/\nu}$ be the Sobolev constant for the ball $B$ given by \eqref{eq:sobolev}, and set $\beta = \nu / (\nu - 2)$. By the H\"older inequality, we have for any $v \in \F_{\mbox{\tiny{c}}}(B)$,
 \[ \int_B v^{2(1+2/\nu)} d\mu  
    \leq  \left( \int_B v^{2\beta} d\mu \right)^{1/\beta}  \left( \int_B v^2 d\mu \right)^{2/\nu}. \]
So, \eqref{eq:sobolev} gives   
\begin{align} \label{eq:HoelderSobolev}
 \int_B v^{2(1+2/\nu)} d\mu  
    \leq  E(B) \left( \int_B d\Gamma(v,v) + r^{-2} \int_B v^2 d\mu \right) \left( \int_B v^2 d\mu \right)^{2/\nu}.
\end{align}

Since $u$ is a very weak subsolution, we have in particular that $u \in L^2_{\mbox{\tiny{loc}}}(I \to \F;B)$.

Set $w:=\psi u$. Then for almost every $t \in I$, $v:=w(t,\cdot)$ is in $\F_{\mbox{\tiny{c}}}(B)$ and satisfies \eqref{eq:HoelderSobolev}. Integrating over $I^-_{\sigma'}$ and applying H\"older inequality, we get
\begin{align} \label{eq:HoelderSobolev integrated}
 \int_{I^-_{\sigma'}} \int_B w^{2\theta} d\mu dt 
    \leq & E(B) \left(\int_{I^-_{\sigma'}} \int_B d\Gamma(w,w)dt + r^{-2} \int_{I^-_{\sigma'}} \int_B w^2 d\mu dt\right) \nonumber \\
    & \sup_{t \in I^-_{\sigma'}} \left( \int_B w^2 d\mu \right)^{2/\nu},
\end{align}
where $\theta=1+2/\nu$.
Note that the right hand side is finite by Theorem \ref{thm:estimate subsol p>2} (applied with $p=2$).
Hence the left hand side is finite and this means that $w^{\theta} $ is in $L^2(I^-_{\sigma'} \to L^2(B))$. In fact, applying Theorem
\ref{thm:estimate subsol p>2} with $p = 2\theta$, we obtain that $w^{\theta}$ is in $L^2(I^-_{\sigma''} \to \F)$ for any $0 < \sigma'' < \sigma'$. 
In particular, $u^{\theta} \in L^2_{\mbox{\tiny{loc}}}(I^-_{\sigma''} \to \F;\sigma'' B)$. Since the choice of $\sigma'',\sigma',\sigma$ was arbitrary, 
the same argument (applied with a cutoff function $\hat \psi$ which takes the value $1$ on $\sigma B$) proves that 
$u^{\theta} \in L^2_{\mbox{\tiny{loc}}}(I^-_{\sigma} \to \F;\sigma B)$.
 
Now set $w:=\psi u^{\theta}$. Then $w(t,\cdot)$ is in $\F_{\mbox{\tiny{c}}}(B)$ for a.e.~$t \in I^-_{\sigma}$ and satisfies \eqref{eq:HoelderSobolev integrated}. 
Repeating the argument of the whole previous paragraph, we get $u^q \in L^2_{\mbox{\tiny{loc}}}(I^-_{\sigma} \to \F';\sigma B)$ for $q = \theta^2$.
By induction, we obtain that for any positive integer $i$ and $q = \theta^i$, $w:=\psi u^{pq/2}$ is in $L^2(I^-_{\sigma} \to \F)$, $w(t,\cdot) \in \F_{\mbox{\tiny{c}}}(B)$ for a.e.~$t \in I^-_{\sigma}$, and $w$ satisfies \eqref{eq:HoelderSobolev integrated}. 

Recall that $d\Gamma(\psi,\psi) \leq 4(\omega r)^{-2} d\mu$.
Thus,
\begin{align}  \label{eq:5.2.14} 
&        \int \int_{Q^-_{\sigma'}} u^{p q \theta} d\bar\mu
=  \int \int_{Q^-_{\sigma'}} (\psi u^{pq/2})^{2\theta} d\bar\mu    \nonumber \\
\leq &  E(\sigma B) \left( \int_{I^-_{\sigma'}} \int_{\sigma B} d\Gamma(\psi u^{pq/2},\psi u^{pq/2}) dt + r^{-2} \int_{I^-_{\sigma'}} \int_B \psi^2 u^{pq} d\bar\mu \right) \nonumber \\
& \left( \sup_{t \in I^-_{\sigma'}} \int_{\sigma B} \psi^2 u^{pq} d\mu \right)^{2/\nu} \nonumber \\
\leq &  E(\sigma B) \left( \frac{C q^2}{(\omega r)^2} [ (A_1+ A_2' r^2)(p^2+1)] \int \int_{Q^-_{\sigma}} u^{pq} d\bar\mu \right)^{\theta},
\end{align} 
for some constant $C>0$. Observe the different roles of $p$ (which is fixed) and $q$ (which will be absorbed in the constant $A^{i+1}$ below). 

Set $\omega_i = (1 - \delta) 2^{-i}$ so that $\sum_{i=1}^{\infty} \omega_i = 1 - \delta$. Set also $\sigma_0 = 1$, $\sigma_{i+1} = \sigma_i - \omega_i = 1 - \sum_{j=1}^i \omega_j$. 
Applying \eqref{eq:5.2.14} with $q = q_i = \theta^i$, $\sigma = \sigma_i$, $\sigma' = \sigma_{i+1}$, we obtain
\begin{align*} 
& \int \int_{Q^-_{\sigma_{i+1}}} u^{p \theta^{i+1}} d\bar\mu \\
\leq &  E(B) \left( A^{i+1} [ (1 - \delta) r]^{-2} [ (A_1+ A_2' r^2)(p^2+1) ] \int \int_{Q^-_{\sigma_i}} u^{p \theta^i} d\bar\mu \right)^{\theta},
\end{align*}
where the constant $A$ depends on $\theta$.
Hence,
\begin{align*}
& \left( \int \int_{Q^-_{\sigma_{i+1}}} u^{p \theta^{i+1}} d\bar\mu \right)^{ \theta^{-1-i} } \\
& \leq    E(B)^{ \sum \theta^{-1-j} } A^{ \sum (j+1) \theta^{-j} }  
    \big( [ (1 - \delta) r]^{-2} [  (A_1+ A_2' r^2)(p^2+1) ] \big)^{ \sum \theta^{-j} } \! \int \! \int_{Q} u^p d\bar\mu,
\end{align*}
where all the summations are taken from $0$ to $i$. Letting $i$ tend to infinity, we obtain
 \[ \sup_{Q^-_{\delta}} \{ u^p \}  \leq  E(B)^{\nu/2} \left( A [ (1-\delta) r]^{-2-\nu} [(A_1+ A_2' r^2)(p^2+1)]^{1+\nu/2} \right) \Vert u \Vert^p_{p,Q}. \]
As $E(B) = S_Y V(x,r)^{-2/\nu} r^2$, this yields \eqref{eq:MVE subsol p>2}.

At this stage of the proof, Theorem \ref{thm:u loc bounded} below already follows. Thus, in the case $1+\eta < p < 2$ the assertion can be proved similarly, by using Lemma \ref{lem:estimate subsol p<2} and Theorem \ref{thm:u loc bounded}.
\end{proof}

\begin{remark}
For $0 < p < 1+\eta$, an estimate similar to \eqref{eq:MVE subsol p>2} for subsolutions holds true with constants depending on $p$. See \cite[Theorem 2.2.3, Theorem 5.2.9]{SC02}. 
\end{remark}

\begin{theorem} \textbf{\em (Local boundedness of weak solutions)} \label{thm:u loc bounded}
Suppose {\em Assumptions \ref{as:e_t}} and {\em \ref{as:sobolev}} are satisfied.
Any non-negative local very weak subsolution $u$ of $\frac{\partial}{\partial t}u = L_t u$ on $Y$ is locally bounded.
In particular, any local very weak solution of $u$ of $\frac{\partial}{\partial t}u = L_t u$ on $Y$ is locally bounded.
\end{theorem}

\begin{proof}
Follows from the proof of Theorem \ref{thm:MVE subsol p>2} together with Proposition \ref{prop:|u| is subsolution}.
\end{proof}

\begin{theorem} \textbf{\em (Mean value estimate with $p<0$)}
Suppose {\em Assumptions \ref{as:e_t}} and {\em \ref{as:sobolev}} are satisfied.
Let $p \in (-\infty, 0)$. Fix a ball $B = B(x,r)$ with $B(x,2r) \subset Y$ and let $\tau >0$.
Then there exists a constant $A=A(\tau, \nu, C_1 \mbox{ - } C_5)$ such that, for any real $a$, any $0 < \delta < \delta' \leq 1$, and any non-negative locally bounded
local very weak supersolution $u$ of the heat equation for $L_t$ in $Q = Q(\tau,x,a,r)$, we have
\begin{align*}
 \sup_{Q^-_{\delta}} \{ u_{\varepsilon}^p \}  
\leq  \frac{ A \, S_Y^{\nu/2} [(A_1+ A_2' \tau r^2)(p^2+1)]^{1+\nu/2} }{ (\delta' - \delta)^{2+\nu} r^2 V(x,r) } 
      \int_{Q^-_{\delta'}} u_{\varepsilon}^p d\bar \mu.
\end{align*}
\end{theorem}
The above theorem can be proved analogously to the proof of Theorem \ref{thm:MVE subsol p>2}, by applying Lemma \ref{lem:estimate supsol} instead of Theorem \ref{thm:estimate subsol p>2}.

\begin{theorem} \textbf{\em (Mean value comparison)} \label{thm:5.2.17}
Suppose {\em Assumptions \ref{as:e_t}} and {\em \ref{as:sobolev}} are satisfied.
Fix a ball $B = B(x,r)$ with $B(x,2r) \subset Y$ and let $\tau > 0$. Set $\theta = 1 + 2/\nu$ and fix $p_0 \in (0,\theta)$. Then there exists a 
constant $A=A(p_0, \tau, \nu, C_1 \mbox{ - } C_5)$ such that for any real $a$, $0 < \delta < \delta' \leq 1$, $0 < p < p_0 / \theta$, and 
any non-negative locally bounded local very weak supersolution $u$ of the heat equation for $L_t$ in $Q = Q(\tau,x,a,r)$, we have
\begin{align*}
 \left( \int_{Q^+_{\delta}} u_{\varepsilon}^{p_0} d\bar\mu \right)^{\frac{1}{p_0}}
\leq  \left[ \frac{ A \, S_Y^{\nu / 2}  [(A_1+ A_2' \tau r^2)(p_0^2+1)]^{1 + \nu / 2} }{ (\delta' - \delta)^{2+\nu} r^{2} \mu(B) } \right]^{\frac{1}{p} - \frac{1}{p_0}}
    \left( \int_{Q^+_{\delta'}} u_{\varepsilon}^p d\bar \mu  \right)^{\frac{1}{p}}.
\end{align*}
\end{theorem}

\begin{proof} We follow \cite[Theorem 5.2.17]{SC02}.
For simplicity, we assume for the proof that $\tau = \delta' = 1$.
Let $0 < p < p_0 / \theta$. By Lemma \ref{lem:estimate supsol}, we have
\begin{align*}
&        \sup_{I^+_{\sigma'}} \left\{ \int_B u_{\varepsilon}^p \psi^2 d\mu \right\} 
           + a_1 \int_{Q^+_{\sigma}} \psi^2 d\Gamma(u_{\varepsilon}^{p/2},u_{\varepsilon}^{p/2}) dt \\
\leq &  A_1(p^2+1) \int_{Q^+_{\sigma}} u_{\varepsilon}^p d\Gamma(\psi,\psi) dt 
       + A_2'(p^2+1) \int_{Q^+_{\sigma}} u_{\varepsilon}^p \psi^2 d\bar\mu.
\end{align*}
Similar to the proof of Theorem \ref{thm:MVE subsol p>2} but with the cylinders $Q_{\sigma}^-$ replaced by $Q_{\sigma}^+$, we find that for any $0 < \beta < p_0 / \theta < 1$,
\begin{equation}  \label{eq:5.2.14'}
 \int_{Q^+_{\sigma'}} u_{\varepsilon}^{\beta \theta} d\bar\mu
    \leq  E(B) \left[ \frac{A}{(\omega r)^2} [(A_1+ A_2' r^2)(\beta^2+1)] \int_{Q^+_{\sigma}} u_{\varepsilon}^{\beta} d\bar\mu \right]^{\theta},
\end{equation}
where $E(B) = S_Y r^2 V(x,r)^{-2/\nu}$.

Define $p_i = p_0 \theta^{-i}$. We first prove the claim for these values of $p$, using the same iteration as in the proof of \cite[Theorem 2.2.5]{SC02}. Set $\sigma_0 = 1$ and $\sigma_{l-1} - \sigma_{l} = 2^{-l}(1 - \delta)$.
Fix $i \geq 1$, and apply \eqref{eq:5.2.14'} with $\beta = p_i \theta^{j-1}$, $j = 1,2,\ldots,i$, $\sigma = \sigma_{i-1}$, $\sigma' = \sigma_i$. This yields for all $j = 1,\ldots,i$ (note that $A$ may change from line to line), 
\[  \int_{Q^+_{\sigma_i}} u_{\varepsilon}^{p_i \theta^j} d\bar\mu
    \leq  E(B) \left[ A^j [ (1 - \delta) r]^{-2} [(A_1+ A_2' r^2)(p_i^2+1)]  \int_{ Q^+_{\sigma_{i-1}} } u_{\varepsilon}^{p_i \theta^{j-1}} d\bar\mu \right]^{\theta}. \]
Hence,
\[  \int_{Q^+_{\sigma_i}} u_{\varepsilon}^{p_0} d\bar\mu
    \leq  C   \left( \int_{Q} u_{\varepsilon}^{p_i} d\bar\mu \right)^{\theta^i}, \]
where
\[ C = E(B)^{\sum_{j=0}^{i-1} \theta^j} A^{\sum_{j=0}^{i-1} (i-j) \theta^{j+1} }  
          \big( [ (1 - \delta) r]^{-2} [(A_1+ A_2' r^2)(p_0^2+1)] \big)^{ \sum_{j=0}^{i-1} \theta^{j+1} }. \]
Observe that
 \[  \sum_{j=0}^{i-1} \theta^j 
     =  \frac{ \theta^i - 1 }{ \theta - 1}  =  (\nu/2) (p_0/p_i - 1),  \]
\begin{align*}
\sum_{j=0}^{i-1} (i-j) \theta^{j} 
   &  =  \sum_{j=0}^{i-1} \sum_{k=0}^j \theta^k
      =  \sum_{j=0}^{i-1} \frac{ \theta^{j+1} - 1 }{\theta - 1}
     \leq  \theta  \frac{\theta^i - 1}{(\theta -1)^2} \\
   & \leq  \left( \theta / (\theta - 1) \right)^2 \theta^{i-1}  
     \leq  \left( \theta / (\theta - 1) \right)^2(\nu/2) (\theta^{i} - 1) \\
   & \leq  \left( \theta / (\theta - 1) \right)^2(\nu/2) (p_0 / p_i - 1),
\end{align*}
and
 \[  \sigma_i = 1 - \left( \sum_{j=1}^i 2^{-j} \right) (1 - \delta) > \delta. \]
Thus,
\[ \left( \int_{Q^+_{\delta}} u_{\varepsilon}^{p_0} d\bar\mu \right)^{\frac{1}{p_0}}
    \leq  C' \left( \int_{Q} u_{\varepsilon}^{p_i} d\bar\mu \right)^{\frac{1}{p_i}}, \]
where
\[  C' = \left( E(B) A \big( [ (1 - \delta) r]^{-2} [(A_1+ A_2' r^2)(p_0^2+1)] \big)^{ \theta } \right)^{\frac{\nu}{2} \left(\frac{1}{p_i} - \frac{1}{p_0}\right)} \]
To obtain the inequality for any $p \in (0,p_0/\theta)$, see \cite[Theorem 2.2.5]{SC02}. 
\end{proof}

\begin{theorem} \textbf{\em (Weighted Poincar\'e inequality)} \label{thm:weighted PI}
Under Assumption \ref{as:VD+PI}, $(\e,\F)$ satisfies the \emph{weighted Poincar\'e inequality} on $Y$. That is, there exists a constant $C \in (0,\infty)$ such that for any ball $B(x,r) \subset Y$,
\begin{align} \label{eq:weighted PI}
\forall f \in \F, \  \int |f - f_B|^2 \psi^2 d\mu  \leq  C_{\mbox{\tiny{\emph{wPI}}}} \, r^2 \int \psi^2 d\Gamma(f,f),
\end{align}
where $f_B = \int f \psi^2 d\mu \big/ \int \psi^2 d\mu$ is the weighted mean of $f$ over $B(x,r)$, and $\psi(z) = \max\{0,1 - d(x,z)/r \}$.
\end{theorem}

\begin{proof}
This is proved in \cite[Corollary 2.5]{SturmIII}.
\end{proof}

The next lemma is a main lemma. This is the first time that we apply Assumption \ref{as:p=0}.
To keep notation short, we introduce a new constant
\[ C_8 := C_2 + C_3^{1/2} + C_5 + C_7, \]
where $C_2, C_3, C_5, C_7$ are the constants of Assumptions \ref{as:e_t} and \ref{as:p=0}.
Recall that for $\varepsilon \in (0,1)$ we set $u_{\varepsilon}:= u + \varepsilon$.
For $\tau>0$, $a \in \R$, $r>0$, set
\[ I:= (a - \tau r^2, a+ \tau r^2). \]

\begin{lemma}\label{lem:5.4.1} 
Suppose that Assumptions \ref{as:e_t}, \ref{as:p=0} and \ref{as:VD+PI} are satisfied.
Let $\tau > 0$ and $\delta, \eta \in (0,1)$. For any real $a$, any $B=B(x,r)$ with $B(x,2r) \subset Y$, 
any non-negative locally bounded local very weak supersolution $u  \in \mathcal C(\overline{I} \to L^2(X,\mu))$ of the heat equation for $L_t$ in $Q = Q(\tau,x,a,r)$,
 there is a constant $c = c(u,\delta,\eta)$,  and a constant $C \in (0,\infty)$ depending on $C_1$ - $C_7$, $C_{\mbox{\tiny{VD}}}$, $C_{\mbox{\tiny{PI}}}$, such that for all $\lambda > 0$, 
 \[  \bar\mu( \{ (t,z) \in K_+ : \log u_{\varepsilon} < - \lambda - c \} ) \leq C (1+C_8 r^2) r^2 \mu(B) \lambda^{-1} \]
and
 \[  \bar\mu( \{ (t,z) \in K_- : \log u_{\varepsilon} > \lambda - c \} ) \leq C (1+ C_8 r^2) r^2 \mu(B) \lambda^{-1}, \]
where $K_+ = (a, a + \eta \tau r^2) \times \delta B$ and $K_- = (a - (1-\eta) \tau r^2,a) \times \delta B$. 
The constant $C$ does not depend on $u$, $\lambda$, $a$, $x$, or $r$.
\end{lemma}

\subsection{Parabolic Harnack inequality} \label{ssec:PHI}

The next theorem is one of the main results of this paper.
Let $\delta \in (0,1)$, $\tau>0$. For $B=B(x,r) \subset X$ and $a \in \R$, set
\begin{align*}
\delta B & = B(x,\delta r), \\
Q &= (a-\tau r^2, a+ \tau r^2) \times B, \\
 Q_- &= (a - \delta \tau r^2, a - \frac{\delta}{2} \tau r^2) \times \delta B, \\
 Q'_- &= (a - \tau r^2, a- \frac{\delta}{2} \tau r^2) \times \delta B, \\ 
 Q_+ &= (a + \frac{\delta}{2} \tau r^2, a + \delta \tau r^2) \times \delta B.
\end{align*}

%

\begin{theorem}\textbf{\em (Parabolic Harnack inequality)} \label{thm:local VD+PI = local HI}
Suppose {\em Assumptions \ref{as:e_t}, \ref{as:p=0}} and {\em \ref{as:VD+PI}} 
are satisfied. Then the family $(\e_t,\F)$ satisfies the parabolic Harnack inequality on $Y$. 
That is, there is a constant $H_Y=H_Y(\tau, \delta, D_Y, P_Y, C_1 - C_7, (C_2+C_3^{1/2}+C_5+C_7)r^2)$ such that for any $a \in \R$, $B(x,r)$ with $B(x,2r) \subset Y$, 
any non-negative local very weak solution $u \in \mathcal C(\overline{I} \to L^2(X,\mu))$ of the heat equation for $L_t$ in $Q = Q(\tau,x,a,r)$, we have
\[ \sup_{Q_-} u \leq H_Y \inf_{Q_+} u. \]
In particular, each of the bilinear forms $(\e_t,\F)$, $t \in \R$, satisfies the elliptic Harnack inequality on $Y$.
\end{theorem}

\begin{proof}
The parabolic Harnack inequality follows by applying a Bombieri type lemma \cite[Lemma 5.2]{LierlPHIf} to $u_{\varepsilon}$ and to $u_{\varepsilon}^{-1}$. 
See \cite{LierlPHIf} for details. The hypothesis of the Bombieri type lemma are satisfied due to Theorem \ref{thm:u loc bounded}, Theorem \ref{thm:5.2.17}, and Lemma \ref{lem:5.4.1}.
\end{proof}

We remark that in the parabolic Harnack inequality of Theorem \ref{thm:local VD+PI = local HI}, the Harnack constant $H_Y$ depends on the radius $r$. In fact, $H_Y$ depends only on an upper bound $R_0$ for $r$. Therefore, the parabolic Harnack inequality is locally scale-invariant.  

As $r \to \infty$, the Harnack constant blows up. This dependence of $H_Y$ on an upper bound on $r$ is natural and expected. It is caused by the "lower order" terms in the bilinear forms $\e_t$. In the special case when each $\e_t$ is a pure second order perturbation of the reference form $(\e,\F)$ we have the following globally scale-invariant parabolic Harnack inequality. It is a special case of Theorem \ref{thm:local VD+PI = local HI}.

\begin{corollary} \textbf{\em (Global parabolic Harnack inequality)} \label{cor:global PHI}
Suppose {\em Assumptions \ref{as:e_t}, \ref{as:p=0}} and {\em \ref{as:VD+PI}} are satisfied globally on $Y=X$ with $C_2=C_3=C_5=C_7=0$. 
Then the family $(\e_t,\F)$ satisfies a scale-invariant parabolic Harnack inequality on $X$. That is, there is a constant $H$ such that for any $a \in \R$,
 $B(x,r) \subset X$, and 
any non-negative local very weak solution $u \in \mathcal C(\overline{I} \to L^2(X,\mu))$ of the heat equation for $L_t$ in $Q = Q(\tau,x,a,r)$, we have
\[ \sup_{Q_-} u \leq H \inf_{Q_+} u. \]
The constant $H$ depends only on $\tau$, $\delta$, $D_X$, $P_X$, $C_1$, $C_4$, $C_6$.
\end{corollary}
A standard consequence of Corollary \ref{cor:global PHI} is the strong Liouville property.


In the special case of a time-independent symmetric strongly local regular Dirichlet form, we obtain the following characterization of the parabolic Harnack inequality.

\begin{theorem}\textbf{\em (Characterization of the parabolic Harnack inequality in the symmetric strongly local case)}
Let $(\e,\F)$ be a symmetric strongly local regular Dirichlet form on $L^2(X,m)$. Suppose $(\e,\F)$ satisfies (A1), (A2).
Then the following are equivalent:
\begin{enumerate}
\item 
$(\e,\F)$ satisfies the volume doubling property and the Poincar\'e inequality on $X$.
\item
$(\e,\F)$ satisfies the parabolic Harnack inequality. That is, for any parameters $\delta \in (0,1)$, $\tau \in (0,1]$, there is a constant $H \in (0,\infty)$ such that for any $a \in \R$,  $B(x,r) \subset X$, and 
any non-negative local very weak solution $u \in \mathcal C(\overline{I} \to L^2(X,\mu))$ of the heat equation for $L_t$ in $Q = Q(\tau,x,a,r)$, we have
\[ \sup_{Q_-} u \leq H \inf_{Q_+} u. \]
\item
$(\e,\F)$ admits a jointly continuous heat kernel which satisfies the two-sided Gaussian bounds
\begin{align*}
\frac{c_1}{V(x,\sqrt{t-s})} \exp\left( - C_2 \frac{d(x,y)^2}{t-s} \right)
\le 
p(t,x,y)
\le  \frac{C_2}{V(x,\sqrt{t-s})} \exp\left( - c_2 \frac{d(x,y)^2}{t-s} \right)
\end{align*}
for any $x,y \in X$, $t \ge 0$. 
\end{enumerate}
\end{theorem}

\begin{proof}
(i) $\Rightarrow$ (ii) is a special case of Theorem \ref{thm:local VD+PI = local HI}. It was pointed out in \cite{SturmIII} that for the reverse implication (ii) $\Rightarrow$ (i), it suffices to follow \cite{SC92} line by line.
For the implication (i) $\Rightarrow$ (iii) see \cite{SturmIII}. We recover this result in Corollary \ref{cor:global HKE} below.
The implication (iii) $\Rightarrow$ (i) follows from combining \cite[Lemma 5.1]{BGK12} and \cite[Theorem 3.2]{LierlBHPf}.
\end{proof}

\begin{remark}
The equivalence between (i) and (ii) was stated already in \cite[Theorem 3.5]{SturmIII}. Since the implication from volume doubling and Poincar\'e inequality to the parabolic Harnack inequality is the central topic of this paper, we must remark that for this implication \cite{SturmIII} refers to and relies on the reasoning in \cite{SturmII}. We discuss this in detail in Section \ref{ssec:Sturm}.
\end{remark}

The next theorem is another special case of Theorem \ref{thm:local VD+PI = local HI}.

\begin{theorem} \textbf{\em (Stability of parabolic Harnack inequality)}
Let $(\e_t,\F)$, $t \in \R$, be a family of symmetric strongly local regular Dirichlet forms on $L^2(X,\mu)$. Suppose there exists a constant $C \in (0,\infty)$ such that
\begin{align*}
C^{-1} \e_t(f,f) \le \e(f,f) \le C \e_t(f,f), \quad \forall f \in \F, \forall t \in \R.
\end{align*}
Suppose $(\e,\F)$ satisfies (A1), (A2). If $(\e,\F)$ satisfies the parabolic Harnack inequality, then the family $(\e_t,\F)$ satisfies the parabolic Harnack inequality.
\end{theorem}

\begin{corollary}\textbf{\em (H\"older continuity)} \label{cor:Hoelder}
Suppose {\em Assumptions \ref{as:e_t}, \ref{as:p=0}} and {\em \ref{as:VD+PI}}
 are satisfied. Suppose in addition that $\e_t(1,\phi) = 0$ for any $\phi \in \F_{\mbox{\em \tiny{c}}}(X)$.
Fix $\tau > 0$ and $\delta \in (0,1)$. Then there exist $\beta \in (0,1)$ and $H \in (0,\infty)$ such that for any $B(x,2r) \subset Y$, any real $a$, 
any local very weak solution $u \in \mathcal C(\overline{I} \to L^2(X,\mu))$ of the heat equation for $L_t$ in $Q=(a-\tau r^2,a+\tau r^2)\times B(x,r)$ has a continuous representative and satisfies
 \[ \sup_{(t,y),(t',y')\in Q'} \left\{ \frac{ |u(t,y) - u(t',y')| }{ [ |t-t'|^{1/2} + d(y,y')^{\beta} ] } \right \}
\leq \frac{H}{r^{\beta} } \sup_{Q} |u| \]
where $Q'= (a-\delta \tau r^2,a+\delta \tau r^2) \times B(x,\delta r)$.
The constant $H$ depends only on  $\tau$, $\delta$, $D_Y$, $P_Y$, $C_1$-$C_7$ and an upper bound on $(C_2+C_3^{1/2}+C_5+C_7)r^2$.
\end{corollary}

\begin{proof}
The H\"older continuity follows from the parabolic Harnack inequality by a well-known argument. See, e.g., \cite{SC02}.
\end{proof}

\section{Estimates for time-dependent heat kernels} \label{sec:applications}

In this section, we use the notion of local weak solutions whose definition we recall in Section \ref{ssec:weak solutions}.
Let $(\e_t,\F)$, $t \in \R$, be a family of bilinear forms such that, for any $f,g \in \F$, the map $t \mapsto \e_t(f,g)$ is measurable, and such that $(\e_t,\F)$, $t \in \R$ satisfy Assumptions 0, \ref{as:e_t} and \ref{as:p=0} uniformly in $t$.

Recall from Remark \ref{rem:assumptions}(iv) that 
there exist constants $\alpha,c \in (0,\infty)$, depending on $C_1,C_2,C_3$, such that
\[ \e_t(f,f) + \alpha \int f^2 d\mu \ge c \| f \|_{\F}^2, \]
for all $t\in\R$ and all $f \in \F$. When $C_2=C_3=0$, then we may choose $\alpha = c = 0$.

\begin{proposition}\textbf{\em (\cite[Chap.~3, Theorem 4.1 and Remark 4.3]{LM72}.)}
For every $f \in L^2(X,\mu)$ there exists a unique weak solution $u$ of the initial value problem
\begin{align} 
\begin{split} \label{eq:initial value problem}
& \frac{\partial}{\partial t} u = L_t u  \quad \textrm{ on } (s,\infty) \times X, \\
& u(s,\cdot) = f \quad \textrm{ on } X. 
\end{split}
\end{align}
\end{proposition}


For any $s \leq t$ there exists a unique \emph{transition operator}
\[ T^s_t:L^2(X,\mu) \to L^2(X,\mu) \]
associated with $L_t - \frac{\partial}{\partial t}$ such that for every $f \in L^2(X,\mu)$ the unique solution $u$ of \eqref{eq:initial value problem} is 
given by $u:t \mapsto T^s_t f$. See, e.g., \cite[Section 1.3 - 1.4 and 2.4]{SturmII} and \cite{vCas11, NS91}. The map $t \mapsto T^s_t$ is strongly continuous on $[s,\infty)$. Furthermore, $\Vert T^s_t \Vert_{2 \to 2} \leq e^{\alpha(t-s)}$, and $T^r_t = T^s_t \circ T^r_s$ for any $r \leq s \leq t$. 

Similarly, there exists a transition operator $(S^*)^s_t$ for the time-reversed initial value problem $- \frac{\partial}{\partial t} v= L^*_t v$ on $(-\infty,s) \times X$, $v(s,\cdot) = f$ on $X$, where $L^*_t$ is the adjoint of $L_t$.
The transition operators $(T^s_t)$ and $(S^*)^s_t$ preserve positivity.

\begin{proposition} \textbf{\em (Existence of the heat propagator)} \label{prop:T^s_t}
Suppose $(\e,\F)$ satisfies (A1) and for every $a \in X$ there exists $Y = Y_a = B(a,2r_a)$ where $(\e,\F)$ satisfies (A2-$Y$), volume doubling and Poincar\'e inequality.
Then there exists a measurable positive 
function $p:\R \times X \times \R \times X \to [0,\infty)$ with the following properties:
\begin{enumerate}
\item
For every $t>s$, $\mu$-a.e. $x,y \in X$ and every $f \in L^1(X,\mu) + L^{\infty}(X,\mu)$,
 \[ T^s_t f(y) = \int_X p(t,y,s,z) f(z) \mu(dz) \]
and
 \[ (S^*)^t_s f(x) = \int_X p(t,z,s,x) f(z) \mu(dz). \]
\item
For every $s < \sigma < \tau$ and $\mu$-a.e. $x \in X$ the function
 \[ u: (t,y) \mapsto p(t,y,s,x) \]
is a global solution of the equation $L_t u = \frac{\partial}{\partial t} u$ on $(\sigma,\tau) \times X$ and for every $t > \tau > \sigma$ and $\mu$-a.e. $y \in X$ the function
\[ u: (s,x) \mapsto p(t,y,s,x) \]
is a global solution of the equation $L^*_s u = - \frac{\partial}{\partial s} u$ on $(\sigma,\tau) \times X$.
\item
For every $s<r<t$ and $\mu$-a.e. $x,y \in X$,
 \[ p(t,y,s,x) = \int_X p(t,y,r,z) p(r,z,s,x) \mu(dz). \]
\end{enumerate}
\end{proposition}

\begin{proof}
In the special case when each $\e_t$ is a Dirichlet form, the proposition is proved in \cite[Proposition 2.3]{SturmII}. See also \cite[Lemma 3.7]{GH} for a proof in the time-independent case. 
For the general case, see the \cite[Proposition 6.3]{LierlPHIf}.
\end{proof}
 
Let
\[ C_9 := C_2 + C_3^{1/2}  + C_5. \] 

\begin{lemma} \label{lem:for davies gaffney}
For any $t \in \R$ and any $f,g \in \F \cap L^{\infty}(X,\mu)$ with $d\Gamma(g,g) \leq \beta^2 g^2 d\mu$,
\begin{align*}
- \e_t(f,fg^2) 
& \leq 
K' (\beta^2 (1+C_4) + C_9) \int f^2 g^2 d\mu,
\end{align*}
where $K' \in (0,\infty)$ depends only on $C_1$.
\end{lemma}

\begin{proof}
By Lemma \ref{lem:SUP sym2} and Assumption \ref{as:e_t}, we have for any $k,k_1,k_2,k_3 > 0$,
\begin{align*}
- \e(f,fg^2)
& \leq 
\left( 4 k C_1 + k_2 C_4 + k_3 \right) \int f^2 d\Gamma(g,g) \\
& \quad
- \left( \frac{1}{C_1} \left( 1 - \frac{1}{k} \right) - \frac{2 }{k_1} - \frac{2 }{k_2} \right) \int g^2 d\Gamma(f,f) \\
& \quad
+ \left( 2 C_3^{1/2} + 2 k_1 C_2 + \frac{2 C_5}{k_3} \right) \int f^2 g^2 d\mu.
\end{align*}
Choose $k$, $k_1$, $k_2$ so that $\frac{1}{C_1} \left( 1 - \frac{1}{k} \right) - \frac{2 }{k_1} - \frac{2 }{k_2} = 0$. Then the assertion follows.
\end{proof}

\begin{theorem}  \textbf{\em (Upper bound)}
\label{thm:upper HKE}
Suppose $(\e,\F)$ satisfies (A1), (A2), and for every $a \in X$ there exists $Y = Y_a = B(a,2r_a)$ where $(\e,\F)$ satisfies volume doubling and Poincar\'e inequality.
Then for all $t>s$ and $x,y \in X$,
\begin{align*}
p(t,y,s,x)
\leq  C \frac{ \exp\left(- c \frac{d(x,y)^2}{t-s} + K'C_9(t-s) \right)}{V(x,\tau_x)^{\frac{1}{2}} V(y,\tau_y)^{\frac{1}{2}}},
\end{align*}
where $\tau_x = \sqrt{t-s} \wedge r_x$, $\tau_y = \sqrt{t-s} \wedge r_y$. The constant $C > 0$ depends only $C_1$--$C_5$, on an upper bound on $C_9(\tau_x^2 + \tau_y^2)$, and on $D_Y$, $P_Y$ for $Y=Y_x$ and for $Y=Y_y$. $c$ depends only on $C_1$ and $C_4$. $K'$ depends only on $C_1$.
\end{theorem}

\begin{proof}
We follow \cite[Theorem 2.4]{SturmII}.
Let $\psi \in \F_{\mbox{\tiny{c}}} \cap L^{\infty}$ with $d\Gamma(\psi,\psi) \leq d\mu$. 
By \eqref{eq:chain rule for Gamma}, $e^{\psi} \in \F_{\mbox{\tiny{loc}}} \cap L^{\infty}(X,\mu)$. 

Let $\beta,s,t \in \R$ and $f \in L^2(X,\mu) \cap L^{\infty}(X.\mu)$. Note that $T^s_r ((e^{-\beta \psi}f) \in \F \cap L^{\infty}$ for all $r \in [s,t]$.
We have
\begin{align*}
\Vert e^{\beta \psi} T^s_t(e^{-\beta \psi} f) \Vert_2^2 - \Vert f \Vert_2^2
& = \int_s^t \frac{d}{dr} \Vert e^{\beta \psi} T^s_r(e^{-\beta \psi} f) \Vert_2^2 dr \\
& = 2 \int_s^t \int  \frac{\partial}{\partial r} \left( T^s_r(e^{-\beta \psi} f) \right) e^{2\beta \psi} T^s_r(e^{-\beta \psi} f) d\mu \, dr \\
& = 
- 2 \int_s^t \e_r\big( T^s_r(e^{-\beta \psi} f),e^{2\beta \psi}T^s_r(e^{-\beta \psi} f) \big) dr \\
& \leq 
K'(\beta^2 (1+C_4) + C_9)
\int_s^t \Vert e^{\beta \psi}T^s_r(e^{-\beta \psi} f) \Vert_2^2 dr,
\end{align*}
by Lemma \ref{lem:for davies gaffney}, where $K' \in (0,\infty)$ depends only on $C_1$.
An application of Gronwall's lemma yields the Davies-Gaffney estimate
\begin{align*}
\Vert e^{\beta \psi} T^s_t(e^{-\beta \psi} f) \Vert_2^2
\leq 
e^{K'(\beta^2 (1+C_4) + C_9) (t-s)}
\Vert f \Vert_2^2.
\end{align*}
Now the assertion follows by repeating \cite[Proof of Theorem 2.4]{SturmII} with $K = K' (1+C_4)$, $k\lambda = - K'C_9$ and $(1+|k \lambda| r_i^2) = (A_1 + A_2 C_9 r_i^2)$, with $A_1$, $A_2$ as in Section \ref{ssec:Mean value estimates}. 
\end{proof}

\begin{definition}
For a ball $B=B(a,r) \subset X$, the Dirichlet-type forms on $B$ are defined as
\[ \e^D_{B,t}(f,g) = \e_t(f,g),  \quad f,g \in D(\e^D_B), \]
where the domain $D(\e^D_B) = \F^0(B)$ is defined as the closure of $\F \cap \mathcal C_{\mbox{\em \tiny{c}}}(B)$ in $\F$ for the norm $\| \cdot \|_{\F}$.
Let $T^D_B(t,s)$ be the associated transition operator and $p^D_B(t,y,s,x)$ the Dirichlet propagator.
\end{definition}

\begin{theorem}\textbf{\em (Estimates for the Dirichlet heat propagator)} \label{thm:basic p^D_B estimate} 
Suppose $(\e,\F)$ satisfies  {\em(A1), (A2)}, volume doubling and Poincar\'e inequality on $Y = B(a,2r_a)$ for some $a \in X$, $r_a >0$.
Let $B=B(a,r_a/2)$.
\begin{enumerate}
\item
For any fixed $\epsilon \in (0,1)$ there are constants $c', C' \in (0,\infty)$ such that for any $x,y \in B(a,(1-\epsilon)r_a/2)$ and $0 < \epsilon (t-s) \leq (r_a/2)^2$, the Dirichlet propagator $p^D_B$ is bounded below by
 \[  p^D_B(t,y,s,x) \geq \frac{c'}{V(x,\sqrt{t-s} \wedge R_x)} \exp\left( - C' \frac{d(x,y)^2}{t-s} \right), \]
where $R_x = d(x,\partial B)/2$.
\item 
There exist constants $c, C \in (0,\infty)$ such that for any $x,y \in B$, $t > s$, the Dirichlet propagator $p^D_B$ is bounded above by
\begin{equation} \label{eq:iii}
 p^D_B(t,y,s,x) 
\leq C \frac{\exp \left( - c \frac{d(x,y)^2}{(t-s)} + K' C_9 (t-s) \right)}
            {V(x,\sqrt{t-s} \wedge (r_a/2))^{1/2} V(y,\sqrt{t-s} \wedge (r_a/2))^{1/2}}.
\end{equation}
\end{enumerate}
The constants $c',C'$ depend only on $C_1$--$C_7$, $D_Y$, $P_Y$ and on an upper bound on $(C_2+C_3^{1/2} + C_5 + C_7)( (t-s) \wedge r_a^2)$.

The constant $C$ depends on $C_1$--$C_5$, on an upper bound on $C_9((t-s) \wedge r_a^2)$, and on $D_Y$, $P_Y$. The constant $c$ depends only on $C_1$ and $C_4$. $K'$ depends only on $C_1$.
\end{theorem}

\begin{proof} Statement (ii) follows from Theorem \ref{thm:upper HKE} and the set monotonicity of the heat propagator (see \cite[Proposition 6.8]{LierlPHIf}).
To show the on-diagonal estimate in (i) we follow the proof of \cite[Theorem 5.4.10]{SC02}. 
Let $0 < \epsilon (t-s) \leq (r_a/2)^2$ and $x \in B(a,(1-\epsilon)r_a/2)$. Let $r = \sqrt{t-s} \wedge R_x$. Let $\psi$ be a smooth function such that $0 \leq \psi \leq 1$, $\psi=1$ on $B(x,r)$ and $\psi=0$ on $X \setminus B(x,2r)$. Define
\begin{align*}
u(t,y) = \begin{cases} T^D_B(t,s) \psi(y) \quad & \textrm{ if } t > s, \\
                           \psi(y) & \textrm{ if } t \leq s.
         \end{cases}
\end{align*}
One can show that $u$ is a local weak solution of
\[ \hat L_t u = \frac{\partial}{\partial t} u \quad \textrm{ on } Q' = (-\infty,+\infty) \times B(x,r), \]
where
\begin{align*}
 \hat L_t = \begin{cases} L_t \quad & \textrm{ if } t > s, \\
                     L  & \textrm{ if } t \leq s.
       \end{cases}
\end{align*}
Applying the parabolic Harnack inequality of Theorem \ref{thm:local VD+PI = local HI} to $u$ and then to $p^D_B(\cdot,\cdot,s,z)$, we get
\begin{align*}
1 = u(s,x)
& \leq  C' \, u(s+(t-s)/2,x) \\
& =     C' \int  p^D_B(s+(t-s)/2,x,s,z)  \psi(z) \mu(dz) \\
& \leq  C' \int_{B(x,2r)}  p^D_B(s+(t-s)/2,x,s,z)  \mu(dz) \\
& \leq C' \, V(x,2r) \, p^D_B(s+(t-s)/2,x,s-(t-s)/2,x).
\end{align*}
The constant $C'$ changes from line to line.
Using volume doubling, we get
 \[ p^D_B(t,x,s,x) \geq \frac{C'}{V(x,r)}. \] 
For the off-diagonal estimate, see the proof of \cite[Theorem 4.8]{SturmIII}, and apply the parabolic Harnack inequality of Theorem \ref{thm:local VD+PI = local HI}.
\end{proof}

\begin{theorem} \textbf{\em (Lower bound)}
Suppose $(\e,\F)$ satisfies  {\em(A1), (A2-$Y$)}, volume doubling and Poincar\'e inequality on $Y = B(a,2r_a)$ for some $a \in X$, $r_a >0$. 
Then there are constants $c, C, C' >0$ such that for all $x,y \in B(a,r_a/2)$ and $t>s$, we have
\begin{align*}
p(t,y,s,x)
\geq c  \frac{\exp\left( -C \frac{d(x,y)^2}{t-s} - \frac{C'}{r_a^2} (t-s) \right)}{V(x,\sqrt{t-s} \wedge (r_a/2))}.
\end{align*}
The constants $c,C,C'$ depend only on $C_1$--$C_7$, on an upper bound on $(C_2+C_3^{1/2}+C_5+C_7)r_a^2$, and on $D_Y$, $P_Y$.
\end{theorem}

\begin{proof}
From Theorem \ref{thm:basic p^D_B estimate}(i) we obtain an on-diagonal bound for $t-s < r_a^2$.
The off-diagonal estimate (for any $t>s$) follows from the parabolic Harnack inequality.
\end{proof}

The following corollary provides global two-sided bounds for the heat propagator in situations that generalize the following model case in which the heat equation is given by
\[ \frac{\partial}{\partial t} u = \sum_{i,j=1}^n
\frac{\partial}{\partial x_j} \left( a_{i,j}(t,\cdot)\frac{\partial}{\partial x_i} u \right) \]
 on $\R^n$, with bounded 
measurable time-dependent coefficients $(a_{i,j})$ that are uniformly elliptic but not necessarily symmetric. 
\begin{corollary} \textbf{\em (Two-sided global bounds in the strongly local case)} \label{cor:global HKE}
Suppose $C_2=C_3=C_5=C_7=0$. Suppose $(\e,\F)$ satisfies  {\em(A1), (A2)}, volume doubling and Poincar\'e inequality on $X$. Then 
there are constants $c', C', C, c \in (0,\infty)$ such that for any $x,y \in X$ and $t>s$, we have
\begin{align*}
c'  \frac{\exp\left( -C' \frac{d(x,y)^2}{t-s}  \right)}{V(x,\sqrt{t-s})}
\leq p(t,y,s,x)
\leq  C \frac{ \exp\left(- c \frac{d(x,y)^2}{t-s} \right)}{V(x,\sqrt{t-s})}.
\end{align*}
The constants $c', C', C, c$ depend only on $C_1$, $C_4$, $C_6$, $D_X$, $P_X$.\end{corollary}
Note that, under the assumption of Corollary \ref{cor:global HKE},  
Corollary \ref{cor:Hoelder} provides assorted global time-space 
H\"older continuity estimates for the heat kernel.

\begin{remark}
For the sake of simplicity, in the results described above, we have not tried to capture the sharpest possible Gaussian upper bound as far as the constant in front of  $\frac{d(x,y)^2}{(t-s)}$ in the exponential Gaussian factor is concerned. The reason is that this question is rather unnatural and somewhat irrelevant in the present context of time-dependent forms. We note that, with the parabolic Harnack inequality of Theorem \ref{thm:local VD+PI = local HI} established, it is possible to obtain more detailed Gaussian upper bounds in spirit of \cite{SturmII} and \cite[Section 5.2.3]{SC02} by following the line of reasoning used in these references. 
\end{remark}

\section{Proofs} \label{sec:proofs}
\subsection{Proof of Theorem \ref{thm:estimate subsol p>2}}

\subsubsection{Heuristics}

The proof becomes transparent if we first look at the case when $p=2$, $\e$ is symmetric strongly local, 
and $u$ has a weak time-derivative $\frac{\partial u}{\partial t}$ that is represented by a function in $L^2(I \to \F)$.
 Writing the supremum $\left( \sup_{t \in I^-_{\sigma'}} \int u^p \psi^2 d\mu \right)$ as the integral of its time-derivative, we can apply the hypothesis that $u$ is a weak subsolution of the heat equation. We obtain that, for a suitable choice of $s_0 \in I$ so that $\chi(s_0)=0$,
\begin{align*}
& \sup_{t \in I^-_{\sigma'}} \int u^2 \psi^2 d\mu \\
& = \sup_{t_0 \in I^-_{\sigma'}} \int_{s_0}^{t_0} \left[ \int \frac{\partial u}{\partial t} 2 u \psi^2 \chi d\mu + \int u^2 \psi^2 \chi' d\mu \right] \, dt \\
& \leq \sup_{t_0 \in I^-_{\sigma'}} \left[ \int_{s_0}^{t_0} -2 \e_t(u,u\psi^2) dt + \int_{s_0}^{t_0} \int u^2 \psi^2 \chi' d\mu \, dt \right] \\
& \leq -k \sup_{t_0 \in I^-_{\sigma'}} \left[ \int_{s_0}^{t_0} \int \psi^2 d\Gamma(u,u) dt + K \int_{s_0}^{t_0} \int u^2 d\Gamma(\psi,\psi) +  \int_{s_0}^{t_0} \int u^2 \psi^2 \chi' d\mu \, dt \right],
\end{align*}
for some constants $k,K > 0$. We have applied \eqref{eq:CS} and \eqref{eq:chain rule for Gamma} in the last inequality. Rearranging the terms and making use of the particular choice of $\chi$ easily yields the assertion for this special case. 

In the general case $p > 2$, we will approximate $u^{p-2}$ in terms of bounded functions $u_n:=u \wedge n$ to ensure integrability. This approximation is complicated by the fact that we do not know a priori whether
 $u_n$ is a weak subsolution. Another challenge then arises from the application of the chain rule for the weak time derivative in the 
above argument. This problem will be resolved by taking Steklov averages. Finally, to treat the case of non-symmetric forms, we make use 
of Assumption 0, Assumption \ref{as:e_t}, and apply the estimates that we proved in Section \ref{ssec:algebraic computations}.

\subsubsection{The symmetric strongly local case} \label{ssec:proof sym strongly local}
In order to make the proof accessible for a wider audience, we first give the proof of Theorem \ref{thm:estimate subsol p>2} in the symmetric strongly local case (see \eqref{eq:sym strongly local case}). 

\begin{proof}
Let $n$ be a positive integer. We keep $n$ fixed until we let $n \to \infty$ at the very end of the proof.
Let $u_n := u \wedge n$. We will make use of a function $\mathcal H = \mathcal{H}_n$ which was constructed in \cite{AS67} in such a way that $\mathcal{H}_n'(v) = v (v \wedge n)^{p-2}$. More specifically, let $\mathcal H_n: \R  \to \R$,
\begin{align*}
\mathcal{H}_n(v) := \begin{cases} & \frac{1}{p} v^2 (v \wedge n)^{p-2}, \qquad\qquad\qquad\qquad \, \textrm{if } v \leq n, \\
                               & \frac{1}{2} v^2 (v \wedge n)^{p-2} + n^p \left( \frac{1}{p} - \frac{1}{2} \right), \qquad \textrm{if } v > n.
                 \end{cases}
\end{align*} 
We will use the approach of Cipriani and Grillo \cite{CG01}; for a real number $0 < h < (\tau - \tau') r^2$ let
\[  u_h(t) := \frac{1}{h} \int_t^{t+h} u(s) ds, \quad t \in (a- \tau r^2, a + \tau' r^2),\]
be the Steklov average of $u$.
In this proof, the subscript of the Steklov average will always be denoted as $h$ and should not be confused with the integer $n$.

Recall that $u_h \in L^1((a- \tau r^2, a + \tau' r^2) \to \F)$ by Lemma \ref{lem:steklov convergence}. We will write $u_h(t,\cdot)$ for $u_h(t)$. By Lemma \ref{lem:SUP sym2}, $\mathcal{H}_n(u(t,\cdot)), \mathcal{H}_n(u_h(t,\cdot)) \in \F_{\mbox{\tiny{loc}}}$ at almost every $t$. 
The Steklov average $u_h$ has a strong time-derivative
\[ \frac{\partial}{\partial t} u_h(t,x) = \frac{1}{h} \big( u(t+h,x) - u(t,x) \big). \] 

Let $s_0 = a -\frac{1+\sigma}{2}\tau r^2$. For a.e. $t_0 \in I^-_{\sigma'}$ and $J = (s_0, t_0)$, we obtain
\begin{align*}
& \int_X \mathcal H_n( u_h(t_0,\cdot) )  \psi^2 d\mu \\
& =
\int_J \frac{d}{d t} \left( \int_X \mathcal H_n( u_h ) \psi^2 \chi \, d\mu \right) dt \\
& =
\int_J \int_X \frac{\partial}{\partial t} \left( u_h \right) \mathcal H'_n(u_h) \psi^2 \chi \, d\mu \, dt  
+ \int_J \int_X  \mathcal H_n(u_h) \psi^2 \chi' \, d\mu \, dt \\
& = 
\int_J \frac{1}{h} \int_X \big( u(t+h,\cdot) - u(t,\cdot) \big) \mathcal H'_n(u_h(t,\cdot)) \psi^2 \chi \,  d\mu \, dt 
 + \int_J \int_X  \mathcal H_n(u_h) \psi^2 \chi' \, d\mu \, dt.
\end{align*}
Next, we make use of the assumption that $u$ is a local very weak subsolution of the heat equation. In fact, this is the only place in this proof where the heat equation is used. We obtain
\begin{align} \label{eq:steklov subsol estimate lhs}
& \int_X \mathcal H_n( u_h(t_0,\cdot) )  \psi^2 d\mu  \\
& \leq 
- \int_J \frac{1}{h} \int_t^{t+h}\e_s( u(s,\cdot), \mathcal H'_n(u_h(t,\cdot)) \psi^2 ) \chi(t) ds \, dt
+ \int_J \int_X  \mathcal H_n(u_h) \psi^2 \chi' \, d\mu \, dt \nonumber \\
& \leq \label{eq:want to let h to 0 part1}
- \int_J \frac{1}{h} \int_t^{t+h}\e_s( u(s,\cdot), [\mathcal H'_n(u_h(t,\cdot)) - \mathcal H'_n(u(t,\cdot))] \psi^2 ) ds \, \chi(t) dt \\
& \quad \label{eq:want to let h to 0 part2}
- \int_J \frac{1}{h} \int_t^{t+h} \e_s( u(s,\cdot) - u(t,\cdot), \mathcal H'_n(u(t,\cdot)) \psi^2 ) ds \, \chi(t) dt \\
& \quad \label{eq:steklov subsol estimate iii}
- \int_J \frac{1}{h} \int_t^{t+h}\e_s( u(t,\cdot), \mathcal H'_n(u(t,\cdot)) \psi^2 ) ds \, \chi(t) dt \\
& \quad \label{eq:steklov subsol estimate iv}
+ \int_J \int_X  \mathcal H_n(u_h) \psi^2 \chi' \, d\mu \, dt.
\end{align}
We now proceed in three steps. 
First, we estimate the integrand in \eqref{eq:steklov subsol estimate iii}. 
Second, we show that \eqref{eq:want to let h to 0 part1} and \eqref{eq:want to let h to 0 part2} vanish in the limit as $h \to 0$. 
Third, we let $h \to 0$ on both sides of the above inequality while applying the estimate obtained in Step 1. Then we take the supremum over all $t_0 \in I_{\sigma'}^-$ and let $n \to \infty$. This will conclude the proof.

\paragraph{STEP 1.}
We want to estimate the double integral \eqref{eq:steklov subsol estimate iii}. Let us keep $s$ and $t$ fixed for a moment and consider
the integrand. Write $u$ for $u(t,\cdot)$ and $u_n$ for $u_n(t,\cdot)$. 

By Lemma \ref{lem:SUP sym2}, the Cauchy-Schwarz inequality \eqref{eq:CS}, and \eqref{eq:sym strongly local case}, we have for any $k_1 > 0$,
\begin{align*}  
& \quad - \e^{\mbox{\tiny{s}}}_s (u, u u_n^{p-2}  \psi^2) \nonumber \\
& \leq \frac{4k_1}{p-1} \int u^2 u_n^{p-2}  d\Gamma_s(\psi,\psi) 
        - \left( 1 - \frac{p-1}{k_1} \right) \int  u_n^{p-2}  \psi^2  d\Gamma_s(u,u) \nonumber \\
  & \quad  - (p-2) \int u_n^{p-2}  \psi^2   d\Gamma_s(u_n,u_n) \nonumber \\
& \leq \frac{4 k_1 C_1}{p-1} \int u^2 u_n^{p-2}  d\Gamma(\psi,\psi) 
        - C_1^{-1} \left(1 - \frac{p-1}{k_1} \right) \int  u_n^{p-2}  \psi^2  d\Gamma(u,u) \nonumber \\
  & \quad - (p-2) C_1^{-1} \int u_n^{p-2}  \psi^2   d\Gamma(u_n,u_n).
\end{align*}
Rearranging,
\begin{align*} 
& \quad   - \e_s(u,u u_n^{p-2} \psi^2)  
 +  \frac{p-2}{C_1}  \int \psi^2 u_n^{p-2} d\Gamma(u_n,u_n) 
 + \left( \frac{1}{C_1} - \frac{p-1}{C_1 k_1}  \right)  \int u_n^{p-2} \psi^2 d\Gamma(u,u)  \\
& \leq 
 \left( \frac{4k_1 C_1}{p-1}  \right)  \int u^2 u_n^{p-2} d\Gamma(\psi,\psi).
\end{align*}
This inequality holds for any $s \in (t,t+h)$ (recall that we fixed $s$ and $t$ at the beginning of Step 1). Now we take the Steklov average at $t$, and then we integrate over $J$, on both sides of the inequality. This results in the following estimate for \eqref{eq:steklov subsol estimate iii}. 
\begin{align} \label{eq:result of Step 1 sym strongly local}
& - \int_J \frac{1}{h} \int_t^{t+h}\e_s( u(t,\cdot), \mathcal H'_n(u(t,\cdot)) \psi^2 ) ds \, \chi(t) dt \\
& \quad + \frac{p-2}{C_1} \int_{J} \int \chi \psi^2 u_n^{p-2} d\Gamma(u_n,u_n) dt \nonumber\\
& \quad + \left( \frac{1}{C_1} - \frac{p-1}{C_1 k_1}  \right) \int_{J} \int u_n^{p-2} \chi \psi^2 d\Gamma(u,u) dt \nonumber\\
& \le
 \frac{4k_1 C_1}{p-1}  \int_J \int u^2 u_n^{p-2} d\Gamma(\psi,\psi) dt \nonumber.
\end{align}

\paragraph{STEP 2.}
We show that, for some sequence of reals $h$ that tends to $0$, \eqref{eq:want to let h to 0 part1} and \eqref{eq:want to let h to 0 part2} tend to $0$ as $h \to 0$. 
By Assumption 0, Cauchy-Schwarz inequality,  and Corollary \ref{cor:steklov convergence in L^2},
\begin{align*}
& \lim_{h \to 0} \int_J \frac{1}{h} \int_t^{t+h} \e_s( u(s,\cdot) - u(t,\cdot), \mathcal H'_n(u(t,\cdot)) \psi^2 ) ds \, \chi(t)  dt \\
& \leq
\lim_{h \to 0} 
C_* \int_J \frac{1}{h} \int_t^{t+h} || u(s,\cdot) - u(t,\cdot) ||_{\F} \, ds \, ||\mathcal H'_n(u(t,\cdot)) \psi^2  ||_{\F} \, \chi(t)dt \\
& = 0.
\end{align*}
This shows that \eqref{eq:want to let h to 0 part2} goes to $0$ as $h \to 0$.

Next, we show that \eqref{eq:want to let h to 0 part1} goes to $0$ as $h \to 0$.
For the purpose of this proof, let 
\[ ||u(s,\cdot)||_{\F,\sigma B}^2 := \int_{\sigma B} d\Gamma(u(s,\cdot),u(s,\cdot)) + \int_{\sigma B} u(s,\cdot)^2 d\mu. \]
By the locality of $\e_s$, Assumption 0, and Cauchy-Schwarz inequality,
\begin{align} \label{eq:want to let h to 0 part3}
& \int_J \frac{1}{h} \int_t^{t+h} \e_s\big( u(s,\cdot), [\mathcal H'_n(u_h(t,\cdot)) - \mathcal H'_n(u(t,\cdot))]\psi^2 \big) ds \, \chi(t) dt \nonumber \\
& \leq
C_* \int_J \frac{1}{h} \int_t^{t+h} ||u(s,\cdot)||_{\F,\sigma B}  \, || [\mathcal H'_n(u_h(t,\cdot)) - \mathcal H'_n(u(t,\cdot))]\psi^2 ||_{\F} \, ds \, \chi(t) dt \nonumber \\
& \leq
C_* \left(\int_J \left( \frac{1}{h} \int_t^{t+h} ||u(s,\cdot)||_{\F,\sigma B} ds \right)^2 \chi(t) dt \right)^{1/2} \nonumber \\
& \quad \left( \int_J || [ \mathcal H'_n(u_h(t,\cdot)) - \mathcal H'_n(u(t,\cdot))]\psi^2 ||_{\F}^2  \chi(t) dt \right)^{1/2}.
\end{align}
We will show that the right hand side of \eqref{eq:want to let h to 0 part3} goes to $0$ as $h \to 0$, by considering the two factors separately. We may assume that $h$ is small, more precisely $h < \frac{1-\sigma}{2} \tau r^2$ so that $\chi(t) = 0$ for $t \in (-\infty, s_0 + h]$.  
By Lemma \ref{lem:steklov convergence}, the dominated convergence theorem, Lemma \ref{lem:steklov integrated} and Jensen's inequality, we have
\begin{align*}
\lim_{h \to 0} \int_J \left( \frac{1}{h} \int_t^{t+h} ||u(s,\cdot)||_{{\F,\sigma B}} ds \right)^2 \chi(t) dt  \leq
 \int_J  ||u(t,\cdot)||_{{\F,\sigma B}}^2 \chi(t) dt  < \infty.
\end{align*}

We show that the last integral on the right hand side of \eqref{eq:want to let h to 0 part3} tends to $0$ as $h \to 0$. Let us first 
consider the integrand at some fixed $t \in J$. We write $u_h$ for an $L^2$-representative of $u_h(t,\cdot)$, and $u$ for an
 $L^2$-representative of $u(t,\cdot)$.
 Note that, we cannot apply the chain rule \eqref{eq:chain rule for Gamma} with $\Phi=\mathcal H'$ because $\mathcal H'$ is not 
in $\mathcal C^2(\R)$.

We claim that

\begin{align}
\begin{split} \label{eq:result of shortcut lemma}
& \int d\Gamma((\mathcal H'_n(u_h) - \mathcal H'_n(u))\psi^2,(\mathcal H'_n(u_h) - \mathcal H'_n(u))\psi^2)  \\
& \leq
\int [(u_h \wedge n)^{p-1} - (u \wedge n)^{p-1}]^2 d\Gamma(\psi^2,\psi^2)  \\
& \quad 
+ (p-1)^2 n^{2(p-2)} \int \psi^2 d\Gamma( u_h \wedge n - u \wedge n, u_h \wedge n - u \wedge n)  \\
& \quad + 
(p-1)^2 \int [(u_h \wedge n)^{p-2} - (u \wedge n)^{p-2}]^2 \psi^2 d\Gamma(u \wedge n,u \wedge n)  \\
& \quad +
n^{2(p-2)} \int d\Gamma([(u_h - n)^+ - (u - n)^+]\psi^2,[(u_h - n)^+ - (u - n)^+]\psi^2).
\end{split}
\end{align}
We introduce a shortcut notation $A = u_h (u_h \wedge n)^{p-2} - u (u \wedge n)^{p-2}$.
It is easy to see that
\begin{align*}
A 
& = u_h (u_h \wedge n)^{p-2} - u (u \wedge n)^{p-2} \\
& = (u_h \wedge n)^{p-1} - (u \wedge n)^{p-1} + (u_h - n)^+ (u_h \wedge n)^{p-2} - (u - n)^+ (u \wedge n)^{p-2} \\
& = (u_h \wedge n)^{p-1} - (u \wedge n)^{p-1} + n^{p-2} [(u_h - n)^+ - (u - n)^+ ].
\end{align*}
By the chain rule \eqref{eq:chain rule for Gamma} for $\Gamma$,
\begin{align*}
\int d\Gamma(A\psi^2,A\psi^2)
& \leq
\int (u_h \wedge n)^{p-1} - (u \wedge n)^{p-1} d\Gamma(\psi^2,A \psi^2) \\
& \quad 
+ (p-1) \int (u_h \wedge n)^{p-2} \psi^2 d\Gamma( u_h \wedge n - u \wedge n, A \psi^2) \\
& \quad + 
(p-1) \int [(u_h \wedge n)^{p-2} - (u \wedge n)^{p-2}] \psi^2 d\Gamma(u \wedge n, A\psi^2 ) \\
& \quad +
n^{p-2} \int d\Gamma([(u_h - n)^+ - (u - n)^+]\psi^2, A\psi^2).
\end{align*}
Applying Cauchy-Schwarz inequality \eqref{eq:CS} to each integral on the right hand side,
dividing both sides by $\left( \int d\Gamma(A\psi^2,A\psi^2) \right)^{1/2}$, and squaring both sides, 
we obtain that inequality \eqref{eq:result of shortcut lemma} holds for almost every $t \in J$. Now we multiply both sides of 
inequality \eqref{eq:result of shortcut lemma} by $\chi(t)$ and integrate over $J$. Taking a sequence 
of positive reals $h$ that converges to $0$, we obtain
\begin{align} 
& \lim_{h \to 0} \int_J  \int d\Gamma((\mathcal H'_n(u_h) - \mathcal H'_n(u))\psi^2,(\mathcal H'_n(u_h) - \mathcal H'_n(u))\psi^2) \chi(t) dt \nonumber \\
& \leq 
\lim_{h \to 0} \int_J  \int [(u_h \wedge n)^{p-1} - (u \wedge n)^{p-1}]^2 d\Gamma(\psi^2,\psi^2) \chi(t) dt\\
& \quad 
+ (p-1)^2 n^{2(p-2)}  \lim_{h \to 0} \int_J  || u_h \wedge n - u \wedge n ||_{\F}^2 \chi(t) dt\\
& \quad + 
(p-1)^2  \lim_{h \to 0} \int_J \int [(u_h \wedge n)^{p-2} - (u \wedge n)^{p-2}]^2 \psi^2 d\Gamma(u \wedge n,u \wedge n) \chi(t) dt\\
& \quad +
n^{2(p-2)}  \lim_{h \to 0} \int_J  || [(u_h - n)^+ - (u - n)^+]\psi^2 ||_{\F}^2 \chi(t) dt.
\end{align}
Since $u_h \to u$ in $L^2(I \to \F)$ by Corollary \ref{cor:steklov convergence in L^2}, it follows from the locality of the reference form that $(u_h -n)^+ \to (u -n)^+$ and $u_h \wedge n \to u \wedge n$ in $L^2(I \to \F)$ as $h \to 0$. Passing to a subsequence of reals $h$ that goes to $0$, Proposition \ref{prop:dt-qe convergence} yields that for almost every $t \in J$, $(\widetilde{u_h}(t,x) -n)^+ \to (\widetilde{u}(t,x) -n)^+$ and $\widetilde{u_h}(t,x) \wedge n \to \widetilde{u}(t,x) \wedge n$ at quasi-every $x \in X$.
We obtain that
\begin{align*} 
 \int_J ||  [\mathcal H'_n(u_h(t,\cdot)) - \mathcal H'_n(u(t,\cdot))] \psi^2||_{\F}^2 \chi(t) dt \longrightarrow 0 \quad \mbox{ as } h \to 0.
\end{align*}
Thus, \eqref{eq:want to let h to 0 part1} tends to $0$ as we let $h \to 0$. 

\paragraph*{STEP 3.}
At almost every $t \in (s_0,t_0]$, $u_h(t)$ converges to $u(t)$ in $\F$ by Lemma \ref{lem:steklov convergence}; hence, passing to a subsequence, $u_h(t,x) \to u(t,x)$ pointwise at quasi-every $x \in X$. Therefore $\mathcal{H}_n(u_h(t,\cdot)) \to \mathcal{H}_n(u(t,\cdot))$ quasi-everywhere, and by dominated convergence also in $L^1(X,\mu)$.

Letting $h \to 0$ in \eqref{eq:steklov subsol estimate lhs} - \eqref{eq:steklov subsol estimate iv}, and applying \eqref{eq:result of Step 1 sym strongly local} and the fact \eqref{eq:want to let h to 0 part1} and \eqref{eq:want to let h to 0 part2} vanish in the limit (as proved in Step 2), we obtain that
\begin{align*} 
& \quad  \int \mathcal{H}_n(u(t_0,\cdot)) \psi^2 d\mu 
 +  \frac{p-2}{C_1}  \int_{J} \int \chi \psi^2 u_n^{p-2} d\Gamma(u_n,u_n) dt \\
& \quad + \left( \frac{1}{C_1} - \frac{p-1}{C_1 k_1} \right) \int_{J} \int u_n^{p-2} \chi \psi^2 d\Gamma(u,u) dt \\
& \leq 
 \frac{4k_1 C_1}{p-1} \int_J \int u^2 u_n^{p-2} d\Gamma(\psi,\psi) dt 
+ \int_J \int_X  \mathcal H_n(u) \psi^2 \chi' \, d\mu \, dt.
\end{align*}
We make an appropriate choice of the constant $k_1$ so that $k_1$ is of order $p^2$.
We take the supremum over all $t_0 \in I^-_{\sigma'}$ on both sides of the above inequality, multiply each side by $p$, and take the limit as $n \to \infty$. This completes the proof.
\end{proof}

\begin{remark} \label{rem:no time chain rule with H}
Even in the case when $u$ has a weak time-derivative as defined in Section \ref{ssec:weak time-derivative}, we are not aware of a direct proof of Theorem \ref{thm:estimate subsol p>2} in the literature
that avoids the Steklov averages, unless we know a priori that the subsolution $u$ is locally bounded. 
The difficulty is that the chain rule of Proposition \ref{prop:time chain rule} for the weak time-derivative is not applicable with $\Phi = \mathcal H_n$ because $\mathcal H_n$ is not in
 $\mathcal C^2(\R)$.
\end{remark}

\subsubsection{The general case}

\begin{proof}
The strategy of the proof is the same as in the symmetric strongly local case.

The beginning of the proof until \eqref{eq:steklov subsol estimate iv} is exactly the same as in Subsection \ref{ssec:proof sym strongly local}, so we can jump right to Step 1.

\paragraph{STEP 1.}
We want to estimate 
\begin{align*} 
- \int_J \frac{1}{h} \int_t^{t+h}\e_s( u(t,\cdot), \mathcal H'_n(u(t,\cdot)) \psi^2 ) ds \, \chi(t) dt
\end{align*}
To this end, we decompose the bilinear form $\e_s$ into its symmetric strongly local part, its symmetric zero order part, and its skew-symmetric part. 
Let us keep $s$ and $t$ fixed for a moment and consider
the integrand. Write $u$ for $u(t,\cdot)$ and $u_n$ for $u_n(t,\cdot)$. 

We proceed to estimate separately each part in the 
decomposition
\begin{align} \label{eq:e u estimate}
\e_s(u,u u_n^{p-2} \psi^2) =  \e^{\mbox{\tiny{s}}}_s(u,u u_n^{p-2} \psi^2) + \e^{\mbox{\tiny{sym}}}_s(1,u^2 u_n^{p-2} \psi^2) + \e^{\mbox{\tiny{skew}}}_s(u,u u_n^{p-2} \psi^2).
\end{align}

By Lemma \ref{lem:SUP sym2} and Assumption 1(i), we have for any $k_1 > p-1$,
\begin{align}  \label{eq:e^s estimate}
& \quad - \e^{\mbox{\tiny{s}}}_s (u, u u_n^{p-2}  \psi^2) \nonumber \\
& \leq \frac{4k_1}{p-1} \int u^2 u_n^{p-2}  d\Gamma_s(\psi,\psi) 
        - \left( 1 - \frac{p-1}{k_1} \right) \int  u_n^{p-2}  \psi^2  d\Gamma_s(u,u) \nonumber \\
  & \quad  - (p-2) \int u_n^{p-2}  \psi^2   d\Gamma_s(u_n,u_n) \nonumber \\
& \leq \frac{4 k_1 C_1}{p-1} \int u^2 u_n^{p-2}  d\Gamma(\psi,\psi) 
        - C_1^{-1} \left(1 - \frac{p-1}{k_1} \right) \int  u_n^{p-2}  \psi^2  d\Gamma(u,u) \nonumber \\
  & \quad - (p-2) C_1^{-1} \int u_n^{p-2}  \psi^2   d\Gamma(u_n,u_n).
\end{align}
We will need the following estimate, which follows by strong locality and the chain rule for $\Gamma$.
\begin{align} \label{eq:uu_n^p-2/2}
& \int \psi^2  d\Gamma(u u_n^{\frac{p-2}{2}},u u_n^{\frac{p-2}{2}}) \nonumber \\
& = \int u_n^{p-2} \psi^2  d\Gamma(u - u_n,u - u_n) 
+ \int \psi^2  d\Gamma(u_n^{\frac{p}{2}},u_n^{\frac{p}{2}}) \nonumber \\
& = \int u_n^{p-2} \psi^2  d\Gamma(u,u) - \int u_n^{p-2} \psi^2  d\Gamma(u_n,u_n) \nonumber 
 + \frac{p^2}{4} \int u_n^{p-2} \psi^2  d\Gamma(u_n,u_n) \nonumber \\
& \leq  \int u_n^{p-2} \psi^2  d\Gamma(u,u) + \frac{ p^2}{4} \int u_n^{p-2} \psi^2  d\Gamma(u_n,u_n).
\end{align}
By Proposition \ref{prop:skew identity with p}, Assumption 1(iii), \eqref{eq:Gamma(fg)} and \eqref{eq:uu_n^p-2/2}, we have for any $k_2$, $k_3$, $k_4 >0$,
\begin{align} \label{eq:e^skew estimate}
& \quad - \e^{\mbox{\tiny{skew}}}_s(u,u u_n^{p-2}  \psi^2) \nonumber \\
& = - \e^{\mbox{\tiny{skew}}}_s(u u_n^{\frac{p-2}{2}},u u_n^{\frac{p-2}{2}}  \psi^2) 
    - \frac{2-p}{p} \e^{\mbox{\tiny{skew}}}_s(u_n^{p/2},u_n^{p/2}  \psi^2) \nonumber \\
& \quad  - \frac{2-p}{p} \e^{\mbox{\tiny{skew}}}_s(u_n^p  \psi^2,1) \nonumber \\
& \leq  k_2  \int u^2 u_n^{p-2}  d\Gamma(\psi,\psi)
       +  \left( k_3 \frac{(2-p)^2}{p^2} + \frac{2}{k_4} \right)  \int u_n^{p}  d\Gamma(\psi,\psi) \nonumber \\
  & \quad + \frac{C_4}{k_2} \int \psi^2  d\Gamma(u u_n^{\frac{p-2}{2}}, u u_n^{\frac{p-2}{2}})
       + \left( \frac{C_4}{k_3} + \frac{2}{k_4} \right) \int \psi^2  d\Gamma(u_n^{p/2}, u_n^{p/2}) \nonumber \\
  & \quad + \frac{C_5}{k_2} \int u^2 u_n^{p-2} \psi^2  d\mu +  \left( \frac{C_5}{k_3} + C_5 k_4 \frac{(2-p)^2}{p^2} \right) \int u_n^p \psi^2  d\mu \nonumber \\
& \leq  k_2  \int u^2 u_n^{p-2}  d\Gamma(\psi,\psi)
       +  \left( k_3 \frac{(2-p)^2}{p^2} + \frac{2}{k_4} \right)  \int u_n^{p}  d\Gamma(\psi,\psi) \nonumber \\
  & \quad + \frac{C_4}{k_2} \int u_n^{p-2} \psi^2  d\Gamma(u,u)
       + \frac{p^2}{4} \left(\frac{C_4}{k_2} + \frac{C_4}{k_3} + \frac{2}{k_4} \right) \int u_n^{p-2} \psi^2  d\Gamma(u_n, u_n) \nonumber \\
  & \quad + \frac{C_5}{k_2} \int u^2 u_n^{p-2} \psi^2  d\mu +  \left( \frac{C_5}{k_3} + C_5 k_4 \frac{(2-p)^2}{p^2} \right) \int u_n^p \psi^2  d\mu.
\end{align}
By Assumption 1(ii), \eqref{eq:Gamma(fg)} and \eqref{eq:uu_n^p-2/2}, we have uor any $k_5>0$,
\begin{align} \label{eq:e^sym estimate}
& - \e^{\mbox{\tiny{sym}}}_s(u^2 u_n^{p-2}  \psi^2,1) \nonumber \\
& \leq   \frac{2}{k_5} \int u^2 u_n^{p-2}  d\Gamma(\psi,\psi)
      + \frac{2}{k_5} \int \psi^2  d\Gamma(u u_n^{\frac{p-2}{2}},u u_n^{\frac{p-2}{2}}) \nonumber \\
 & \quad + \left( C_2 k_5 + 2C_3^{1/2} \right) \int u^2 u_n^{p-2} \psi^2  d\mu \nonumber \\
& \leq \frac{2}{k_5} \int u^2 u_n^{p-2}  d\Gamma(\psi,\psi)
 + \frac{2}{k_5} \int u_n^{p-2} \psi^2  d\Gamma(u,u) \nonumber \\
& \quad + \frac{p^2}{2k_5} \int u_n^{p-2} \psi^2  d\Gamma(u_n,u_n) \nonumber \\
& \quad + \left( C_2 k_5 + 2C_3^{1/2} \right) \int u^2 u_n^{p-2} \psi^2  d\mu.
\end{align}
Combining \eqref{eq:e u estimate}, \eqref{eq:e^s estimate}, \eqref{eq:e^skew estimate}, \eqref{eq:e^sym estimate} and rearranging the terms yields
\begin{align*} 
& \quad   - \e_s(u,u u_n^{p-2} \psi^2)    \\ 
& \quad + \left( \frac{p-2}{C_1} - \frac{p^2}{4} \left( \frac{C_4}{k_2} + \frac{C_4}{k_3} + \frac{2}{k_4} + \frac{2}{k_5} \right) \right)  \int \psi^2 u_n^{p-2} d\Gamma(u_n,u_n)  \\
& \quad + \left( \frac{1}{C_1} - \frac{p-1}{C_1 k_1} -\frac{C_4}{k_2} - \frac{2}{k_5} \right)  \int u_n^{p-2} \psi^2 d\Gamma(u,u)  \\
& \leq 
 \left( \frac{4k_1 C_1}{p-1} + k_2 + \frac{2}{k_5} \right)  \int u^2 u_n^{p-2} d\Gamma(\psi,\psi)  \\
& \quad +  \left( k_3 \frac{(2-p)^2}{p^2} + \frac{2}{k_4} \right)  \int u_n^{p} d\Gamma(\psi,\psi) \\
& \quad + \left( \frac{C_5}{k_2} + C_2 k_5 + 2C_3^{1/2}  \right)  \int u^2 u_n^{p-2} \psi^2 d\mu \\
& \quad + \left( \frac{C_5}{k_3} + C_5 k_4 \frac{(2-p)^2}{p^2} \right) \int u_n^p \psi^2 d\mu.
\end{align*}
This inequality holds for any $s \in (t,t+h)$ (recall that we fixed $s$ and $t$ at the beginning of Step 1). Now we take the Steklov average at $t$, and then we integrate over $J$, on both sides of the inequality. This results in an estimate for \eqref{eq:steklov subsol estimate iii}. Therefore, the estimate \eqref{eq:steklov subsol estimate lhs} - \eqref{eq:steklov subsol estimate iv} becomes
\begin{align} 
\begin{split} \label{eq:result of step 1 general case} 
& \quad  \int \mathcal{H}_n(u_h(t_0,\cdot)) \psi^2 d\mu \\
& \quad + \left( \frac{p-2}{C_1} - \frac{p^2}{4} \left( \frac{C_4}{k_2} + \frac{C_4}{k_3} + \frac{2}{k_4} + \frac{2}{k_5} \right) \right) \int_{J} \int \chi \psi^2 u_n^{p-2} d\Gamma(u_n,u_n) dt \\
& \quad + \left( \frac{1}{C_1} - \frac{p-1}{C_1 k_1} -\frac{C_4}{k_2} - \frac{2}{k_5} \right) \int_{J} \int u_n^{p-2} \chi \psi^2 d\Gamma(u,u) dt \\
& \leq 
 \left( \frac{4k_1 C_1}{p-1} + k_2 + \frac{2}{k_5} \right) \int_J \int u^2 u_n^{p-2} d\Gamma(\psi,\psi) dt \\
& \quad +  \left( k_3 \frac{(2-p)^2}{p^2} + \frac{2}{k_4} \right) \int_J \int u_n^{p} d\Gamma(\psi,\psi) dt \\
& \quad + \left( \frac{C_5}{k_2} + C_2 k_5 + 2C_3^{1/2}  \right) \int_J \int u^2 u_n^{p-2} \psi^2 d\mu \, dt \\
& \quad + \left( \frac{C_5}{k_3} + C_5 k_4 \frac{(2-p)^2}{p^2} \right) \int_J \int u_n^p \psi^2 d\mu \, dt \\
& \quad 
- \int_J \frac{1}{h} \int_t^{t+h}\e_s( u(s,\cdot), [\mathcal H'_n(u_h(t,\cdot)) - \mathcal H'_n(u(t,\cdot))] \psi^2 ) \chi(t) ds \, dt \\
& \quad  
- \int_J \frac{1}{h} \int_t^{t+h} \e_s( u(s,\cdot) - u(t,\cdot), \mathcal H'_n(u(t,\cdot)) \psi^2 ) \chi(t) ds \, dt \\
& \quad 
+ \int_J \int_X  \mathcal H_n(u_h) \psi^2 \chi' \, d\mu \, dt. 
\end{split}
\end{align}
This completes Step 1 of the proof.

\paragraph*{STEP 2.}
This part of the proof is identical to Step 2 in Subsection \ref{ssec:proof sym strongly local}. The application of the Cauchy-Schwarz inequality is justified thanks to Assumption 0.

\paragraph*{STEP 3.}
At almost every $t \in (s_0,t_0]$, $u_h(t)$ converges to $u(t)$ in $\F$ by Lemma \ref{lem:steklov convergence}; hence, passing to a subsequence, $u_h(t,x) \to u(t,x)$ pointwise at quasi-every $x \in X$. Therefore $\mathcal{H}_n(u_h(t,\cdot)) \to \mathcal{H}_n(u(t,\cdot))$ quasi-everywhere, and by dominated convergence also in $L^1(X,\mu)$.

Letting $h \to 0$ in \eqref{eq:result of step 1 general case}, and applying the results of Step 2, we obtain that
\begin{align*} 
& \quad  \int \mathcal{H}_n(u(t_0,\cdot)) \psi^2 d\mu \\
& \quad + \left( \frac{p-2}{C_1} - \frac{p^2}{4} \left( \frac{C_4}{k_2} + \frac{C_4}{k_3} + \frac{2}{k_4} + \frac{2}{k_5} \right) \right) \int_{J} \int \chi \psi^2 u_n^{p-2} d\Gamma(u_n,u_n) dt \\
& \quad + \left( \frac{1}{C_1} - \frac{p-1}{C_1 k_1} -\frac{C_4}{k_2} - \frac{2}{k_5} \right) \int_{J} \int u_n^{p-2} \chi \psi^2 d\Gamma(u,u) dt \\
& \leq 
 \left( \frac{4k_1 C_1}{p-1} + k_2 + \frac{2}{k_5} \right) \int_J \int u^2 u_n^{p-2} d\Gamma(\psi,\psi) dt \\
& \quad +  \left( k_3 \frac{(2-p)^2}{p^2} + \frac{2}{k_4} \right) \int_J \int u_n^{p} d\Gamma(\psi,\psi) dt \\
& \quad + \left( \frac{C_5}{k_2} + C_2 k_5 + 2C_3^{1/2}  \right) \int_J \int u^2 u_n^{p-2} \psi^2 d\mu \, dt \\
& \quad + \left( \frac{C_5}{k_3} + C_5 k_4 \frac{(2-p)^2}{p^2} \right) \int_J \int u_n^p \psi^2 d\mu \, dt \\
& \quad 
+ \int_J \int_X  \mathcal H_n(u) \psi^2 \chi' \, d\mu \, dt.
\end{align*}

We take the supremum over all $t_0 \in I^-_{\sigma'}$ on both sides of the above inequality, multiply each side by $p$, and take the limit as $n \to \infty$. This completes the proof.
\end{proof}

\subsection{Proof of Proposition \ref{prop:|u| is subsolution}}

\begin{proof}
Let $\epsilon >0$ and let $\Phi(x) = \Phi_{\epsilon}(x) := \sqrt{x^2 + \epsilon} - \sqrt{\epsilon}$. 
Note that $\Phi \in \mathcal C^3(\R)$ and $\Phi(0) = \Phi'(0) = 0$. Let $u$ be a local very weak solution in $Q$. Due to the Markov property of the reference form, $|u| \in L^2_{\mbox{\tiny{loc}}}(I \to \F; U)$. By an approximation argument that applies Assumption 0(i), it suffices to show inequality \eqref{eq:weak sol} 
for non-negative functions $\phi \in \mathcal D$. 
For $h>0$ let $u_h(t) = \int_t^{t+h} u(s) ds$ be the Steklov average of $u$. 

For almost every $a,b \in I$, and for $h \in (0,b-a)$ small enough so that $b+h \in I$, we have
\begin{align*}
& 
\int \Phi (u_h(b,\cdot))  \phi \, d\mu - \int \Phi(u_h(a,\cdot))  \phi \, d\mu \\
& = \int_a^b \frac{d}{dt} \left( \int \Phi(u_h) \phi \, d\mu \right) dt \\
& =
\int_a^b \int_X \frac{\partial}{\partial t} \left( u_h \right)  \Phi'(u_h) \phi \, d\mu \, dt  \\
& = 
\int_a^b \frac{1}{h} \int_X \big( u(t+h,\cdot) - u(t,\cdot) \big)  \Phi'(u_h(t,\cdot)) \phi  \, d\mu \, dt.
\end{align*}
Next, we make use of the assumption that $u$ is a local very weak solution of the heat equation in $Q$. We obtain
\begin{align} \label{eq:steklov lhs}
& 
\int \Phi (u_h(b,\cdot))  \phi \, d\mu - \int \Phi(u_h(a,\cdot))  \phi \, d\mu  \\
& = 
- \int_a^b \frac{1}{h} \int_t^{t+h}\e_s( u(s,\cdot), \Phi'(u_h(t,\cdot)) \phi ) ds \, dt  \nonumber \\
& = \label{eq:want to let h to 0 parta}
- \int_a^b \frac{1}{h} \int_t^{t+h}\e_s( u(s,\cdot), [\Phi'(u_h(t,\cdot)) - \Phi'(u(t,\cdot))] \phi )  ds \, dt \\
& \quad \label{eq:want to let h to 0 partb}
- \int_a^b \frac{1}{h} \int_t^{t+h} \e_s( u(s,\cdot) - u(t,\cdot), \Phi'(u(t,\cdot)) \phi )  ds \, dt \\
& \quad \label{eq:steklov subsol estimate c}
- \int_a^b \frac{1}{h} \int_t^{t+h} \e_s( u(t,\cdot), \Phi'(u(t,\cdot)) \phi )  ds \, dt.
\end{align}
We let $h \to 0$. By the same argument as in Step 3 of the proof of Theorem \ref{thm:estimate subsol p>2}, \eqref{eq:steklov lhs} converges to
\begin{align*} 
\int \Phi (u(b,\cdot))  \phi \, d\mu - \int \Phi(u(a,\cdot))  \phi \, d\mu.
\end{align*}
In taking the limit in \eqref{eq:want to let h to 0 partb} and \eqref{eq:want to let h to 0 parta}, we follow the reasoning in Step 2 in the proof of Theorem \ref{thm:estimate subsol p>2}. We easily find that \eqref{eq:want to let h to 0 partb} tends to $0$. To show that \eqref{eq:want to let h to 0 parta} tends to $0$ as $h \to 0$, it suffices to prove that 
\begin{align*}
\int_a^b \Vert \Phi'(u_h(t,\cdot)) - \Phi'(u(t,\cdot))] \phi \Vert_{\F}^2 dt \to 0 \quad \mbox{ as } h \to 0.
\end{align*}
By Lemma \ref{lem:steklov convergence}, $u_h(t,\cdot) \to u(t,\cdot)$ in $\F$. Applying Proposition \ref{prop:dt-qe convergence} and passing to a subsequence which we again denote by $(u_h)$, we obtain that, for almost every $t \in (a,b)$, $\widetilde{u_h}(t,\cdot) \to \widetilde{u}(t,\cdot)$ quasi-everywhere.
Since $\phi$ and $\Phi'$ are bounded, this shows that
\begin{align*}
\int_a^b \Vert \Phi'(u_h(t,\cdot)) - \Phi'(u(t,\cdot))] \phi \Vert_{L^2(X,\mu)}^2 dt \to 0 \quad \mbox{ as } h \to 0,
\end{align*}
by dominated convergence.
Let $A := \Phi'(u_h(t,\cdot)) - \Phi'(u(t,\cdot))$ and notice that by the continuity of $\Phi'$, for almost every $t$, $A \to 0$ quasi-everywhere. Since $\Phi'$ is bounded, we obtain that $\int_a^b A^2 d\Gamma(\phi,\phi) dt \to 0$ as $h \to 0$ by dominated convergence. By \eqref{eq:chain rule for Gamma}, the boundedness of $\phi$, the fact $\Phi''$ is bounded and continuous, and by the dominated convergence theorem, we also obtain that $\int_a^b \phi^2 d\Gamma(A,A) dt \to 0$ as $h \to 0$.
Thus, by \eqref{eq:chain rule for Gamma}, \eqref{eq:Gamma(fg)} and \eqref{eq:CS}, we get
\begin{align*}
\int_a^b d\Gamma(A\phi,A \phi) dt
 \leq 
\int_a^b \phi^2 d\Gamma(A,A) dt + \int_a^b A^2 d\Gamma(\phi,\phi) dt
 \longrightarrow 0.
\end{align*}
This proves that \eqref{eq:want to let h to 0 parta} tends to $0$ as $h \to 0$.

Finally, \eqref{eq:steklov subsol estimate c} converges to
\begin{align*}
-\int_a^b \e_t( u(t,\cdot), \Phi'(u(t,\cdot)) \phi ) dt.
\end{align*}

Observe that 
\[ \int_a^b  \e_t(|u(t,\cdot)|,\phi) dt 
=
\lim_{\epsilon \to 0} 
 \int_a^b \e_t(\Phi(u(t,\cdot)),\phi) ,
\]
by \eqref{eq:chain rule for Gamma}, \eqref{eq:CS}, the fact that $\Phi'$ is bounded and continuous, Proposition \ref{prop:dt-qe convergence}, and the dominated convergence theorem.
Hence,
\begin{align} \label{eq:rhs non-positive}
& \int |u(b,\cdot)|  \phi \, d\mu - \int |u(a,\cdot)|  \phi \, d\mu  + \int_a^b  \e_t(|u(t,\cdot)|,\phi) dt \nonumber \\
& \leq
\lim_{\epsilon \to 0} 
\int \Phi (u(b,\cdot))  \phi \, d\mu - \int \Phi(u(a,\cdot))  \phi \, d\mu + \int_a^b \e_t(\Phi(u(t,\cdot)),\phi) \nonumber \\
& \leq 
\lim_{\epsilon \to 0}
\int_a^b \e_t(\Phi(u(t,\cdot)),\phi) - \e_t( u(t,\cdot), \Phi'(u(t,\cdot)) \phi ) dt.
\end{align} 
We will show that the right hand side in non-positive.
We decompose $\e_t$ into its strongly local symmetric part, its zero  order symmetric part, and its skew symmetric part and consider each part separately.
Writing $u$ for $u(t,\cdot)$, we have by \eqref{eq:chain rule for Gamma}
\begin{align*}
\e_t^{\mbox{\tiny{s}}}(\Phi(u),\phi) - \e_t^{\mbox{\tiny{s}}}(u, \Phi'(u) \phi)
 = 
- \int \Phi''(u) \phi \, d\Gamma_t(u,u) 
 \leq 0,
\end{align*}
and by Assumption 0(i),
\begin{align*}
& \quad \e_t^{\mbox{\tiny{sym}}}(\Phi(u)\phi,1) - \e_t^{\mbox{\tiny{sym}}}(u \Phi'(u) \phi,1) 
 = 
\e_t^{\mbox{\tiny{sym}}}([\Phi(u) - u \Phi'(u)] \phi,1) \\
& = 
\e_t^{\mbox{\tiny{sym}}} \left( \frac{\epsilon \phi}{\sqrt{u^2 + \epsilon}} - \sqrt{\epsilon}\phi,1 \right) \\
& \leq C_* \left( \left\Vert \frac{\epsilon \phi}{\sqrt{u^2 + \epsilon}} \right\Vert_{\F} + ||\sqrt{\epsilon}\phi||_{\F} \right) ||1^{\sharp}||_{\F},
\end{align*}
where $1^{\sharp}$ is some function with bounded $\F$-norm so that $1 = 1^{\sharp}$ on the support of $\phi$.
Letting $\epsilon \to 0$, it is clear that $||\sqrt{\epsilon}\phi||_{\F}$ tends to $0$. To see that also $\left\Vert \frac{\epsilon \phi}{\sqrt{u^2 + \epsilon}} \right\Vert_{\F}$ tends to $0$, observe that
$\frac{\epsilon \phi}{\sqrt{u^2 + \epsilon}} \to 0$ pointwise and in $L^2(X,d\mu)$. Moreover, by \eqref{eq:Gamma(fg)} and \eqref{eq:chain rule for Gamma},
\begin{align*}
& \quad 
\int d\Gamma\left(\frac{\epsilon \phi}{\sqrt{u^2 + \epsilon}},\frac{\epsilon \phi}{\sqrt{u^2 + \epsilon}} \right) \\
& \leq 
2 \int \frac{\epsilon^2}{u^2 + \epsilon} d\Gamma(\phi,\phi) 
+ 2 \epsilon^2 \int \phi^2 d\Gamma\left( \frac{1}{\sqrt{u^2 + \epsilon}}, \frac{1}{\sqrt{u^2 + \epsilon}} \right) \\
& \leq 
2 \int \frac{\epsilon^2}{u^2 + \epsilon} d\Gamma(\phi,\phi) 
+ 2  \int \phi^2 \frac{\epsilon^2 u^2}{(u^2 + \epsilon)^3} d\Gamma(u,u).
\end{align*}
As $\epsilon \to 0$, the right hand side tends to $0$ by the dominated convergence theorem and because $\phi$ is bounded and $\frac{\epsilon^2 u^2}{(u^2 + \epsilon)^3}$
is bounded and tends to $0$ pointwise on $\{u = 0\}$ and on $\{u \neq 0 \}$.
This proves that $\e_t^{\mbox{\tiny{sym}}}(\Phi(u)\phi,1) - \e_t^{\mbox{\tiny{sym}}}(u \Phi'(u) \phi,1) \to 0$ as $\epsilon \to 0$, at almost every $t \in (a,b)$.

Lastly, we consider the skew-symmetric part of $\e_t$. Let $u_m:= (u \wedge m) \vee (-m)$. Then $(u_m - u) \to 0$ in $\F$, $(\Phi(u_m) - \Phi(u)) \to 0$ in $\F$, and $(\Phi'(u_m) \phi - \Phi'(u) \phi) \to 0$ in $\F$ by \eqref{eq:chain rule for Gamma}, \eqref{eq:CS} and Proposition \ref{prop:u_mg converges in F}. 
Hence, and by Assumption 0(i), we can approximate $u$ by $u_m$. More precisely,
\begin{align*}
& \quad \e_t^{\mbox{\tiny{skew}}}(\Phi(u),\phi) - \e_t^{\mbox{\tiny{skew}}}(u,\Phi'(u) \phi) \\
& = 
\lim_{m \to \infty} \left[ \e_t^{\mbox{\tiny{skew}}}(\Phi(u_m),\phi) - \e_t^{\mbox{\tiny{skew}}}(u_m,\Phi'(u_m) \phi) \right].
\end{align*}
By  the density of $\mathcal D$ in $(\F,\| \cdot \}_{\F})$ and \cite[Theorem 2.1.4]{FOT94}, there exists a sequence $(f_k) = (f_{k,m}) \subset \mathcal D$ that converges to $u_m$ in $\F$ and so that the quasi-continuous versions $\widetilde{f_k} \leq m$ converge to $\widetilde{u_m}$ quasi-everywhere. 
Hence,
\begin{align*}
& \quad \e_t^{\mbox{\tiny{skew}}}(\Phi(u),\phi) - \e_t^{\mbox{\tiny{skew}}}(u,\Phi'(u) \phi) \\
& = 
\lim_{k,m \to \infty} \left[ \e_t^{\mbox{\tiny{skew}}}(\Phi(f_{k,m}),\phi) - \e_t^{\mbox{\tiny{skew}}}(f_{k,m},\Phi'(f_{k,m}) \phi) \right].
\end{align*}
Keeping $m$ and $k$ fixed for a moment, we write $f$ for $f_{k,m}$. 
By the chain rule for $\l$ and the Leibniz rule and the chain rule for $\r$ which hold by Assumption 0(iv),
\begin{align*}
& \quad 
\e_t^{\mbox{\tiny{skew}}}(\Phi(f),\phi) - \e_t^{\mbox{\tiny{skew}}}(f,\Phi'(f) \phi) \\
& =
\l(\Phi(f),\phi) - \l(f, \Phi'(f) \phi) + \r(\Phi(f),\phi) - \r(f, \Phi'(f) \phi) \\
& = 
\r(\Phi(f) - f \Phi'(f),\phi) - \r(f \phi \Phi''(f),f) \\
& =
\frac{1}{2} \Big[ 
\e_t^{\mbox{\tiny{skew}}}(\Phi(f) - f \Phi'(f),\phi) 
- \e_t^{\mbox{\tiny{skew}}}(\Phi(f) - f \Phi'(f)\phi,1) \\
& \quad +  \e_t^{\mbox{\tiny{skew}}}(f,f \phi \Phi''(f)) 
+ \e_t^{\mbox{\tiny{skew}}}(f^2 \Phi''(f) \phi,1) \Big].
\end{align*}
Letting now $k,m \to \infty$, we find that
\begin{align*}
& f_{k,m} \to u, \\
& \Phi(f_{k,m}) \to \Phi(u), \\
& \Phi'(f_{k,m}) \to \Phi'(u), \\
& f_{k,m}\Phi'(f_{k,m}) \to u \Phi'(u), \\
& f_{k,m}\Phi'(f_{k,m})\phi \to u \Phi'(u)\phi, \\
& f_{k,m}^2 \Phi''(f_{k,m}) \phi \to u^2 \Phi''(u) \phi,
\end{align*}
both pointwise and in $\F$.
Hence, by Assumption 0(i),
\begin{align*}
& \quad 
\e_t^{\mbox{\tiny{skew}}}(\Phi(u),\phi) - \e_t^{\mbox{\tiny{skew}}}(u,\Phi'(u) \phi) \\
& =
\frac{1}{2} \Big[ 
\e_t^{\mbox{\tiny{skew}}}(\Phi(u) - u \Phi'(u),\phi) 
- \e_t^{\mbox{\tiny{skew}}}(\Phi(u) - u \Phi'(u)\phi,1) \\
& \quad +  \e_t^{\mbox{\tiny{skew}}}(u,u \phi \Phi''(u)) 
+ \lim_{m \to \infty} \e_t^{\mbox{\tiny{skew}}}(u_m^2 \Phi''(u_m) \phi,1) \Big].
\end{align*}
We have shown above that $[\Phi(u) - u \Phi'(u)]$ and $[\Phi(u) - u \Phi'(u)]\phi$ tend to $0$ in $\F$ as $\epsilon \to 0$. It is not hard to see that also $u \phi \Phi''(u)$ and $u_m^2 \Phi''(u_m) \phi$ tend to $0$ in $\F$. 
Hence, by Assumption 0(i), 
\begin{align*}
\e_t^{\mbox{\tiny{skew}}}(\Phi(u),\phi) - \e_t^{\mbox{\tiny{skew}}}(u,\Phi'(u) \phi)
 \longrightarrow 0 \quad \mbox{ as } \epsilon \to 0,
\end{align*}
at almost every $t \in (a,b)$.

Now it follows from the dominated convergence theorem, Assumption 0(i), and the boundedness of $\phi$, $\Phi'$ and $\Phi''$ that the right hand side of \eqref{eq:rhs non-positive} is non-positive, so $|u|$ is a local very weak subsolution of the heat equation.
\end{proof}

\subsection{Proof of Lemma \ref{lem:estimate subsol p<2}}

\begin{proof}
Let $u_{\varepsilon}:= u + \varepsilon$ for $\varepsilon > 0$. 
Let $u_{\varepsilon,h} := (u_{\varepsilon})_h$ be the Steklov average of $u_{\varepsilon}$.

We use the fact that the Steklov average has a strong time-derivative, and the assumption that $u$ is a local very weak subsolution of the heat equation in $Q$. Let $s_0 = a-\frac{1+\sigma}{2}\tau r^2$.
For a.e. $t_0 \in I^-_{\sigma'}$ and $J = (s_0, t_0)$, we have
\begin{align}
& \quad  \frac{1}{p} \int u_{\varepsilon,h}(t_0,\cdot)^p  \psi^2 d\mu  \label{eq:u_e steklov LHS} \\
& = \frac{1}{p} \int u_{\varepsilon,h}(t_0,\cdot)^p \psi^2 \chi (t_0) d\mu  -  \frac{1}{p} \int u_{\varepsilon,h}(s_0,\cdot)^p  \psi^2 \chi(s_0) d\mu  \nonumber \\
& = 
 \int_J \frac{1}{h} \big( u(t+h, \cdot) - u(t,\cdot) \big) u_{\varepsilon,h}(t,\cdot)^{p-1} \psi^2 \chi(t)  d\mu \, dt 
       + \int_J \int \frac{\chi'(t)}{p} u_{\varepsilon,h}(t,\cdot)^p \psi^2 d\mu \, dt \nonumber \\
& \leq 
 - \int_J \frac{1}{h} \int_t^{t+h} \e_s(u(s,\cdot),u_{\varepsilon,h}(t,\cdot)^{p-1} \psi^2 \chi(t) ) ds \, dt 
 + \int_J \int \frac{\chi'}{p} u_{\varepsilon,h}^p  \psi^2 d\mu \, dt \nonumber \\
& \leq 
 - \int_J \frac{1}{h} \int_t^{t+h} \e_s(u(s,\cdot),[u_{\varepsilon,h}(t,\cdot)^{p-1} - u_{\varepsilon}(t,\cdot)^{p-1}] \psi^2) ds \, \chi(t) dt  \label{eq:u_e steklov i} \\
& \quad  - \int_J \frac{1}{h} \int_t^{t+h} \e_s(u(s,\cdot) - u(t,\cdot),u_{\varepsilon}(t,\cdot)^{p-1} \psi^2 ) ds \, \chi(t) dt  \label{eq:u_e steklov ii} \\
& \quad  - \int_J \frac{1}{h} \int_t^{t+h} \e_s(u(t,\cdot),u_{\varepsilon}(t,\cdot)^{p-1} \psi^2 ) ds \, \chi(t)dt \label{eq:u_e steklov iii} \\
& \quad  + \int_J \int \frac{\chi'}{p} u_{\varepsilon,h}^p  \psi^2 d\mu \, dt. \label{eq:u_e steklov iv} 
\end{align}

\paragraph*{STEP 1.}
We decompose the integrand in \eqref{eq:u_e steklov iii} as 
\begin{align*}
- \e_s(u,u_{\varepsilon}^{p-1} \psi^2 )
& \leq 
 - \e^{\mbox{\tiny{s}}}_s(u_{\varepsilon},u_{\varepsilon}^{p-1}  \psi^2 ) 
 - \e^{\mbox{\tiny{sym}}}_s(u_{\varepsilon}^p \psi^2 ,1) 
 - \e^{\mbox{\tiny{skew}}}_s(u_{\varepsilon},u_{\varepsilon}^{p-1}\psi^2 ) \nonumber \\
& \quad 
+ \e_s(\varepsilon,u_{\varepsilon}^{p-1} \psi^2 ).
\end{align*}
and estimate each part separately.
By Lemma \ref{lem:SUP sym2}, Assumption 1(i) and Remark \ref{rem:assumptions}(i), we have for any $k_1 \geq 1$,
\begin{align*}
&- \e^{\mbox{\tiny{s}}}_s (u_{\varepsilon}, u_{\varepsilon}^{p-1}  \psi^2) \\
 \leq & \, \frac{4 k_1}{p-1} \int u_{\varepsilon}^p  d\Gamma_s(\psi,\psi) 
     - (p-1) \left(1 - \frac{1}{k_1} \right) \int  u_{\varepsilon}^{p-2}  \psi^2  d\Gamma_s(u_{\varepsilon},u_{\varepsilon}) \\
 \leq & \, \frac{4 k_1 C_1}{p-1} \int u_{\varepsilon}^p  d\Gamma(\psi,\psi)
       - (p-1) C_1^{-1} \left(1 - \frac{1}{k_1} \right) \int u_{\varepsilon}^{p-2}  \psi^2   d\Gamma(u_{\varepsilon},u_{\varepsilon}).
\end{align*}
By Assumption 1(ii), Remark \ref{rem:assumptions}(i), and \eqref{eq:Gamma(fg)}, we have for any $k_2 > 0$,
\begin{align*}
|\e^{\mbox{\tiny{sym}}}_s(u_{\varepsilon}^p  \psi^2,1)| 
\leq & \frac{2}{k_2} \int u_{\varepsilon}^p  d\Gamma(\psi,\psi) \nonumber
 + \frac{p^2}{2k_2} \int \psi^2 u_{\varepsilon}^{p-2}  d\Gamma(u_{\varepsilon},u_{\varepsilon}) \\
& +  ( C_2 k_2 + 2C_3^{1/2}) \int u_{\varepsilon}^p  \psi^2 d\mu.
\end{align*}
By Corollary \ref{cor:skew identity with p for u_eps}, Assumption 1(iii), Remark \ref{rem:assumptions}(i), and \eqref{eq:Gamma(fg)}, we have for any $k_3$, $k_4 >0$,
\begin{align*}
& - \e^{\mbox{\tiny{skew}}}_s(u_{\varepsilon},u_{\varepsilon}^{p-1}  \psi^2) \\
= & - \frac{2}{p} \e^{\mbox{\tiny{skew}}}_s(u_{\varepsilon}^{p/2},u_{\varepsilon}^{p/2}  \psi^2) 
   - \frac{2-p}{p} \e^{\mbox{\tiny{skew}}}_s(u_{\varepsilon}^p  \psi^2,1) \\
\leq &  \left( k_3 \frac{4}{p^2} + \frac{2}{k_4} \right)  \int u_{\varepsilon}^{p}  d\Gamma(\psi,\psi)
    + \left( \frac{C_4}{k_3} + \frac{2}{k_4} \right) \int \psi^2  d\Gamma(u_{\varepsilon}^{p/2}, u_{\varepsilon}^{p/2}) \\
     & +   \left( \frac{C_5}{k_3} + C_5 k_4 \frac{(2-p)^2}{p^2} \right) \int u_{\varepsilon}^p \psi^2  d\mu.
\end{align*}
By Lemma \ref{lem:e(1,)}, we have for any $k_5\geq 1$,
\begin{align*}
 |\e_s(\varepsilon,u_{\varepsilon}^{p-1}  \psi^2)| 
\leq & \frac{4}{k_5} \int u_{\varepsilon}^p  d\Gamma(\psi,\psi) 
+ \frac{(p-1)^2}{k_5} \int \psi^2 u_{\varepsilon}^{p-2}  d\Gamma(u_{\varepsilon},u_{\varepsilon}) \\
& + 2(C_2+C_3^{1/2}+C_5) k_5 \int u_{\varepsilon}^p \psi^2  d\mu.
\end{align*}

Combining these estimates with inequality \eqref{eq:u_e steklov LHS} - \eqref{eq:u_e steklov iv}, we obtain
\begin{align*} 
&  \frac{1}{p} \int u_{\varepsilon,h}^p(t_0,\cdot) \psi^2 d\mu \\
& \quad  + \left( \frac{p-1}{C_1} \left(1 - \frac{1}{k_1} \right) - \frac{p^2}{2k_2} - \frac{p^2 C_4}{4 k_3} - \frac{p^2}{2 k_4} - \frac{(p-1)^2}{k_5} \right) \\
& \qquad \cdot
\int_J \frac{1}{h} \int_t^{t+h} \int \psi^2  u_{\varepsilon}^{p-2}  d\Gamma(u_{\varepsilon},u_{\varepsilon}) ds \, \chi(t) dt \\
& \leq 
 \left( \frac{4 k_1 C_1}{p-1} + \frac{2}{k_2} + k_3 \frac{4}{p^2} + \frac{2}{k_4} + \frac{4}{k_5} \right) 
\int_J \frac{1}{h} \int_t^{t+h} \int u_{\varepsilon}^{p}  d\Gamma(\psi,\psi) ds \, \chi(t) dt \\
& \quad  + C' \int_J \frac{1}{h} \int_t^{t+h} \int u_{\varepsilon}^{p}  d\Gamma(\psi,\psi) ds \, \chi(t) dt \\
& \quad - \int_J \frac{1}{h} \int_t^{t+h} \e_s(u(s,\cdot),[u_{\varepsilon,h}(t,\cdot)^{p-1} - u_{\varepsilon}(t,\cdot)^{p-1}] \psi^2) ds \, \chi(t) dt \\
& \quad  - \int_J \frac{1}{h} \int_t^{t+h} \e_s(u(s,\cdot) - u(t,\cdot),u_{\varepsilon}(t,\cdot)^{p-1} \psi^2 ) ds \, \chi(t) dt  \\
& \quad  + \int_J \int \frac{2}{p \omega \tau r^2} u_{\varepsilon,h}^p  \psi^2 d\mu \, dt.
\end{align*}
where $C' = C_2 k_2 + 2C_3^{1/2} + \frac{C_5}{k_3} + C_5 k_4 \frac{(2-p)^2}{p^2} + 2(C_2+C_3^{1/2}+C_5) k_5$.

\paragraph*{STEP 2.} 
Now we consider the limit as $h \to 0$.
As in Step 2 of the proof of Theorem \ref{thm:estimate subsol p>2}, we see that \eqref{eq:u_e steklov i} and \eqref{eq:u_e steklov ii} go to $0$. Appropriate choices of $k_1,k_2, k_3, k_4, k_5$ allow us to let $h \to 0$ in the remaining terms, similarly to Step 3 in the proof Theorem \ref{thm:estimate subsol p>2}.
Finally, we take the supremum over all $t_0 \in I^-_{\sigma'}$ on both sides of the inequality, and then let $\varepsilon \to 0$. This completes the proof.
\end{proof}

\subsection{Proof of Lemma \ref{lem:estimate supsol}}

\begin{proof}
Let $u_{\varepsilon}:= u + \varepsilon$ for $\varepsilon > 0$. 
Let $u_{\varepsilon,h} := (u_{\varepsilon})_h$ be the Steklov average of $u_{\varepsilon}$.
We first consider the case $p<0$. 
We use the fact that the Steklov average has a strong time-derivative, and the assumption that $u$ is a local very weakly supersolution of the heat equation in $Q$. Let $s_0 = a-\frac{1+\sigma}{2}\tau r^2$.
For a.e. $t_0 \in I^-_{\sigma'}$ and $J = (s_0, t_0)$, we have
\begin{align}
& \quad  \int u_{\varepsilon,h}(t_0,\cdot)^p  \psi^2 d\mu  \nonumber \\
& = \int u_{\varepsilon,h}(t_0,\cdot)^p \psi^2 \chi (t_0) d\mu  - \int u_{\varepsilon,h}(s_0,\cdot)^p  \psi^2 \chi(s_0) d\mu  \nonumber \\
& = 
 p \int_J \frac{1}{h} \big( u(t+h, \cdot) - u(t,\cdot) \big) u_{\varepsilon,h}(t,\cdot)^{p-1} \psi^2 \chi(t)  d\mu \, dt 
       + \int_J \int \chi'(t) u_{\varepsilon,h}(t,\cdot)^p \psi^2 d\mu \, dt \nonumber \\
& \leq 
 - p \int_J \frac{1}{h} \int_t^{t+h} \e_s(u(s,\cdot),u_{\varepsilon,h}(t,\cdot)^{p-1} \psi^2 \chi(t) ) ds \, dt 
 + \int_J \int \chi' u_{\varepsilon,h}^p  \psi^2 d\mu \, dt \nonumber \\
& \leq 
 - p \int_J \frac{1}{h} \int_t^{t+h} \e_s(u(s,\cdot),[u_{\varepsilon,h}(t,\cdot)^{p-1} - u_{\varepsilon}(t,\cdot)^{p-1}] \psi^2) ds \, \chi(t) dt  \nonumber \\
& \quad  - p \int_J \frac{1}{h} \int_t^{t+h} \e_s(u(s,\cdot) - u(t,\cdot),u_{\varepsilon}(t,\cdot)^{p-1} \psi^2 ) ds \, \chi(t) dt  \nonumber \\
& \quad  - p \int_J \frac{1}{h} \int_t^{t+h} \e_s(u(t,\cdot),u_{\varepsilon}(t,\cdot)^{p-1} \psi^2 ) ds \, \chi(t)dt \label{eq:u_e steklov p<0} \\
& \quad  + \int_J \int \chi' u_{\varepsilon,h}^p  \psi^2 d\mu \, dt. \nonumber 
\end{align}

In the case $p \in (0,1-\eta)$, we choose a different function $\chi:\R \to \R$. Namely, we let $\chi$ be such that $0 \leq \chi \leq 1$, $\chi = 0$ in 
$(a + \sigma\tau r^2, \infty)$, $\chi = 1$ in $(-\infty, a + \sigma' \tau r^2)$, and $0 \geq \chi' \geq - 2/(\omega \tau r^2)$. 
Further, let $s_1 = a + \frac{1+\sigma}{2}\tau r^2$.
We obtain that for almost every $t_1 \in I^+_{\sigma'}$ and $J=(t_1,s_1)$,
\begin{align}
& \quad  - \int u_{\varepsilon,h}(t_1,\cdot)^p  \psi^2 d\mu  \nonumber \\
& = - \int u_{\varepsilon,h}(t_1,\cdot)^p \psi^2 \chi (t_1) d\mu  + \int u_{\varepsilon,h}(s_1,\cdot)^p  \psi^2 \chi(s_0) d\mu  \nonumber \\
& = 
- p \int_J \frac{1}{h} \big( u(t+h, \cdot) - u(t,\cdot) \big) u_{\varepsilon,h}(t,\cdot)^{p-1} \psi^2 \chi(t)  d\mu \, dt 
       - \int_J \int \chi'(t) u_{\varepsilon,h}(t,\cdot)^p \psi^2 d\mu \, dt \nonumber \\
& \leq 
  p \int_J \frac{1}{h} \int_t^{t+h} \e_s(u(s,\cdot),u_{\varepsilon,h}(t,\cdot)^{p-1} \psi^2 \chi(t) ) ds \, dt 
 - \int_J \int \chi' u_{\varepsilon,h}^p  \psi^2 d\mu \, dt \nonumber \\
& \leq 
  p \int_J \frac{1}{h} \int_t^{t+h} \e_s(u(s,\cdot),[u_{\varepsilon,h}(t,\cdot)^{p-1} - u_{\varepsilon}(t,\cdot)^{p-1}] \psi^2) ds \, \chi(t) dt  \nonumber \\
& \quad  + p \int_J \frac{1}{h} \int_t^{t+h} \e_s(u(s,\cdot) - u(t,\cdot),u_{\varepsilon}(t,\cdot)^{p-1} \psi^2 ) ds \, \chi(t) dt  \nonumber \\
& \quad  + p \int_J \frac{1}{h} \int_t^{t+h} \e_s(u(t,\cdot),u_{\varepsilon}(t,\cdot)^{p-1} \psi^2 ) ds \, \chi(t)dt \label{eq:u_e steklov p>0} \\
& \quad  - \int_J \int \chi' u_{\varepsilon,h}^p  \psi^2 d\mu \, dt. \nonumber 
\end{align}

\paragraph*{STEP 1.}
The following estimates hold in both cases, $p<0$ and $p \in (0,1-\eta)$.
By Corollary \ref{cor:skew identity with p for u_eps}, we can decompose the integrand in \eqref{eq:u_e steklov p<0} and \eqref{eq:u_e steklov p>0} as 
\begin{align*}
|p| \e_s(u,u_{\varepsilon}^{p-1} \psi^2 )
& = 
 |p| \e^{\mbox{\tiny{s}}}_s(u_{\varepsilon},u_{\varepsilon}^{p-1}  \psi^2 ) 
 + |p| \e^{\mbox{\tiny{sym}}}_s(u_{\varepsilon}^p \psi^2 ,1) 
 + |p| \e^{\mbox{\tiny{skew}}}_s(u_{\varepsilon},u_{\varepsilon}^{p-1}\psi^2 ) \nonumber \\
& \quad 
- |p| \e_s(\varepsilon,u_{\varepsilon}^{p-1} \psi^2 ) \\
& = \e^{\mbox{\tiny{s}}}_s(u_{\varepsilon},u_{\varepsilon}^{p-1}  \psi^2 ) 
+ |p| \left( \e^{\mbox{\tiny{sym}}}_s(u_{\varepsilon}^{p} \psi^2,1) 
   - \varepsilon \e^{\mbox{\tiny{sym}}}_s( u_{\varepsilon}^{p-1} \psi^2,1)  \right) \\
& \quad + 2 |\e^{\mbox{\tiny{skew}}}_s(u_{\varepsilon}^{p/2}, u_{\varepsilon}^{p/2}\psi^2)| 
+ |2-p| |\e^{\mbox{\tiny{skew}}}_s(u_{\varepsilon}^p \psi^2,1)| 
+ |p||\e^{\mbox{\tiny{skew}}}_s(\varepsilon, u_{\varepsilon}^{p-1} \psi^2)|.
\end{align*}
We will estimate the parts on the right hand side separately.
By Assumption 1(iii), Remark \ref{rem:assumptions}(i), and \eqref{eq:Gamma(fg)}, we obtain that for any $k_1, k_2, k_3 > 0$,
\begin{align} \label{eq:e^skew estimate p<1}
& 2 |\e^{\mbox{\tiny{skew}}}_s(u_{\varepsilon}^{p/2}, u_{\varepsilon}^{p/2}\psi^2)| 
+ |2-p| |\e^{\mbox{\tiny{skew}}}_s(u_{\varepsilon}^p \psi^2,1)| 
+ |p||\e^{\mbox{\tiny{skew}}}_s(\varepsilon, u_{\varepsilon}^{p-1} \psi^2)| \nonumber \\
\leq & \left( 4k_1 + \frac{2}{k_2} + \frac{2}{k_3} \right) \int u_{\varepsilon}^p  d\Gamma(\psi,\psi) \nonumber \\
& + \left( \frac{C_4}{k_1} + \frac{2}{k_2} + \frac{2}{k_3} \frac{(p-1)^2}{p^2} \right) \int \psi^2  d\Gamma(u_{\varepsilon}^{p/2},u_{\varepsilon}^{p/2}) \\
& +  C_5 \left( \frac{1}{k_1} + |2-p|^2 k_2 + |p|^2 k_3 \right) \int u_{\varepsilon}^p  \psi^2 d\mu. \nonumber
\end{align}
By Assumption 1(ii), Remark \ref{rem:assumptions}(i), and \eqref{eq:Gamma(fg)}, we have for any $k_4>0$,
\begin{align} \label{eq:e^sym estimate p<1}
|p| \left( \e^{\mbox{\tiny{sym}}}_s(u_{\varepsilon}^{p} \psi^2,1) 
              - \varepsilon \e^{\mbox{\tiny{sym}}}_s( u_{\varepsilon}^{p-1} \psi^2,1)  \right)
& \leq  \frac{4}{k_4} \int u_{\varepsilon}^p  d\Gamma(\psi,\psi) \nonumber \\
& \quad + \frac{2}{k_4} \left( 1 + \frac{(p-1)^2}{p^2} \right) \int \psi^2  d\Gamma(u_{\varepsilon}^{p/2},u_{\varepsilon}^{p/2}) \\
& \quad + 2 ( C_2 k_4 p^2 + 2C_3^{1/2} |p|) \int u_{\varepsilon}^p  \psi^2 d\mu. \nonumber
\end{align}
By Lemma \ref{lem:SUP sym2}, Assumption \ref{as:e_t}(i) and Remark \ref{rem:assumptions}(i), we have for any $k_5 \geq 1$,
\begin{align} \label{eq:e_t s}
|p| \e^{\mbox{\tiny{s}}}_s(u_{\varepsilon},u_{\varepsilon}^{p-1}  \psi^2) 
 \leq & 4k_5 C_1 \frac{|p|}{|p-1|} \int u_{\varepsilon}^p  d\Gamma(\psi,\psi)  \nonumber \\
       & - \left(1 - \frac{1}{k_5}\right) \frac{|p-1|}{|p|} \frac{4}{C_1} \int \psi^2  d\Gamma(u_{\varepsilon}^{p/2},u_{\varepsilon}^{p/2}).
\end{align}

In view of the above estimates, the proof of \eqref{eq:supsol p<0} and \eqref{eq:supsol 0<p<1} can now be finished similarly to the proof of Lemma \ref{lem:estimate subsol p<2}, by using the estimates
 \eqref{eq:e^skew estimate p<1}, \eqref{eq:e^sym estimate p<1}, and \eqref{eq:e_t s} in Step 1, and by letting $h \to 0$ in Step 2. 
In Step 3, we take the supremum over all $t_0 \in I^-_{\sigma'}$ in the case $p<0$, and the supremum over all $t_1 \in I^+_{\sigma'}$ in the case $p \in (0,1-\eta)$. We finish the proof by letting $\varepsilon \to 0$.
\end{proof}

\subsection{Proof of Lemma \ref{lem:5.4.1}}

\begin{proof}
We present the proof only for the case $\tau = 1$. The proof of the general case is very similar. Let $\psi(z) = \max\{0,(1 - d(x,z)) / r' \} \in \F_{\mbox{\tiny{c}}}(B) \cap C_{\mbox{\tiny{c}}}(B)$, 
where $r'>0$ is slightly smaller than $r$. Note that $d\Gamma(\psi,\psi) \leq c r^{-2} d\mu$.
Let $u_{\varepsilon,h}$ be the Steklov average of $u_{\varepsilon}$. Then, using the fact that the Steklov average has a strong time-derivative and the assumption that $u$ is local very weak supersolution, we obtain
\begin{align} \label{eq:will let h to 0 LHS}
 & - \frac{d}{dt} \int \log u_{\varepsilon,h}(t) \psi^2 d\mu \\
& = 
- \frac{1}{h} \int [ u(t+h) - u(t) ] \frac{1}{u_{\varepsilon,h}(t)} \psi^2 d\mu \nonumber \\
& \leq 
 \frac{1}{h} \int_t^{t+h} \e_s \left( u(s),\frac{1}{u_{\varepsilon,h}(t)} \psi^2 \right) ds \nonumber \\
& =  \label{eq:will let h to 0 i}
 \frac{1}{h} \int_t^{t+h} \e_s \left( u(s),\frac{1}{u_{\varepsilon,h}(t)} \psi^2  - \frac{1}{u_{\varepsilon}(t)} \psi^2 \right) ds \\
& \quad \label{eq:will let h to 0 ii}
+  \frac{1}{h} \int_t^{t+h} \e_s \left( u(s) - u(t),\frac{1}{u_{\varepsilon}(t)} \psi^2 \right) ds \\
& \quad \label{eq:will let h to 0 iii}
+  \frac{1}{h} \int_t^{t+h} \e_s \left( u_{\varepsilon}(t),\frac{1}{u_{\varepsilon}(t)} \psi^2 \right) 
                                 -  \e_s \left( \varepsilon,\frac{1}{u_{\varepsilon}(t)} \psi^2 \right)ds  \\
& = \label{eq:will let h to 0 RHS}
I_1(h) + I_2(h) + I_3(h).
\end{align}
We will see below that $I_1(h)$ and $I_2(h)$ tend to $0$ in the appropriate sense as $h \to 0$. We consider $I_3(h)$.
For almost every $s \in [t,t+h]$, we have by the chain rule for $\Gamma_s$, the Cauchy-Schwarz inequality \eqref{eq:CS} and Assumption 1(i),
\begin{align*}
\e^{\mbox{\tiny{s}}}_s(u_{\varepsilon}(t),u_{\varepsilon}^{-1}(t) \psi^2)
& =  \int 2\psi d\Gamma_s(\log u_{\varepsilon}(t),\psi) 
   - \int \psi^2 d\Gamma_s(\log u_{\varepsilon}(t),\log u_{\varepsilon}(t)) \\
&\leq  k' \int  d\Gamma(\psi,\psi) 
     -\frac{1}{k'} \int \psi^2 d\Gamma(\log u_{\varepsilon}(t),\log u_{\varepsilon}(t)),
\end{align*}
for some constant $k' > 1$.
By Lemma \ref{lem:approximation nonsym} , Assumption \ref{as:p=0}, Assumption \ref{as:e_t}(ii) and Lemma \ref{lem:e(1,)}, we have
\begin{align*}
& \quad \e_s^{\mbox{\tiny{skew}}}(u_{\varepsilon},u_{\varepsilon}^{-1}(t) \psi^2)
      + \e_s^{\mbox{\tiny{sym}}}(\psi^2,1) 
      - \e_s(\varepsilon,u_{\varepsilon}^{-1}(t) \psi^2) \nonumber \\
& \leq  
k'' \int d\Gamma(\psi,\psi) + \frac{1}{k''} \int \psi^2 d\Gamma(\log u_{\varepsilon}(t), \log u_{\varepsilon}(t)) 
 + k'' C_8 \int \psi^2 d\mu,
\end{align*}
for sufficiently large $k'' > 1$.
Thus, for sufficiently large $k > 1$, we obtain that
\begin{align} \label{eq:lem5.4.1a}
 & \quad - \frac{d}{dt} \int \log u_{\varepsilon,h}(t) \psi^2 d\mu 
+ \frac{1}{k} \int \psi^2 d\Gamma(\log u_{\varepsilon}(t), \log u_{\varepsilon}(t))  \nonumber \\
& \leq 
I_1(h) + I_2(h) + \left( k \int d\Gamma(\psi,\psi) + k C_8 \int \psi^2 d\mu \right).
\end{align}

Let 
\[ W(t) :=  -\frac{\int \log u_{\varepsilon}(t) \psi^2 d\mu }{ \int \psi^2 d\mu} \quad \mbox{ and} \quad W_h(t) :=  -\frac{\int \log u_{\varepsilon,h}(t) \psi^2 d\mu }{ \int \psi^2 d\mu}. \]
By the weighted Poincar\'e inequality of Theorem \ref{thm:weighted PI}, it holds for a.e. $t \in I$ that
\begin{align*}
\int |-\log u_{\varepsilon}(t) - W(t)|^2 \psi^2 d\mu  \leq  C_{\mbox{\tiny{wPI}}} \, r^2 \int \psi^2 d\Gamma(\log u_{\varepsilon}(t), \log u_{\varepsilon}(t)).
\end{align*}
This and \eqref{eq:lem5.4.1a} give
\begin{align*} 
& \quad  \frac{d}{dt} W_h(t) + \frac{1}{C r^2 \mu(B)}  \int_{\delta B} |-\log u_{\varepsilon}(t) - W(t)|^2 \psi^2 d\mu \\
& \leq 
\frac{I_1(h) + I_2(h)}{\int \psi^2 d\mu} + (C' r^{-2} + k C_8),
\end{align*}
for some constants $C,C'>0$.
Writing
\begin{align*}
\overline{w}(t,z) & = -\log u_{\varepsilon}(t,z) - (C'r^{-2} + k C_8) (t-a), \\
\overline{W}(t) & = W(t) -  (C'r^{-2} + k C_8) (t-a), \\
\overline{W_h}(t) & = W_h(t) -  (C'r^{-2} + k C_8) (t-a),
\end{align*}
we obtain for $t > a$ that
\begin{align} \label{eq:5.4.2} 
\frac{d}{dt} \overline{W_h}(t)  + \frac{1}{C r^2 \mu(B)} \int_{\delta B} |\overline{w} - \overline{W}|^2 \psi^2 d\mu  
\leq \frac{I_1(h) + I_2(h)}{\int \psi^2 d\mu}.
\end{align}
For $\lambda > 0$, set
\begin{align*}
\Omega^-_t(\lambda) & = \{ z \in \delta B : \overline{w}(t,z) < - \lambda + \overline{W}(a) \}, \\
\Omega^+_t(\lambda) & = \{ z \in \delta B : \overline{w}(t,z) > \lambda + \overline{W}(a) \}.
\end{align*}
Then, for any $a < t$,
\begin{align} \label{eq:W und lambda} 
\overline{w}(t,z) - \overline{W}(t) > \lambda + \overline{W}(a) - \overline{W}(t) \geq \lambda
\end{align}
in $\Omega^+_{t}(\lambda)$, because  $W(a) - W(t) \geq 0$ (to see this, integrate \eqref{eq:5.4.2}  and then let $h \to 0$).
Applying \eqref{eq:W und lambda} in the inquality \eqref{eq:5.4.2}, we get
\[ \frac{d}{dt} \overline{W_h}(t) + \frac{1}{C r^2 \mu(B)} |\lambda + \overline{W}(a) - \overline{W}(t)|^2 \mu(\Omega^+_{t}(\lambda)) 
\leq \frac{I_1(h) + I_2(h)}{\int \psi^2 d\mu}.  \]
Dividing by $|\lambda + \overline{W_h}(a) - \overline{W_h}(t)|^2$, we can rewrite this inequality as
\begin{align*}
& \quad 
\frac{d}{dt} |\lambda + \overline{W_h}(a) - \overline{W_h}(t)|^{-1} 
+ \frac{|\lambda + \overline{W}(a) - \overline{W}(t)|^2}{|\lambda + \overline{W_h}(a) - \overline{W_h}(t)|^2} \frac{\mu(\Omega^+_{t}(\lambda))}{C r^2 \mu(B)} \\
& \leq  
\frac{I_1(h) + I_2(h)}{\int \psi^2 d\mu} \frac{1}{|\lambda + \overline{W_h}(a) - \overline{W_h}(t)|^2},
\end{align*}
or, equivalently,
\begin{align} \label{eq:mu(Omega+)}
  \mu(\Omega^+_t(\lambda)) 
& \leq 
C r^2 \mu(B) \bigg( 
 - \frac{d}{dt} |\lambda + \overline{W_h}(a) - \overline{W_h}(t)|^{-1} \frac{|\lambda + \overline{W_h}(a) - \overline{W_h}(t)|^2}{|\lambda + \overline{W}(a) - \overline{W}(t)|^2} \\
 & \quad
  + \frac{I_1(h) + I_2(h)}{\int \psi^2 d\mu} \frac{1}{|\lambda + \overline{W}(a) - \overline{W}(t)|^2}\bigg).
\end{align}
Next, we want to bound $\frac{|\lambda + \overline{W_h}(a) - \overline{W_h}(t)|^2}{|\lambda + \overline{W}(a) - \overline{W}(t)|^2}$
by $(1 + \epsilon(h))$ for some small $\epsilon(h) >0$ that tends to $0$ as $h \to 0$. By the triangle inequality,
\begin{align*}
 \frac{|\lambda + \overline{W_h}(a) - \overline{W_h}(t)|^2}{|\lambda + \overline{W}(a) - \overline{W}(t)|^2}
& \leq \frac{\left( |\lambda + \overline{W}(a) - \overline{W}(t)| + |\overline{W}(a) - \overline{W_h}(a)| + |\overline{W}(t) - \overline{W_h}(t)|\right)^2}{|\lambda + \overline{W}(a) - \overline{W}(t)|^2}.
\end{align*}
We show that $ |\overline{W}(t) - \overline{W_h}(t)| \to 0$ uniformly in $t$ as $h \to 0$. Since $W(t)$ is decreasing in $t$, we have $W_h(t) \leq W(t)$. 
Hence, by Jensen's inequality and the Cauchy-Schwarz inequality,
\begin{align*}
 |\overline{W}(t) - \overline{W_h}(t)| 
& = W(t) - W_h(t) \\
& = \frac{1}{\int \psi^2 d\mu} \left( - \int \log u_{\epsilon}(t) \psi^2 d\mu + \int \log u_{\epsilon,h}(t) \psi^2 d\mu \right) \\
& \leq \frac{1}{\int \psi^2 d\mu} \left( - \int \log u_{\epsilon}(t) \psi^2 d\mu + \int \frac{1}{h} \int_t^{t+h} \log u_{\epsilon}(s) ds \, \psi^2 d\mu \right) \\
& \leq \frac{1}{\left(\int \psi^2 d\mu\right)^{1/2}}  \left( \int \left( - \log u_{\epsilon}(t) \psi  +  \frac{1}{h} \int_t^{t+h} \log u_{\epsilon}(s) \psi ds \right)^2 d\mu \right)^{1/2}.
\end{align*}
By Corollary \ref{cor:steklov convergence uniformly in t}, the right hand side converges to $0$ uniformly in $t$, as $h \to 0$.

Integrating \eqref{eq:mu(Omega+)} from $a$ to $a+\eta r^2$, we obtain
\begin{align*}
& \quad
 \overline{\mu}\left( \big\{(t,z) \in K_+ : \overline{w}(t,z) > \lambda + \overline{W}(a) \big\} \right) \\
& \leq 
C r^2 \mu(B) \left( \lambda^{-1} (1+ \epsilon(h)) + \frac{1}{ \lambda^2 \int \psi^2 d\mu} \int_a^{a+\eta r^2} I_1(h) + I_2(h) dt \right)
\end{align*} 
Letting $h \to 0$, we find that $\int_a^{a+\eta r^2} I_1(h)dt \to 0$ by Assumption 0(i), Cauchy-Schwarz inequality, and Corollary \ref{cor:steklov convergence in L^2}. 
Similarly, we find that $\int_a^{a+\eta r^2} I_2(h)dt \to 0$ by Assumption 0(i), triangle inequality, Cauchy-Schwarz inequality, \eqref{eq:chain rule for Gamma}, the local boundedness of $u$, the dominated convergence theorem, and Corollary \ref{cor:steklov convergence in L^2}.

Hence, we get
\[ \overline{\mu}\left( \big\{(t,z) \in K_+ : \overline{w}(t,z) > \lambda + \overline{W}(a) \big\} \right) 
\leq C r^2 \mu(B) \lambda^{-1}, \]
and hence
\begin{align*}
& \overline{\mu}\left( \big\{(t,z) \in K_+ : \log u_{\varepsilon}(t,z) + (C' r^{-2} + k' C_8) (t-a) < - \lambda - \overline{W}(a) \big\} \right)  \\
\leq & \ C r^2 \mu(B) \lambda^{-1}.
\end{align*}
Finally,
\begin{align*}
& \quad \ \overline{\mu}\left( \big\{(t,z) \in K_+ : \log u_{\varepsilon}(t,z) < - \lambda - \overline{W}(a) \big\} \right)  \\
& \leq \overline{\mu}\left( \big\{(t,z) \in K_+ : \log u_{\varepsilon}(t,z) + (C' + k' C_8 r^2)r^{-2} (t-a) < - \lambda/2 - \overline{W}(a) \big\} \right) \\
& \quad \ + \overline{\mu}\left( \big\{(t,z) \in K_+ : (C' + k' C_8 r^2) r^{-2} (t-a) > \lambda/2 \big\} \right) \\
& \leq C'' (1+C_8r^2) r^2 \mu(B) \lambda^{-1}.
\end{align*}
This proves the first inequality in Lemma \ref{lem:5.4.1}. By a similar argument, using $\Omega^-_t$ instead of $\Omega^+_t$,
and the average $\frac{1}{h} \int_{t-h}^t f(s) ds$ intead of the usual Steklov average $f_h$, we obtain the second inequality.
\end{proof}

\section{Appendix}

\subsection{Remarks on the novelty and relevance of Assumptions 0, \ref{as:e_t} and \ref{as:p=0}} \label{ssec:Sturm}
 
The paper \cite{SturmII} outlined a proof of mean value estimates in the context of non-symmetric Dirichlet forms. In this subsection we explain the novelty and relevance of our Assumptions 0, \ref{as:e_t} and \ref{as:p=0} in comparison to the assumptions made in \cite{SturmII}.

The parabolic Harnack inequality for non-symmetric Dirichlet forms is not treated in Sturm's work. The assumptions in \cite{SturmII} on non-symmetric Dirichlet forms appear to be insufficient to prove the parabolic Harnack inequality, even under the additional assumption that weak solutions to the heat equation are locally bounded. Indeed, we use Assumption \ref{as:p=0} to obtain the parabolic Harnack inequality from mean value estimates. An assumption of this sort was not made in \cite{SturmII}.

Sturm \cite{SturmII} considers a reference form $(\e,\F)$ and a family of local regular Dirichlet forms $(\e_t)_{t \in \R}$ with domain $D(\e_t) = \F$. He assumes that, for all $f,g \in \F$, the map $t \mapsto \e_t(f,g)$ is measurable, and Assumption 0(i) is satisfied uniformly in $t$. The assumption that each $\e_t$ is a local Dirichlet form is stronger than Assumption 0(iv). In addition, Sturm makes two assumptions (UP) with $\gamma = 0$ and (SUP), which read as follows.

\textbf{Uniform parabolicity (UP).} { \em 
There exist constants $K,k \in (0,\infty)$ and $\gamma \in [0,\infty)$ such that
\begin{align*}
- \e_t(u,v) \le K [ \e(\sqrt{uv},\sqrt{uv}) - \e(u,v) ] - k \e(\sqrt{uv},\sqrt{uv}) + \gamma \int_X u v \, d\mu,
\end{align*}
for all $t \in \R$ and all $u,v \in \F$ with $uv \ge 0$ and $\sqrt{uv} \in \F$.
}

In case each $\e_t$ is symmetric strongly local, (UP) with $\gamma =0$ is equivalent to Assumption 0(iii).
There seems to be no obvious relation between (UP) and Assumption 0(ii). Nevertheless, (UP) implies \eqref{eq:lower bounded form}, so existence of weak solutions to the heat equation is guaranteed under (UP) (cf.~Remark \ref{rem:lower bounded form}).

\textbf{Strong uniform parabolicity (SUP).} {\em
There exists a constant $\kappa = \kappa(Y) \ge 1$ such that
\begin{align} \label{eq:SUP}
-\frac{p-1}{2} \e_t(u, u^{p-1} \phi^2) 
 \le & \kappa \int u^p d\Gamma(\phi,\phi) 
        - \frac{1}{\kappa} \left(1-\frac{1}{p}\right)^2 \int \phi^2  d\Gamma(u^{p/2},u^{p/2}),
\end{align}
for all $p \in \R$, all non-negative $u \in \F_{\mbox{\tiny{loc}}}(Y)$ with $u+u^{-1} \in L^{\infty}(Y,\mu)$, and all $\phi \in \F_{\mbox{\tiny{c}}}(Y) \cap L^{\infty}(Y,\mu)$, for $Y=X$ or at least for sufficiently many open sets $Y \subset X$.
}

For simplicity, let us discuss here the case $Y=X$. We remark that Assumption \ref{as:e_t} and Assumption \ref{as:p=0} can be localized to sufficiently many open subsets $Y \subset X$ and still yield most of our main results  (see \cite{LierlPHIf}). We chose to avoid this level of generality in the present paper for the benefit of a clearer presentation.

Under Assumption 0 and Assumption \ref{as:e_t}, there exists a constant $\kappa \ge 1$ such that
\begin{align} \label{eq:our SUP}
\begin{split}
-\frac{p-1}{2} \e_t(u, u^{p-1} \phi^2) 
 \leq & \kappa \int u^p d\Gamma(\phi,\phi) 
        - \frac{1}{\kappa} \left( 1 - \frac{1}{p} \right)^2 \int \phi^2  d\Gamma(u^{p/2},u^{p/2}) \\
       & + \kappa (p^2+1) \int u^p \phi^2 d\mu,
 \end{split}
\end{align}
for all $0 \neq p \in \R$, $t \in \R$, all non-negative $u \in \F_{\mbox{\tiny{loc}}}(Y)$ with $u+u^{-1} \in L^{\infty}(Y,\mu)$, and all $\phi \in \F_{\mbox{\tiny{c}}}(Y) \cap L^{\infty}(Y,\mu)$.
This follows from Lemma \ref{lem:SUP sym2} and Proposition \ref{prop:skew identity with p}. Notice the restriction $p \ne 0$, we come back to this issue below by referring to Assumption \ref{as:p=0}. 

Estimate \eqref{eq:our SUP} is the same as \eqref{eq:SUP} except for the zero order term  $(p^2+1) \int u^p \phi^2 d\mu$. Without any zero order term, \eqref{eq:SUP} can hold only for purely second order forms $\e_t$. The bilinear form of Example \ref{ex:Euclidean} satisfies \eqref{eq:SUP} provided that $b_i = d_i = c = 0$ for all $i$.  
It was argued at the end of \cite[Subsection 2.2]{SturmII} that lower order terms (i.e. nonzero coefficients $b_i, d_i, c$) could be included in the consideration by replacing $\e_t$ by $\e_t \pm \alpha(\cdot,\cdot)$ for some constant $\alpha$. This is true for coefficients $b_i$ and $c$. A nonzero coefficient $d_i$ creates a zero order term with prefactor $(p^2+1)$ rather than the prefactor $\pm (p-1)$ claimed in \cite[p.~298]{SturmII}.

A less minor problem with assumption (SUP) is that it does not appear to be sufficient to obtain mean value estimates, except for weak solutions whose local boundedness is known a priori. 
Sturm \cite{SturmII} suggested to approximate weak solutions by bounded functions $(u_n)$, similar to the method in \cite{AS67}. 
A correct way of approximating $\e_t(u,u^{p-1} \psi^2)$ is not by $\e_t(u,u_n^{p-1} \psi^2)$ as claimed in \cite{SturmII} but by $\e_t(u,u u_n^{p-2} \psi^2)$ as, e.g., in the work of Aronson and Serrin \cite{AS67}.
 A technical difficulty then arises from the fact that $x \mapsto x (x \wedge n)^{p-2}$ is only a $\mathcal C^1$-function, while weak time-derivatives lack a $\mathcal C^1$-chain rule (see Proposition \ref{prop:time chain rule} for the $\mathcal C^2$-chain rule). Aronson-Serrin resolved this issue by a smoothing argument that involves convolution with some kernel. It is not outlined in \cite{SturmII} how this argument would translate to the metric space setting. The method of Steklov averages employed in the present paper is an alternative way to handle this difficulty.

\newpage

\subsection{Extensions of the bilinear forms $\l_*$ and $\r_*$}

\begin{proposition} \label{prop:extending L}
Suppose $(\e_*,\F)$ satisfies {\em Assumption 0(i)}, and for all $f,g \in \F \cap L^{\infty}(X)$ with $fg \in \F_{\mbox{\em \tiny c}}$, 
\begin{align} \label{eq:weak sector type condition}
 |\e_*(fg,1)| + |\e_*(1,fg)|\le C_* \|f\|_\F\|g\|_\F.
\end{align}
Then the maps $\mathcal K_l:(f,g)\mapsto \e_*(fg,1)$ and
$\mathcal K_r:(f,g)\mapsto \e_*(1,fg)$, extend continuously from $\big( \F\cap \mathcal C_{\mbox{\em \tiny{c}}}(X) \big) \times \big(\F\cap \mathcal C_{\mbox{\em \tiny{c}}}(X) \big)$ to $\F\times\F$, and satisfy 
\[ \mathcal K_l(f,g) \leq C_* ||f||_{\F} ||g||_{\F}, \quad \mathcal K_r(f,g) \leq C_* ||f||_{\F} ||g||_{\F},
\]
for all $f,g \in \F$. Moreover, for any $f,g \in \F \cap L^{\infty}(X,\mu)$  with $fg \in \F_{\mbox{\em{\tiny{c}}}}$,
\begin{align} \label{eq:K_l and K_r weak}
 \e_*(fg,1) = K_l(f,g), \quad \e_*(1,fg) = K_r(f,g).
 \end{align}

As a consequence, the bilinear forms $\l_*$ and $\r_*$ extend continuously to $\F\times\F$, and the extensions satisfy
\begin{align} \label{eq:extension of l}
\l_*(u,v) & = \frac{1}{4} \big[ \e_*(uv,1) - \e_*(1,uv) + 
\e_*(u,v) - \e_*(v,u) \big],\\
\r_*(u,v) & = \frac{1}{4} \big[ \e_*(1,uv) - \e_*(uv,1) + \e_*(u,v) - \e_*(v,u) \big]
= - \l(v,u), \nonumber
\end{align}
for any $u,v \in \F \cap L^{\infty}(X,\mu)$ with $uv \in \F_{\mbox{\em{\tiny{c}}}}$. 
If in addition $\l_*$ satisfies the Leibniz rule of Definition \ref{def:leibniz rule}(i), 
then, for any $u,v,f \in \F \cap L^{\infty}(X,\mu)$ with $uvf \in \F_{\mbox{\em \tiny{c}}}$, we have
\begin{align*}
 \l_*(uf,v) = \l_*(u,fv) + \l_*(f,uv),
 \end{align*}
and, for any $v \in \F_{\mbox{\em \tiny{c}}}$, any $u_1, u_2, \ldots, u_m \in \F \cap L^{\infty}(X,\mu)$ and $u = (u_1, \ldots, u_m)$, 
and for any $\Phi \in \mathcal C^2(\R^m)$, we have the chain rule
\begin{align*}
\l_*(\Phi(u),v) = \sum_{i=1}^{m} \l_*(u_i, \Phi_{x_i}(u) v).
\end{align*}
\end{proposition}
\begin{proof}
For $f,g \in \F$, consider sequences $(f_n)$, $(g_n)$ in $\F\cap \mathcal C_{\mbox{\tiny{c}}}(X)$ which converge in $\F$ to $f$ and $g$, respectively.
Using the fact that $f_n g_m \in \mathcal F_{\mbox{\tiny{c}}}$ and \eqref{eq:weak sector type condition}, we see that $\e_*(f_n g_n,1)$ is a Cauchy sequence. If $(\hat f_n),(\hat g_n) \subset \F\cap \mathcal C_{\mbox{\tiny{c}}}(X)$ are other sequences 
converging in $\F$ to $f$ and $g$, respectively, then, constructing another sequence $\check f_{2n}:=f_n$, $\check f_{2n+1}:=\hat f_n$ and analogously constructing $(\check g_n)$,
we see that the limit $\mathcal K_l(f,g)$ of the Cauchy sequence $\e_*(f_ng_n,1)$ is independent of the choice of the sequences $(f_n), (g_n) \subset \F\cap \mathcal C_{\mbox{\tiny{c}}}(X)$. This shows that $\mathcal K_l$ extends continuously to $\F \times \F$ and satisfies the sector type condition $\mathcal K_l(f,g) \leq C_* ||f||_{\F} ||g||_{\F}$ for all $f,g \in \F$.

Next, we show that $K_l(f,g) = \e_*(fg,1)$ for any $g\in \F\cap \mathcal C_{\mbox{\tiny{c}}}(X)$ and $f\in \F \cap L^\infty(X)$. Note that $fg\in \F_{\mbox{\tiny{c}}}$. Let $(f_n)$ be a sequence in $\F\cap \mathcal C_{\mbox{\tiny{c}}}(X)$ that converges to $f$ in $\F$. 
Then $f_n g \in \F_{\mbox{\tiny{c}}}$, $\e_*(f_ng,1)=\mathcal K_l(f_n,g)$, and $|\e_*((f-f_n)g,1)|\le C_*\|f-f_n\|_\F\|g\|_\F$ by \eqref{eq:weak sector type condition}. This proves that $\e_*(fg,1)=\mathcal K_l(f,g)$ in this case.

Assume know that $f,g$ are such that $f,g\in \F\cap L^\infty(X)$ 
with $fg\in \F_{\mbox{\tiny{c}}}$. Let $(g_n)$ be a sequence in $\F\cap \mathcal C_{\mbox{\tiny{c}}}(X)$ 
which converges to $g$ in $\mathcal F$. By the above argument, we obtain that $\e_*(f g_n,1)=\mathcal K_l(f,g_n)$, and,  by \eqref{eq:weak sector type condition}, $|\e_*(f(g-g_n),1)|\le C_*\|f\|_\F\|g-g_n\|_\F$. 
Letting $n \to \infty$, it follows that
$$\e_*(fg,1)=\mathcal K_l(f,g)$$ as desired.

The product rule and of the chain rule follow from Definition \ref{def:leibniz rule} by a simple approximation argument together with Assumption 0(i) and \eqref{eq:extension of l}.
\end{proof}

Under the stronger hypothesis \eqref{A2} we can relax the boundedness condition on $f,g$ in \eqref{eq:K_l and K_r weak}.
\begin{proposition} \label{pro-A2}
Suppose $(\e_*,\F)$ satisfies {\em Assumption 0(i)}.
Assume that
\begin{equation}\label{A2}
\forall\,f,g\in \mathcal F\mbox{ with } fg\in \F_{\mbox{\em \tiny{c}}}, \quad
|\e_*(fg,1)|+|\e_*(1,fg)|\le C_* \|f\|_\F \|g\|_\F.
\end{equation}
Then 
$$\mathcal K_l(f,g)= \e_*(fg,1) \mbox{ and } 
\mathcal K_r(f,g)=\e_*(1,fg)$$
under either one of the following conditions: 
\begin{itemize}
\item $f,g\in \F$ with 
$fg\in \mathcal F_{\mbox{\em \tiny{c}}}$, $f\ge 0$ and $g\ge 0$.
\item  for all $f \in \F$, $g\in \F\cap L^\infty(X,\mu)$, 
$fg\in \F_{\mbox{\em \tiny{c}}}$ and $g\ge 0$.
\end{itemize}
\end{proposition}
\begin{proof} Assume that $f\in \F$ and $g\in \F\cap L^\infty(X,\mu)$
with $fg\in \F_{\mbox{\tiny{c}}}$ and $g\ge 0$. Let $f_n:= (f\wedge n) \vee (-n)$. Observe 
that  $f_n=f +(f+n)^- - (f-n)^+$ and, consequently, 
$$f_ng=fg +(f+n)^-g-(f-n)^+g= fg +(fg+ng)^- -(fg-ng)^+.$$
 Since $0\le (f-n)^+g\le |fg|$, $0\le (f+n)^-g \le |fg|$ and $fg\in \F_c$, $g\in \F$, we have
$(fg+ng)^-,(fg-ng)^+\in \F_c$ and $f_n g \in \F_{\mbox{\tiny{c}}}$. 
By Proposition \ref{prop:extending L}, it follows that
$\e_*(f_ng,1)=\mathcal K_l (f_n,g)$
and, by (\ref{A2}), 
$$|\e_*((f_n-f)g,1)|\le \|f-f_n\|_\F \|g\|_\F.$$
Hence $\e_*(fg,1)= \mathcal K_l(f,g)$ in this case.

Next, assume that $f,g\in \mathcal F$ 
with $fg\in \F_{\mbox{\tiny{c}}}$ and  $f,g\ge 0$. 
Let $g_n:=g\wedge n$. Then since $f\ge 0$, $fg_n\in \mathcal F_{\mbox{\tiny{c}}}$, and we have
$\e_*(fg_n,1)=\mathcal K_l (f,g_n)$ by what was just proved.
Since
$$|\e_*(f(g-g_n),1)|\le \|f\|_\F \|g-g_n\|_\F$$ 
we conclude that $\e_*(fg,1)=\mathcal K_l (f,g)$.
\end{proof}

\subsection{Background on local weak solutions} \label{ssec:weak solutions}
\subsubsection{Weak time-derivative} \label{ssec:weak time-derivative}

The notion of a local weak solution to the heat equation in the context of metric measure Dirichlet spaces has undergone some ambiguity. There are different ways of defining local weak solutions and these are not exactly equivalent. See, e.g., \cite{SturmII, GHL09, BGK12, EldredgeSC14}. 

The notion of local weak solutions that we use in this paper is very close to the definition in \cite{EldredgeSC14} which is similar but slightly different from the definition in \cite{SturmII}. It is consistent with the abstract existence theory in, e.g., Lions and Magenes \cite{LM72}.

For an open relatively compact time interval $I$ and a separable Hilbert space $H$, let $L^2(I \to H)$ be the Hilbert space of those functions $v: I \to H$ such that 
 \[  \Vert v \Vert_{L^2(I \to H)} = \left( \int_I \Vert v(t) \Vert_H^2 \, dt \right)^{1/2} < \infty. \]

We say that a function $v \in L^2(I \to H)$ has a distributional time-derivative (also called weak time derivative) that can be represented by a function in $L^2(I \to H)$, if there exists $v' \in L^2(I \to H)$ such that for all smooth compactly supported functions $\phi:I \to H$ we have
\begin{align} \label{eq:weak time derivative}
 \int \left(\frac{\partial}{\partial t} \phi (t), v(t) \right)_{H} dt = - \int \left( v'(t) , \phi(t) \right)_{H} dt.
 \end{align}
From \eqref{eq:weak time derivative} it is easy to see that $\frac{1}{h} ( v(t+h) - v(t))$ converges to $v'(t)$ weakly in $L^2(I \to H)$ as $h \to 0$.

Let $W^1(I \to H) \subset L^2(I \to H)$ be the Hilbert space of those functions $v: I \to H$ in $L^2(I \to H)$ whose distributional time derivative $v'$ can be represented by functions in $L^2(I \to H)$, equipped with the norm
 \[  \Vert v \Vert_{W^1(I \to H)} = \left( \int_I \Vert v(t) \Vert_H^2 + \Vert v'(t) \Vert_H^2 \, dt \right)^{1/2} < \infty. \]
Cf.~\cite{LM72, RR93}.

Identifying $L^2(X,\mu)$ with its dual space and using the dense embeddings $\F \subset L^2(X,\mu) \subset \F'$, where $\F'$ is the dual space of $\F$, we set
 \[ \F(I \times X) = L^2(I \to \F) \cap W^1(I \to \F'), \]
which is a Hilbert space with norm $\Vert v \Vert_{\F(I \times X)} = \left( \int_I || v(t) ||^2_{\F} + || v'(t)||^2_{\F'} dt \right)^{1/2}$.
A function $v \in L^2(I \to \F)$ is in $\F(I \times X)$ if and only if there exists $v' \in L^2(I \to \F')$ such that for any smooth compactly supported function $\phi:I \to \F$ it holds
\[ \int_I \int_X v(t) \frac{d}{dt} \phi(t) d\mu \, dt = - \int_I \langle v'(t),\phi(t) \rangle_{\F',\F} dt. \]

It is well-known (and easy to see since $L^2(X,\mu)$ is separable) that $L^2(I \to L^2(X,\mu))$
can be identified with  $L^2(I \times X,dt \times d\mu)$. Indeed, continuous functions with compact support in $I \times X$ are dense in both spaces and the two norms coincide on these functions.

We recall the following fact from \cite[Lemma 11.4]{RR93},
\begin{align} \label{eq:C(I to L^2)}
 \F(I \times X) \subset \mathcal C(\bar I \to L^2(X,\mu)).
\end{align}
Therefore, a function $u \in \F(I \times X)$ can be considered as a continuous path $t \mapsto u(t,\cdot)$ in $L^2(X,\mu)$. 

Let  $U \subset X$ be open. Let
 \[ \F_{\mbox{\tiny{loc}}}(I \times U) \] 
be the set of all measurable functions $u:I \times U \to \R$ such that for any open interval $J$ relatively compact in $I$, 
and any open subset $A$ relatively compact in $U$, there exists a function $u^{\sharp} \in \F(I \times X)$ such that $u^{\sharp} = u$ a.e. in $J \times A$.
Let
\begin{align*}
 \F_{\mbox{\tiny{c}}}(I \times U)  = & \{ u \in \F(I \times X): u \mbox{ has compact support in } I \times U \}.
\end{align*}



Let $\mathcal C(\overline{I} \to \F)$ be the space of continuous functions from $I$ to $\F$.

Let $\mathcal C^{\infty}((-\infty,\infty) \to \F)$ be the space of smooth functions from $(-\infty,\infty)$ to $\F$. 
Let $\mathcal C^{\infty}(\overline{I} \to \F):= \mathcal C^{\infty}((-\infty,\infty) \to \F)\big|_{\overline I}$. The following lemma is proved in \cite[Lemma 25.1]{Wlo87en}.

\begin{lemma} \label{lem:wloka density}
$\mathcal C^{\infty}(\overline{I} \to \F)$ is dense in $\F(I \times X)$.
\end{lemma}

The next proposition provides a $C^2$-chain rule for weak time-derivatives. It is important to note that we do not have a $C^1$-chain rule. A similar result, namely a $C^2$-chain rule for a different notion of weak time-derivatives, is proved in \cite[Lemma 5.1]{GHL09}.

\begin{proposition}[Chain rule for the weak time-derivative]
\label{prop:time chain rule}
Let $u \in \F(I \times X)$. Let  $v \in \F_{\mbox{\em \tiny{c}}}(X)\cap L^{\infty}(X)$.
Let $\Phi \in \mathcal C^2(\R)$ with $\Phi(0)=0$, and $\Phi'$, $\Phi''$ bounded.
Then $\Phi'(u)v \in L^2(I \to \F)$, and
 at a.e. $a,b \in I$,
\begin{align} \label{eq:time chain rule}
\int_X \Phi(u(b,\cdot)) v \, d\mu - \int_X \Phi(u(a,\cdot)) v \, d\mu
 = 
\int_a^b \left\langle \frac{\partial}{\partial t} u, \Phi'(u) v \right\rangle_{\F',\F} dt.
\end{align}
\end{proposition}

\begin{proof}
By Lemma \ref{lem:wloka density}, there exists a sequence $(f_k) \subset \mathcal C^{\infty}(\overline{I} \to \F)$ so that $f_k \to u$ in $\F(I \times X)$. 
Since each $f_k$ has a strong time derivative $\frac{\partial}{\partial t} f_k$, it holds
\begin{align} \label{eq:time chain rule u_n 1}
\int_X \Phi(f_k(b,\cdot)) v \, d\mu - \int_X \Phi(f_k(a,\cdot)) v \, d\mu
& = 
\int_a^b \int_X \left( \frac{\partial}{\partial t} \Phi(f_k) \right) v \, d\mu \, dt \nonumber \\
& = 
\int_a^b \int_X \left( \frac{\partial}{\partial t} f_k \right) \Phi'(f_k) v \, d\mu \, dt.
\end{align}
Observe that  $\Phi'(f_k) v \in L^2(I \to \F)$. 
Identifying $L^2(X,\mu)$ with its dual space and using the embeddings $\F \subset L^2(X,\mu) \subset \F'$, we can rewrite \eqref{eq:time chain rule u_n 1} as
\begin{align}  \label{eq:time chain rule u_n 2}
\int_X \Phi(f_k(b,\cdot)) v \, d\mu - \int_X \Phi(f_k(a,\cdot)) v \, d\mu
& =
\int_a^b \left\langle \frac{\partial}{\partial t} f_k, \Phi'(f_k) v \right\rangle_{\F',\F} dt.
\end{align}

We will prove \eqref{eq:time chain rule} by letting $k \to \infty$ on both sides of \eqref{eq:time chain rule u_n 2}. 
We know that $f_k \to u$ in $L^2(I \to \F)$. 
Using Proposition \ref{prop:dt-qe convergence} and passing to a subsequence, we can assume that, for almost every $t \in I$,
$\widetilde{f_k}(t,\cdot) \to \widetilde{u}(t,\cdot)$ quasi-everywhere.
By the chain rule for $\Gamma$, we have $\Phi(u(t,\cdot)) \in \F$.
Due to the continuity of $\Phi$, $\Phi'$ and $\Phi''$, we obtain that for almost every $t \in I$,
\begin{align} \label{eq:pointwise Phi(u)}
 \Phi(\widetilde{f_k}(t,\cdot)) \to \Phi(\widetilde{u}(t,\cdot)) \quad \mbox{ quasi-everywhere},
\end{align}
$ \Phi'(\widetilde{f_k}(t,\cdot)) \to \Phi'(\widetilde{u}(t,\cdot))$ and  
$\Phi''(\widetilde{f_k}(t,\cdot)) \to \Phi''(\widetilde{u}(t,\cdot))$ quasi-everywhere.

Observe that $|\Phi(u) - \Phi(f_k)| \leq (\sup_{y \in \R} \Phi'(y)) |u- f_k|$.
Since $\Phi'$ is bounded and $f_k \to u$ in $L^2(I \to \F)$, we get that at a.e. $t \in I$,
\begin{align} \label{eq:Phi(u) 1} 
\int_X |\Phi(u(t,\cdot))v - \Phi(f_k(t,\cdot))v| d\mu \to 0
\quad \mbox { as } k \to \infty.
\end{align}
Hence, for a.e. $a,b \in I$, the left hand side of \eqref{eq:time chain rule u_n 2} converges to the left hand side of \eqref{eq:time chain rule} as $k \to \infty$.

Next, we consider the right hand side of \eqref{eq:time chain rule u_n 2}. 
By \eqref{eq:pointwise Phi(u)} and because $\Phi'$, $\Phi''$ and $v$ are bounded, and $f_k \to u$ in  $L^2(I \to \F)$,
\begin{align*}
& \quad \int_I \int d\Gamma( \Phi'(f_k) v - \Phi'(u) v, \Phi'(f_k) v - \Phi'(u) v ) dt \\
& \leq 
 2  \int_I \int (\Phi'(f_k) - \Phi'(u))^2 d\Gamma(v,v) dt + 2 \int_I \int \Phi''(f_k)^2 v^2 d\Gamma ( f_k - u, f_k - u) dt \\
& \quad + 2 \int_I \int ( \Phi''(u)^2 v^2 - \Phi''(f_k)^2 v^2) d\Gamma(u,u) dt  \\
& \quad + 4 \int_I \int (\Phi''(f_k)v - \Phi''(u) v) \Phi''(f_k)v d\Gamma(u,f_k-u) dt \\
& \quad + 2 \int_I \int (\Phi''(f_k)v - \Phi''(u) v) \Phi''(f_k)v d\Gamma(u,u) dt \\
& \longrightarrow 0, \quad \mbox{ as } k \to \infty.
\end{align*}
It follows that $(\Phi'(f_k) v - \Phi'(u) v) \to 0$ in $L^2(I \to \F)$.
Since $f_k \to u$ in $\F(I \times X)$, we also have that $\frac{\partial}{\partial t} f_k \to \frac{\partial}{\partial t} u$ in $L^2(I \to \F')$.
Hence,
\begin{align*}
& \int_I \left\langle \frac{\partial}{\partial t} u, \Phi'(u) v \right\rangle_{\F',\F} dt - \int_I \left\langle \frac{\partial}{\partial t} f_k, \Phi'(f_k) v \right\rangle_{\F',\F} dt  \\
& = 
 \int_I \left\langle \frac{\partial}{\partial t} u - \frac{\partial}{\partial t} f_k, \Phi'(u) v \right\rangle_{\F',\F} dt 
+ \int_I \left\langle \frac{\partial}{\partial t} f_k, \Phi'(u)v - \Phi'(f_k) v \right\rangle_{\F',\F} dt \\
& \longrightarrow 0 \quad \mbox{ as } k \to \infty.
\end{align*}
\end{proof}

\begin{corollary}
\label{cor:time int by parts simple version}
Let $u \in \F_{\mbox{\em \tiny{loc}}}(I \times X)$ and $v \in \F_{\mbox{\em \tiny{c}}}(X)$.
Then at a.e. $a,b \in I$,
\begin{align*}
 \int u(b,\cdot)  v \, d\mu - \int u(a,\cdot)  v \, d\mu  = 
\int_a^b \left\langle \frac{\partial}{\partial t} u(t,\cdot), v \right\rangle_{\F',\F} dt.
\end{align*}
\end{corollary}
\begin{proof}
If $v$ is bounded, then the assertion follows from  Proposition \ref{prop:time chain rule} applied with $\Phi(x) = x$. 
If $v$ is unbounded, then approximate $v$ by $v_m := v \wedge m$.
\end{proof}

The next corollary collects some properties of weak time-derivatives, including the product rule. 

\begin{corollary} \label{cor:some properties of weak time-derivative}
\begin{enumerate}
 \item 
Let $u \in \F_{\mbox{\em \tiny{loc}}}(I \times X)$ and $v \in \F_{\mbox{\em \tiny{c}}}(X) \cap L^{\infty}(X)$. 
Let $\Phi \in \mathcal C^2(\R)$ with $\Phi(0)=0$, and $\Phi'$, $\Phi''$ bounded. Let $\chi \in \mathcal C^1(I)$ with $\chi=0$ in a 
neighborhood of $a \in I$. Then we have $\chi \Phi'(u)v \in L^2(I \to \F)$, and for a.e. $b \in I$, 
\begin{align*}
& \int \chi(b) \Phi(u(b,\cdot))  v \, d\mu \\
& = 
\int_a^b \left\langle \frac{\partial}{\partial t} u(t,\cdot), \chi(t)\Phi'(u(t,\cdot)) v \right\rangle_{\F',\F} dt
+ \int_a^b \chi'(t) \Phi(u(t,\cdot)) v d\mu \, dt.
\end{align*}
\item
Let $u,\phi \in \F_{\mbox{\em \tiny{loc}}}(I \times X) \cap L^{\infty}(I \times X)$ and $v \in \F_{\mbox{\em \tiny{c}}}(X) \cap L^{\infty}(X)$.
Then at a.e. $a,b \in I$,
\begin{align*}
& \int (u\phi)(b,\cdot))  v \, d\mu - \int (u\phi)(a,\cdot)  v \, d\mu \\
& = 
\int_a^b \left\langle \frac{\partial}{\partial t} u(t,\cdot), \phi(b,\cdot)) v \right\rangle_{\F',\F} dt
+ \int_a^b \left\langle \frac{\partial}{\partial t} \phi(t,\cdot), u(a,\cdot)) v \right\rangle_{\F',\F} dt.
\end{align*}
\end{enumerate}
\end{corollary}

\begin{proof}
(i) follows by repeating the proof of Proposition \ref{prop:time chain rule} with the obvious changes to account for the presence of the function $\chi$.
(ii) If $u = \phi$, then the assertion follows from  Proposition \ref{prop:time chain rule} applied with some function $\Phi \in C^2(\R)$ that satisfies $\Phi(x) = x^2$ for all $|x| \leq \sup_t \sup_{y \in \mbox{supp}(v)} u(t,y)$. 
The general case then follows by polarization.
\end{proof}

\subsubsection{Local weak solutions}

For every $t \in \R$, let $(\e_t,\F)$ be a (possibly non-symmetric) local 
bilinear form. We assume further that for every 
$f,g \in \F$ the map $t \mapsto \e_t(f,g)$ is measurable and that, 
for each $t$, $\e_t$ satisfies the structural hypotheses
introduced in Assumption 0. 

\begin{definition} \label{def:local weak solution}
Let $I$ be an open interval and $U \subset X$ open. Set $Q = I \times U$. A function $u: Q \to \R$ is a \emph{local weak solution} of the heat equation $\frac{\partial}{\partial t} u = L_t u$ in $Q$, if
\begin{enumerate}
\item
$u \in \F_{\mbox{\emph{\tiny{loc}}}}(Q)$,
\item 
\begin{align*} 
 \forall \phi \in \F_{\mbox{\emph{\tiny{c}}}}(Q), \quad \int_I \left\langle \frac{\partial}{\partial t} u , \phi \right\rangle_{\F',\F} dt + \int_I \e_t(u(t,\cdot),\phi(t,\cdot)) dt = 0.
\end{align*} 
\end{enumerate} 
\end{definition}

\begin{remark}
For $u \in \F_{\mbox{\tiny{loc}}}(Q)$, 
$\int_I \left\langle \frac{\partial}{\partial t} u, \phi \right\rangle_{\F',\F} dt$ should be read as $\int_I \left\langle \frac{\partial}{\partial t} u^{\sharp} , \phi \right\rangle_{\F',\F} dt$, where $u^{\sharp} \in \F(I \times X)$ with $u=u^{\sharp}$ a.e.~on $\mbox{supp}(\phi)$.
\end{remark}

\begin{proposition} \label{prop:weak and very weak}
The following are equivalent:
\begin{enumerate}
\item
$u$ is a local weak solution of $\frac{\partial}{\partial t} u = L_t u$ in $Q = I \times U$ (in the sense of Definition \ref{def:local weak solution}).
\item
$u$ is a local very weak solution of $\frac{\partial}{\partial t} u = L_t u$ in $Q = I \times U$ (in the sense of Definition \ref{def:local very weak solution}) and, for every $v \in \F_{\mbox{\em \tiny{c}}}(U)$,
 the function $b \mapsto \int  u(b)  v \, d\mu$ is continuous on $I$. 
\end{enumerate}
\end{proposition}

\begin{proof}
The forward implication is immediate from \eqref{eq:C(I to L^2)} and Corollary \ref{cor:time int by parts simple version}. 
We prove the converse.
Let $u$ be a very weak solution. By the continuity assumption in (ii),
\[ \int [ u(b) - u(a) ] v \, d\mu +  \int_a^b \e_t(u(t),v) dt = 0, \quad \forall v \in \F_{\mbox{\tiny{c}}}(U), \]
holds for all $a,b \in I$.
It is clear that $b \mapsto \int_a^b \e_t(u(t),v) dt$ is differentiable on $I$ with derivative $\e_b(u(b),v)$. 
Hence, $b \mapsto \int u(b) v \, d\mu$ is differentiable in $I$, and 
\[ \left[\frac{d}{db} \int u(b) v \, d\mu \right] +  \e_b(u(b),v) =  0. \]
In particular, for any smooth compactly supported function $\chi:I \to \R$, we get
\[ \int_I \left[ \frac{d}{dt} \int u(t) v  d\mu \right] \chi(t) \,  dt + \int_I  \e_t(u(t),v \chi(t))  dt =  0, 
\quad \forall v \in \F_{\mbox{\tiny{c}}}(U). \]
Integration by parts 
yields that
\[ - \int_I \left\langle \frac{\partial}{\partial t} \phi(t), u(t)  \right\rangle_{\F',\F} dt + \int_I  \e_t(u(t),\phi(t))  dt = 0, \]
for all $\phi(t) = \chi(t) v$, where $\chi:I \to \R$ is smooth with compact support in $I$, and $v \in \F_{\mbox{\tiny{c}}}(U)$.
Applying \cite[Lemma 25.1]{Wlo87en}, this extends to all functions  $\phi \in \F_{\mbox{\tiny{c}}}(I \times U)$.
By \cite[Lemma 1]{EldredgeSC14}, this proves that $u$ is a local weak solution.
\end{proof}

\def\cprime{$'$} \def\cprime{$'$}


\begin{thebibliography}{10}

\bibitem{Aro67}
{\sc D.~G. Aronson}, {\em Bounds for the fundamental solution of a parabolic
  equation}, Bull. Amer. Math. Soc., 73 (1967), pp.~890--896.

\bibitem{AS67}
{\sc D.~G. Aronson and J.~Serrin}, {\em Local behavior of solutions of
  quasilinear parabolic equations}, Arch. Rational Mech. Anal., 25 (1967),
  pp.~81--122.

\bibitem{BGK12}
{\sc M.~T. Barlow, A.~Grigor'yan, and T.~Kumagai}, {\em On the equivalence of
  parabolic {H}arnack inequalities and heat kernel estimates}, J. Math. Soc.
  Japan, 64 (2012), pp.~1091--1146.

\bibitem{BM95}
{\sc M.~Biroli and U.~Mosco}, {\em {A} {S}aint-{V}enant type principle for
  {D}irichlet forms on discontinuous media}, Annali di Matematica Pura ed
  Applicata, 169 (1995), pp.~125--181.
\newblock 10.1007/BF01759352.

\bibitem{CF12}
{\sc Z.-Q. Chen and M.~Fukushima}, {\em Symmetric {M}arkov processes, time
  change, and boundary theory}, vol.~35 of London Mathematical Society
  Monographs Series, Princeton University Press, Princeton, NJ, 2012.

\bibitem{CG01}
{\sc F.~Cipriani and G.~Grillo}, {\em Uniform bounds for solutions to
  quasilinear parabolic equations}, J. Differential Equations, 177 (2001),
  pp.~209--234.

\bibitem{DiestelUhl77}
{\sc J.~Diestel and J.~J. Uhl, Jr.}, {\em Vector measures}, American
  Mathematical Society, Providence, R.I., 1977.
\newblock With a foreword by B. J. Pettis, Mathematical Surveys, No. 15.

\bibitem{EE87}
{\sc D.~E. Edmunds and W.~D. Evans}, {\em Spectral theory and differential
  operators}, Oxford Mathematical Monographs, The Clarendon Press, Oxford
  University Press, New York, 1987.
\newblock Oxford Science Publications.

\bibitem{EldredgeSC14}
{\sc N.~Eldredge and L.~Saloff-Coste}, {\em Widder's representation theorem for
  symmetric local {D}irichlet spaces}, J. Theoret. Probab., 27 (2014),
  pp.~1178--1212.

\bibitem{FS86}
{\sc E.~B. Fabes and D.~W. Stroock}, {\em A new proof of {M}oser's parabolic
  {H}arnack inequality using the old ideas of {N}ash}, Arch. Rational Mech.
  Anal., 96 (1986), pp.~327--338.

\bibitem{FOT94}
{\sc M.~Fukushima, Y.~{\=O}shima, and M.~Takeda}, {\em Dirichlet forms and
  symmetric {M}arkov processes}, vol.~19 of de Gruyter Studies in Mathematics,
  Walter de Gruyter \& Co., Berlin, 1994.

\bibitem{GH}
{\sc A.~Grigor'yan and J.~Hu}, {\em Upper bounds of heat kernels on doubling
  spaces}, Mosc. Math. J., 14 (2014), pp.~505--563, 641--642.

\bibitem{GHL09}
{\sc A.~Grigor{\cprime}yan, J.~Hu, and K.-S. Lau}, {\em Heat kernels on metric
  spaces with doubling measure}, in Fractal geometry and stochastics {IV},
  vol.~61 of Progr. Probab., Birkh\"auser Verlag, Basel, 2009, pp.~3--44.

\bibitem{Gri91}
{\sc A.~A. Grigor{\cprime}yan}, {\em The heat equation on noncompact
  {R}iemannian manifolds}, Mat. Sb., 182 (1991), pp.~55--87.

\bibitem{HMS10}
{\sc Z.-C. Hu, Z.-M. Ma, and W.~Sun}, {\em On representations of non-symmetric
  dirichlet forms}, Potential Analysis, 32 (2010), pp.~101--131.

\bibitem{JS87}
{\sc D.~Jerison and A.~S{\'a}nchez-Calle}, {\em Subelliptic, second order
  differential operators}, in Complex analysis, {III} ({C}ollege {P}ark, {M}d.,
  1985--86), vol.~1277 of Lecture Notes in Math., Springer, Berlin, 1987,
  pp.~46--77.

\bibitem{KMS01}
{\sc K.~Kuwae, Y.~Machigashira, and T.~Shioya}, {\em Sobolev spaces,
  {L}aplacian, and heat kernel on {A}lexandrov spaces}, Math. Z., 238 (2001),
  pp.~269--316.

\bibitem{LierlHKEf}
{\sc J.~Lierl}, {\em The {D}irichlet heat kernel in inner uniform domains in
  fractal-type spaces}.
\newblock submitted 2016, preprint available.

\bibitem{LierlPHIf}
\leavevmode\vrule height 2pt depth -1.6pt width 23pt, {\em {P}arabolic
  {H}arnack inequality on fractal-type metric measure {D}irichlet spaces}.
\newblock arXiv:1509.04804.

\bibitem{LierlBHPf}
\leavevmode\vrule height 2pt depth -1.6pt width 23pt, {\em Scale-invariant
  boundary {H}arnack principle on inner uniform domains in fractal-type
  spaces}, Potential Anal., 43 (2015), pp.~717--747.

\bibitem{LierlSC3}
{\sc J.~Lierl and L.~Saloff-Coste}, {\em The {D}irichlet heat kernel in inner
  uniform domains: local results, compact domains and non-symmetric forms}, J.
  Funct. Anal., 266 (2014), pp.~4189--4235.

\bibitem{LierlSC1}
\leavevmode\vrule height 2pt depth -1.6pt width 23pt, {\em Scale-invariant
  boundary {H}arnack principle in inner uniform domains}, Osaka J. Math., 51
  (2014), pp.~619--656.

\bibitem{LM72}
{\sc J.-L. Lions and E.~Magenes}, {\em Non-homogeneous boundary value problems
  and applications. {V}ol. {I}}, Springer-Verlag, New York, 1972.
\newblock Translated from the French by P. Kenneth, Die Grundlehren der
  mathematischen Wissenschaften, Band 181.

\bibitem{MR92}
{\sc Z.~M. Ma and M.~R{\"o}ckner}, {\em Introduction to the theory of
  (nonsymmetric) {D}irichlet forms}, Universitext, Springer-Verlag, Berlin,
  1992.

\bibitem{MaRock}
\leavevmode\vrule height 2pt depth -1.6pt width 23pt, {\em Markov processes
  associated with positivity preserving coercive forms}, Canad. J. Math., 47
  (1995), pp.~817--840.

\bibitem{Moser64}
{\sc J.~Moser}, {\em A {H}arnack inequality for parabolic differential
  equations}, Comm. Pure Appl. Math., 17 (1964), pp.~101--134.

\bibitem{Moser67}
\leavevmode\vrule height 2pt depth -1.6pt width 23pt, {\em Correction to: ``{A}
  {H}arnack inequality for parabolic differential equations''}, Comm. Pure
  Appl. Math., 20 (1967), pp.~231--236.

\bibitem{Moser71}
{\sc J.~Moser}, {\em On a pointwise estimate for parabolic differential
  equations}, Comm. Pure Appl. Math., 24 (1971), pp.~727--740.

\bibitem{Nash58}
{\sc J.~Nash}, {\em Continuity of solutions of parabolic and elliptic
  equations}, Amer. J. Math., 80 (1958), pp.~931--954.

\bibitem{NS91}
{\sc J.~R. Norris and D.~W. Stroock}, {\em Estimates on the fundamental
  solution to heat flows with uniformly elliptic coefficients}, Proc. London
  Math. Soc. (3), 62 (1991), pp.~373--402.

\bibitem{Osh13}
{\sc Y.~Oshima}, {\em Semi-{D}irichlet forms and {M}arkov processes}, vol.~48
  of de Gruyter Studies in Mathematics, Walter de Gruyter \& Co., Berlin, 2013.

\bibitem{PE84}
{\sc F.~O. Porper and S.~D. {\`E}{\u\i}del{\cprime}man}, {\em Two-sided
  estimates of the fundamental solutions of second-order parabolic equations
  and some applications of them}, Uspekhi Mat. Nauk, 39 (1984), pp.~107--156.

\bibitem{RR93}
{\sc M.~Renardy and R.~C. Rogers}, {\em An introduction to partial differential
  equations}, vol.~13 of Texts in Applied Mathematics, Springer-Verlag, New
  York, 1993.

\bibitem{SC92}
{\sc L.~Saloff-Coste}, {\em A note on {P}oincar\'e, {S}obolev, and {H}arnack
  inequalities}, Internat. Math. Res. Notices,  (1992), pp.~27--38.

\bibitem{SC95}
\leavevmode\vrule height 2pt depth -1.6pt width 23pt, {\em Parabolic {H}arnack
  inequality for divergence-form second-order differential operators},
  Potential Anal., 4 (1995), pp.~429--467.
\newblock Potential theory and degenerate partial differential operators
  (Parma).

\bibitem{SC02}
{\sc L.~Saloff-Coste}, {\em Aspects of {S}obolev-type inequalities}, vol.~289
  of London Mathematical Society Lecture Note Series, Cambridge University
  Press, Cambridge, 2002.

\bibitem{Str88}
{\sc D.~W. Stroock}, {\em Diffusion semigroups corresponding to uniformly
  elliptic divergence form operators}, in S\'eminaire de {P}robabilit\'es,
  {XXII}, vol.~1321 of Lecture Notes in Math., Springer, Berlin, 1988,
  pp.~316--347.

\bibitem{SturmI}
{\sc K.-T. Sturm}, {\em Analysis on local {D}irichlet spaces. {I}.
  {R}ecurrence, conservativeness and {$L^p$}-{L}iouville properties}, J. Reine
  Angew. Math., 456 (1994), pp.~173--196.

\bibitem{SturmII}
\leavevmode\vrule height 2pt depth -1.6pt width 23pt, {\em Analysis on local
  {D}irichlet spaces. {II}. {U}pper {G}aussian estimates for the fundamental
  solutions of parabolic equations}, Osaka J. Math., 32 (1995), pp.~275--312.

\bibitem{Stu95geometry}
\leavevmode\vrule height 2pt depth -1.6pt width 23pt, {\em On the geometry
  defined by {D}irichlet forms}, in Seminar on {S}tochastic {A}nalysis,
  {R}andom {F}ields and {A}pplications ({A}scona, 1993), vol.~36 of Progr.
  Probab., Birkh\"auser, Basel, 1995, pp.~231--242.

\bibitem{SturmIII}
\leavevmode\vrule height 2pt depth -1.6pt width 23pt, {\em Analysis on local
  {D}irichlet spaces. {III}. {T}he parabolic {H}arnack inequality}, J. Math.
  Pures Appl. (9), 75 (1996), pp.~273--297.

\bibitem{Sturm06}
\leavevmode\vrule height 2pt depth -1.6pt width 23pt, {\em On the geometry of
  metric measure spaces. {I}}, Acta Math., 196 (2006), pp.~65--131.

\bibitem{vCas11}
{\sc J.~A. van Casteren}, {\em Markov processes, {F}eller semigroups and
  evolution equations}, vol.~12 of Series on Concrete and Applicable
  Mathematics, World Scientific Publishing Co. Pte. Ltd., Hackensack, NJ, 2011.

\bibitem{Wlo87en}
{\sc J.~Wloka}, {\em Partial differential equations}, Cambridge University
  Press, Cambridge, 1987.
\newblock Translated from the German by C. B. Thomas and M. J. Thomas.

\end{thebibliography}

\end{document}